\documentclass[review]{elsarticle}

\usepackage{a4wide,amssymb,graphics,graphicx,textcomp}
\usepackage{enumerate,textcomp,multirow,amsmath}
\usepackage{algorithm}
\usepackage{algorithmic}
\usepackage{url}
\usepackage{xcolor} 
\usepackage{amsmath, amsthm}
\usepackage{hyperref}

\newtheorem{lemma}{Lemma}
\newtheorem{theorem}{Theorem}

\newtheorem{remark}{Remark}
\newtheorem{problem}{Problem}

\newcommand {\mat}  [1] {\left[\begin{array}{#1}}
	\newcommand {\rix}      {\end{array}\right]}

\newcommand{\eproof}{\space
	{\ \vbox{\hrule\hbox{\vrule height1.3ex\hskip0.8ex\vrule}\hrule}}\par}

\def\rank{\mathop{\mathrm{rank}}}
\def\real{\mathop{\mathrm{Re}}}

\def\trace{\mathop{\mathrm{trace}}}
\newcommand{\pls}{{\dagger}}
\newcommand{\wt}{\widetilde}
\newcommand{\wh}{\widehat}

\newcommand{\C}{{\mathbb C}}
\newcommand{\R}{{\mathbb R}}

\newcommand{\D}{{\Delta}}
\newcommand{\lS}{{\mathcal S}}

\newcommand{\lm}{{\lambda}}

\definecolor{brightpink}{rgb}{1.0, 0.0, 0.5}

\journal{Linear Algebra and its Applications}

\begin{document}
	
	\begin{frontmatter}
		
\title{Doubly structured mapping problems of the form $\Delta x=y$ and $\Delta^*z=w$}

\author{Mohit Kumar Baghel}
\author{Punit Sharma}
\fntext[]{Department of Mathematics, Indian Institute of Technology Delhi, Hauz Khas, 110016, India; \texttt{\{maz188260, punit.sharma\}@maths.iitd.ac.in.}
P.S. acknowledges the support of the DST-Inspire Faculty Award (MI01807-G) by Government of India.
}
		
\begin{abstract}
For a given class of structured matrices $\mathbb S$, we find necessary and sufficient conditions on vectors $x,w\in \C^{n+m}$ and $y,z \in \C^{n}$ for which there exists $\Delta=[\Delta_1~\Delta_2]$ with $\Delta_1 \in \mathbb S$ and $\Delta_2 \in \C^{n,m}$ such that $\Delta x=y$ and $\Delta^*z=w$. We also characterize the set of all such mappings $\Delta$ and provide sufficient conditions on vectors $x,y,z$, and $w$ to investigate a $\Delta$ with minimal Frobenius norm. The structured classes $\mathbb S$ we consider include (skew)-Hermitian, (skew)-symmetric, pseudo(skew)-symmetric, $J$-(skew)-symmetric, pseudo(skew)-Hermitian, positive (semi)definite, and dissipative matrices. These mappings are then used in computing the structured eigenvalue/eigenpair backward errors of matrix pencils arising in optimal control. 
\end{abstract}
		
		
\end{frontmatter}
	
\section{Introduction}

\begin{problem}[Doubly structured mapping problem]\label{def:dsmprob}For a given class of structured matrices $\mathbb S \subseteq \C^{n,n}$, and vectors $x,w \in \C^{n+m}$ and $y,z \in \C^n$, we consider the following mapping problem:
\begin{itemize}
\item {\it Existence}: Find necessary and sufficient conditions on vectors $x,y,z$, and $w$ for the existence of $\Delta=[\Delta_1~\Delta_2]$, where $\Delta_1 \in \mathbb S$ and $\Delta_2 \in \C^{n,m}$ such that $\Delta x=y$ and $\Delta^*z=w$.

We call such a mapping $\Delta$ as \emph{doubly structured mapping} (DSM) as it has two structures defined on it; (i) conjugate transpose of $\Delta$ satisfies $\Delta^*z=w$, and (ii) $\Delta$ has the form $\Delta=[\Delta_1~\Delta_2]$ with $\Delta_1 \in \mathbb S$. 
\item {\it Characterization}: Determine the set 
\begin{equation}\label{eq:defSSd}
\mathcal S_d^{\mathbb S}:=\{\Delta :~\Delta=[\Delta_1~\Delta_2],\,\Delta_1 \in \mathbb S, \Delta_2 \in \C^{n,m},\ \Delta x=y,\, \Delta^* z=w \}
\end{equation}
of all such doubly structured mappings.
\item {\it Minimal Frobenius norm}: Characterize all solutions to the doubly structured mapping problem that have minimal Frobenius norm.
\end{itemize}
\end{problem}
The structures we consider on $\Delta_1$ in a DSM problem include symmetric, skew-symmetric, pseudosymmetric, pseudoskew-symmetric, Hermitian, skew-Hermitian, pseudo-Hermitian, pseudoskew-Hermitian, positive (negative) semidefinite, and dissipative matrices.

The minimal norm solutions to such doubly structured mappings can be very handy in the perturbation analysis of matrix pencils arising in control systems~\cite{MehMS18}. In particular, for the computation of structured eigenvalue/eigenpair backward errors of matrix pencils $L(z)$ of the form
\begin{equation}\label{eq:defpenciL}
L(z)=M+zN := \mat{ccc} 0 & J-R & B\\(J-R)^* & 0 & 0 \\ B^* & 0 & S \rix +z
\mat{ccc} 0 & E &0 \\ -E^* & 0 & 0 \\ 0& 0& 0 \rix, 
\end{equation}
where $J,R,E,Q \in \C^{n,n}$, $B \in \C^{n,m}$ and $S \in \C^{n,m}$ satisfy 
$J^*=-J$, $R^*=R $ is positive semidefinite, $E^*=E$, and $S^*=S$ is positive definite. The pencil $L(z)$ arises in  $H_{\infty}$ control problems  and in the passivity analysis of dynamical systems~\cite{KunM08,Meh91}. One example of such a system is a port-Hamiltonian descriptor system~\cite{BeaMXZ18,Sch13}.
Our work is motivated by~\cite{MehMS18}, where the eigenpair backward errors have been computed while preserving the block and symmetry structures of pencils of the form $L(z)$, where only the Hermitian structure was considered on $R$. The definiteness structure on $R$ describes the energy dissipation in the system and guarantees the stability of the underlying port-Hamiltonian system~\cite{MehMW18,GilS18}. This makes it essential to preserve the definiteness of $R$ to preserve the system's port-Hamiltonian structure. 

The standard mapping problem for matrices is to find $\Delta \in \C^{n,n}$ for given vectors $x \in \C^{n}$ and $y \in \C^n$, such that $\Delta x=y$. Such mapping problems have been well studied in~\cite{MacMT08}, where authors provide complete, unified, and explicit solutions for structured mappings from Lie and Jordan algebras associated with orthosymmetric scalar products. The minimal norm solutions to the structured mappings provide an important tool in solving nearness problems for control systems, e.g.~\cite{BorKMS14,BorKMS15,MehMS16,MehMS17,BhaGS21}. The DSMs extend the mapping problem of finding $\Delta \in \C^{n,n}$ for given vectors $x,y,z,w \in \C^n$ such that $\Delta x=y$ and $\Delta^*z=w$~\cite{MehMS16,MehMS17}. 


This paper is organized as follows:~In Section~\ref{sec:prelim}, we review some preliminary results on mapping problems. In Section~\ref{sec3:J&Lalgebra}, we present necessary and sufficient conditions for the existence of DSMs with structures belonging to a Jordan or a Lie algebra. In particular, we consider doubly structured Hermitian, skew-Hermitian, symmetric, and skew-symmetric mapping problems. We provide solutions to the doubly structured semidefinite mapping problem in Section~\ref{sec:dssmp}. In Section~\ref{sec:DSDMs}, we introduce two types of doubly structured dissipative mappings. The  minimal norm solutions to the DSM problems are then used in estimating various structured eigenpair  backward errors for the pencil $L(z)$ arising in control systems, see Section~\ref{sec:app_strbacerr}.
		
\paragraph{Notation} In the following, we denote the identity matrix of size $n \times n$ by $I_n$, the spectral norm of a matrix or a vector by $\|\cdot\|$ and the Frobenius norm by ${\|\cdot\|}_F$. The Moore-Penrose pseudoinverse of a matrix or a vector $X$ is denoted by $X^\pls$
	and  $\mathcal P_X=I_n-XX^\pls$ denotes the orthogonal projection onto the null space of $n \times n$ matrix $X^*$. For a square matrix $A$, its Hermitian and skew-Hermitian parts are respectively denoted by $A_H=\frac{A+A^*}{2}$ and $A_S=\frac{A-A^*}{2}$.
	For $A=A^* \in \mathbb F^{n,n}$, where $\mathbb F \in \{\R,\C\}$, we denote
	$A \succ 0$ ($A\prec 0$) and $A\succeq 0$ ($A\preceq$) if $A $ is Hermitian positive definite (negative definite) and Hermitian positive semidefinite (negative semidefinite).  $\Lambda(A)$ denotes the set of all eigenvalues of the matrix $A$. ${\rm  Herm}(n)$, ${\rm  SHerm}(n)$, ${\rm  Sym}(n)$, and ${\rm SSym}(n)$
stand respectively for the set of $n \times n$ Hermitian, skew-Hermitian, symmetric, and skew-symmetric matrices.

\section{Preliminaries } \label{sec:prelim}
In this section, we state some elementary lemmas  and recall some mapping results that will be necessary to solve the DSM problem. 
	
\begin{lemma}{\rm \cite{Alb69}} \label{lem:psd}
Let the integer $s$ be such that $0<s<n$, and $R=R^* \in \C^{n,n}$ be partitioned as $R=\mat{cc}B & C^*\\C & D\rix$ with $B\in \C^{s,s}$, $C\in \C^{n-s,s}$ and $D \in \C^{n-s,n-s}$. Then $R \succeq 0$ if and only if
\begin{enumerate}
\item $B \succeq 0$,
\item $\ker(B) \subseteq \ker(C)$, and
\item $D-CB^{\pls}C^* \succeq 0$, where $B^{\pls}$ denotes the Moore-Penrose pseudoinverse of $B$.
\end{enumerate}
\end{lemma}
 
\begin{lemma}\label{lem8} {\rm \cite{BhaGS21}}
Let $X,Y \in \C^{n,m}$. Suppose that $\rank(X)=r$ and consider the reduced singular value decomposition $X=U_1 \Sigma_1 V_1^*$ with $U_1 \in \C^{n,r}, \Sigma_1 \in \C^{r,r}$ and $V_1 \in \C ^{m,r}$. If $X^*Y + Y^*X \succeq 0$, then $U_1^* \left( YX^\pls + (YX^\pls)^*\right) U_1 \succeq 0$.
 \end{lemma}

\begin{lemma}\label{lem: mohit}
	Let $X, Y, Z, W \in \C^{n,m}$ with $X^*W=Y^*Z$. Suppose that $\rank(X)=\rank(Z)=r$ and consider the reduced singular value decompositions $X=U_1\Sigma_1 V_1^*$ and $Z=\tilde U_1 \tilde \Sigma_1 \tilde V_1^*$, where $U_1,\tilde U_1 \in \C^{n,r}$, $V_1,\tilde V_1 \in \C^{m,r}$, and $\Sigma_1,\tilde \Sigma_1$ are the diagonal matrices containing the nonzero singular values of $X$ and $Z$, respectively. If $U_1=\tilde U_1$, then 
\begin{eqnarray}
		U_1^* \left( YX^\pls \pm (YX^\pls)^*\right) U_1 = U_1^{*} \left( YX^\pls \pm WZ^\pls\right) U_1.
\end{eqnarray}
\end{lemma}
\begin{proof}
The proof follows using the fact that $U_1=\tilde U_1$ and $X^*W=Y^*Z$. In fact, we have 
\begin{eqnarray}
\begin{split}
			U_1^{*} \left( YX^\pls \pm WZ^\pls\right) U_1 &= U_1^{*}  Y X^\pls U_1 \pm \Sigma_1^{-1} V_1^* V_1 \Sigma_1 U_1^* WZ^\pls U_1 \\ \nonumber
			&= U_1^*  YX^\pls U_1 \pm \Sigma_1^{-1} V_1^* X^* W Z^\pls U_1  \quad \, (\because X^*=V_1 \Sigma_1 U_1^*) \label{t1}\\ 
			&= U_1^*  YX^\pls U_1 \pm \Sigma_1^{-1} V_1^* Y^*Z Z^\pls U_1  \quad ~~(\because X^*W=Y^*Z) \label{t2} \\
			&= U_1^*  YX^\pls U_1 \pm \Sigma_1^{-1} V_1^* Y^*U_1 U_1^* U_1  \quad \, (\because U_1=\tilde U_1,\,U_1 U_1^* =ZZ^\pls) \nonumber \\
			&= U_1^*  YX^\pls U_1 \pm \Sigma_1^{-1} V_1^* Y^* U_1   \nonumber \\
			&=U_1^*  YX^\pls U_1 \pm U_1^* U_1 \Sigma_1^{-1} V_1^* Y^* U_1 \quad ( \because U_1^* U_1=I_r)   \nonumber \\ 
			&=U_1^*  YX^\pls U_1 \pm U_1^* X^{\pls ^ *} Y^* U_1 \quad \quad \quad ~\, (\because X^{\pls ^ *}= U_1 \Sigma_1^{-1} V_1^* ) \nonumber \\
			&=U_1^{*} \left( YX^\pls \pm (YX^\pls)^*\right) U_1. \nonumber \\
\end{split}
\end{eqnarray}
\end{proof}	

\begin{lemma} \label{lemma: ineq_psd} {\rm \cite[P.57]{ciarlet1989introduction}}
	Let $A,B \in \C ^{n,n}$. If $B \succeq 0, A-B \succeq 0$, then ${\| A \|} \geq {\|B\|}$ where $\|\cdot\|$ is any unitarily invariant matrix norm.
\end{lemma}

\subsection{Mapping results}	

Here, we recall some mapping results from the literature that will be useful  in solving DSM problem for various structures. We start with a result from~\cite{sun1993backward} that solves the standard mapping problem with no structure on it.  

\begin{theorem}{\rm \cite{sun1993backward}} \label{map result: unstruct}
	Let $x,y \in \C^n \setminus \{0\}$ and define $\lS := \{ \D \in \C^{n,n} ~:~ \D x=y\}$. Then $\lS \neq \emptyset$ and 
\[ 
\lS = \{ y x^\pls + Z \mathcal P_x ~:~ Z \in \C^{n,n} \}.
\]
Moreover, 
$ \inf_{\D \in \lS} { \| \D \| }_F= \frac{\|y\|}{\|x\|} ,
$
where the infimum is attained by $\hat{\D} = yx^\pls$.
\end{theorem}

We next recall a few mapping results from~\cite{adhikari2008backward} that give a characterization and minimal Frobenius norm solution for the Hermitian, skew-Hermitian, complex symmetric, and complex skew-symmetric mapping problems. 

\begin{theorem}{\rm \cite{adhikari2008backward}} \label{theorem: herm map}
	Let $x,y \in \C^n\setminus \{0\}$ and let $\lS := \{ \D \in \C^{n,n} ~: ~  \D^*=\D,\,\D x=y \}$. Then $\mathcal S \neq \emptyset$ if and only if $x^*y \in \R$. If the later condition holds, then 
\begin{equation}\label{thm:Hermmap}
\lS= \left\{ y x^\pls + (y x ^\pls) ^* - (x^\pls y)x x^\pls + \mathcal P_x H \mathcal P_x ~: ~H \in \C^{n,n},\, H^*=H\right \}.
\end{equation}
Moreover, $ \inf_{\D \in \lS } { \| \D \| }_F^2=   2 {\|yx^\pls\|}_F^2 - 
\trace( (x^\pls y)^* x x^\pls)$,
where the infimum is uniquely attained by $\hat{\D} = yx^\pls + (yx^\pls)^* - (x^\pls y) xx^\pls$ which is obtained by setting $H=0$ in~\eqref{thm:Hermmap}.
\end{theorem}


\begin{theorem}{\rm \cite{adhikari2008backward}} \label{theorem: skew herm map}
Let $x,y \in \C^n \setminus \{0\}$ and let $\lS := \{ \D \in \C^{n,n} ~: ~ \D x=y, \D^*=-\D \}$. Then $\mathcal S \neq \emptyset$ if and only if $x^*y \in i \R$. If the later condition holds, then 
\[
\lS= \{ y x^\pls - (y x ^\pls) ^* + (x^\pls y)x x^\pls + \mathcal P_x S \mathcal P_x ~: ~  S \in \C^{n,n},\,S^*=-S\}.
\]
Moreover, $\inf_{\D \in \lS} { \| \D \| }_F^2 =  2 {\|yx^\pls\|}_F^2 - \trace( (x^\pls y)^* x x^\pls)$,
where the infimum is uniquely attained by $\hat{\D} = yx^\pls - (yx^\pls)^* + (x^\pls y) xx^\pls$. 
\end{theorem}

		
		\begin{theorem}{\rm \cite{adhikari2008backward}} \label{theorem: Complex-sym map}
			Let $x,y \in \C^n\setminus \{0\}$ and let $\lS := \{ \D \in \C^{n,n} ~: ~  \D^T=\D,\,\D x=y \}$. Then $\mathcal S \neq \emptyset$ and
			\begin{equation}\label{thm:Complex-sym-map}
				\lS= \left\{ y x^\pls + (y x ^\pls) ^T -  (x x^\pls)^T y x^\pls + (\mathcal P_x)^T H \mathcal P_x ~: ~H \in \C^{n,n},\, H^T=H\right \}.
			\end{equation}
			Moreover, $ \inf_{\D \in \lS } { \| \D \| }_F^2=   2 {\|yx^\pls\|}_F^2 - 
			\trace( y x^\pls (yx^\pls)^* (x x^\pls)^T)$,
			where the infimum is uniquely attained by $\hat{\D} = y x^\pls + (y x ^\pls) ^T -  (x x^\pls)^T y x^\pls$ which is obtained by setting $H=0$ in~\eqref{thm:Complex-sym-map}.
		\end{theorem}

		\begin{theorem}{\rm \cite{adhikari2008backward}} \label{theorem: Complex skew-sym map}
			Let $x,y \in \C^n\setminus \{0\}$ and let $\lS := \{ \D \in \C^{n,n} ~: ~  \D^T=-\D,\,\D x=y \}$. Then $\mathcal S \neq \emptyset$ if and only if $x^Ty =0$. If the later condition holds, then 
			\begin{equation}\label{thm:Complex skew-sym-map}
				\lS= \left\{ y x^\pls - (y x ^\pls) ^T +  (x x^\pls)^T y x^\pls + (\mathcal P_x)^T H \mathcal P_x ~: ~H \in \C^{n,n},\, H^T=-H\right \}.
			\end{equation}
			Moreover, $ \inf_{\D \in \lS } { \| \D \| }_F^2=   2 {\|yx^\pls\|}_F^2 - 
			\trace( y x^\pls (yx^\pls)^* (x x^\pls)^T)$,
			where the infimum is uniquely attained by $\hat{\D} = y x^\pls - (y x ^\pls) ^T +  (x x^\pls)^T y x^\pls$ which is obtained by setting $H=0$ in~\eqref{thm:Complex skew-sym-map}.
\end{theorem}

The next result from~\cite{MehMS17} solves the mapping problem of finding $\D \in \C^{n,n}$ for given vectors $x,y,z,,w \in \C^{n}$ such that $\D x=y$ and $\D^*z=w$.

\begin{theorem}{\rm \cite{MehMS17}} \label{theorem : 2.1 mehl2017}
	Let $x,w \in \C^{m} \setminus \{0\},y,z \in \C^n \setminus \{0\}$ and let \[\lS := \{ \D \in \C^{n,m} ~ : ~ \D x= y, \D ^*z=w\}.\] Then $\mathcal S \neq \emptyset$ if and only if  $x^*w=y^*z$. If the later condition holds, then 
	\[ 
	\lS = \{ yx^\pls + (wz^\pls)^* - (wz^\pls)^* x x^\pls + \mathcal P_z R \mathcal P_x ~ :~ R \in \C^{n,m} \}.
	\]
Moreover, $\inf_{\D \in \lS} {\|\D\|}_F^2 ={\|yx^\pls\|}_F^2 + {\|wz^\pls\|}_F^2 - \trace(wz^\pls (wz^\pls)^* x x^\pls)$,
where the infimum is uniquely attained by $\hat{\D}=  yx^\pls + (wz^\pls)^* - (wz^\pls)^* x x^\pls$.
\end{theorem}

We close this subsection by stating two results that provide solutions to  the positive semidefinite mapping problem and dissipative mapping problem, respectively. 

\begin{theorem}{\rm \cite[Theorem 2.3]{MehMS16}} \label{lem:psdmap}
Let $x,y \in \C^n \setminus \{0\}$ and let
$\lS := \{ \D \in \C^{n,n} ~ : ~ \D x= y,\, \D =\D^* \succeq 0\}$.
Then $\lS\neq \emptyset$ if and only if $x^*y > 0$. If the later condition holds, then 
\[ 
\lS = \left \{ \frac{y y^*}{x^*y} + \mathcal P_x K \mathcal P_x ~ :~ K \in \C^{n,m},\, K \succeq 0 \right \}.
\]
Moreover, $ \inf_{\D \in \lS} {\|\D\|}_F^2 =  \frac{\|y\|_F^2}{x^*y}$ 
and infimum is uniquely attained by the rank one matrix $\hat{\D}=\frac{y y^*}{x^*y}$.
\end{theorem}

\begin{theorem}{\rm \cite{BhaGS21}} \label{map result : MAIN BhaGS21a} 
	Let $x,y \in \C^{n}\setminus\{0\}$ and let  $\mathcal S:=\{\Delta \in \C^{n,n}~:~\Delta+\Delta^* \succeq 0,~\Delta x=y\}$.
 Then $\mathcal S\neq \emptyset$ if and only if  $\real{(x^*y)} \geq 0$.
 Moreover, if $\real{(x^*y)} >0$, then 
\begin{eqnarray}\label{eq:S_char 1}
		\mathcal S=\Big\{yx^\pls +{(yx^\pls)}^* \mathcal P_x +xx^\pls Z\mathcal P_x + \mathcal P_x K \mathcal P_x + \mathcal P_x G \mathcal P_x:~Z,K,G \in \C^{n,n}~\text{satisfy}~\eqref{eq:cond1 a} \Big\},
\end{eqnarray} 
		where
		\begin{eqnarray}\label{eq:cond1 a}
G^*=-G,\quad K\succeq 0, \quad \text{and}\quad K-\frac{1}{4\real{(x^*y)}}\left(2y+Z^*x\right)\left(2y+Z^*x\right)^*.
		\end{eqnarray}
Further,
\[
\inf_{\Delta \in \mathcal S} {\|\Delta\|}_F^2 \; = \; 2\frac{{\|y\|}^2}{{\|x\|}^2}-\frac{{|x^*y|}^2}{{\|x\|}^4},
\]
where the infimum is uniquely attained by  $\mathcal H:=yx^\pls -{(yx^\pls)}^*  \mathcal P_x$,
		which is obtained by setting $K=0$, $G=0$, and $Z=-2(yx^\pls)^*$ in~\eqref{eq:S_char 1}.
\end{theorem}

\section{DSM problems with structures belonging to a Jordan or a Lie algebra} \label{sec3:J&Lalgebra}

In this section, we define the structures on $\D_1$ in a DSM problem that are associated with some orthosymmetric scalar product. Let $M \in \C^{n,n}$ be unitary such that $M$ is either symmetric or skew-symmetric or Hermitian or skew-Hermitian. Define the scalar product ${\langle{\cdot,\cdot\rangle}}_M:\C^n \times \C^n \longmapsto \C$ by 
\[
{\langle{x,y\rangle}}_M := \begin{cases}
y^TMx, & \text{ bilinear form}\\
y^*Mx, & ~\text{sesquilinear}.
\end{cases}
\]
Then the adjoint of a matrix $A \in \C^{n,n}$ with respect to the scalar product  ${\langle{\cdot,\cdot\rangle}}_M$ is denoted by $A^\star$ and defined by 
\[
A^\star=\begin{cases}
M^{-1}A^TM, & \text{ for bilinear form}\\
M^{-1}A^*M, & ~\text{for sesquilinear}.
\end{cases}
\]
We also have a Lie algebra $\mathbb L$ and a Jordan algebra $\mathbb J$, associated with ${\langle{\cdot,\cdot\rangle}}_M$ defined by 
\begin{equation*}
\mathbb L : = \left\{A \in \C^{n,n}: ~A^\star=-A\right\}\quad \text{and}\quad \mathbb J : = \left\{A \in \C^{n,n}: ~A^\star=A\right\},
\end{equation*}
respectively, see~\cite{MacMT08} for more details. We refer to~\cite[Table 2.1]{MacMT08} for some known structured matrices in some $\mathbb L$ or $\mathbb J$ associated with a scalar product. This include symmetric, skew-symmetric, pseudosymmetric, pseudoskew-symmetric, Hermitian, skew-Hermitian, pseudo-Hermitian, pseudoskew-Hermitian, etc.

If $\mathbb S \in \{\mathbb L ,\mathbb J \}$ and if we define $M\mathbb S := \{MA : ~ A \in \mathbb S \}$, then it is easy to check that 
\begin{equation}\label{eq:matrixequil}
A \in \mathbb S \Longleftrightarrow {MA} \in \{{\rm  Herm}(n),{\rm  SHerm}(n),{\rm  Sym}(n),{\rm SSym}(n)\}.
\end{equation}
In view of~\eqref{eq:matrixequil}, the following result shows that the doubly structured mapping problems with $\D_1 \in \{{\rm  Herm}(n),{\rm  SHerm}(n),{\rm  Sym}(n),{\rm SSym}(n)\}$ are prototypes of more general structured matrices belonging to Jordan and Lie algebras~\cite{MacMT08,AdhA16}.

\begin{theorem}\label{thm:LJreduction}
Consider a Lie algebra $\mathbb L$ and a Jordan algebra $\mathbb J$, associated with a scalar product ${\langle{\cdot,\cdot\rangle}}_M$ and let $\mathbb S \in \{\mathbb L ,\mathbb J \}$. Then for given vectors $x,w \in \C^{n+m}$ and $y,z \in \C^n$, $\mathcal S_d^{\mathbb S} \neq \emptyset$ if and only if $\mathcal S_d^{\mathbb S'} \neq \emptyset$ for some 
$\mathbb S' \in \{ {\rm  Herm}(n),{\rm  SHerm}(n),{\rm  Sym}(n),{\rm SSym}(n)\}$. 
Further,  $\hat \Delta$ is of minimal Frobenius norm in $ \mathcal S_d^{\mathbb S}$ if and only if $M \hat \Delta$ is of minimal Frobenius norm in $ \mathcal S_d^{\mathbb S'}$.
\end{theorem}

Given Theorem~\ref{thm:LJreduction}, the DSM problem with structures belonging to $\mathbb L$ or $\mathbb J$ can be solved by reducing it to the DSM problem for $\mathbb S' \in \{ {\rm  Herm}(n),{\rm  SHerm}(n),{\rm  Sym}(n),{\rm SSym}(n)\}$. Thus, in the following, we only consider the DSMs with structures $\mathbb S' \in \{ {\rm  Herm}(n),{\rm  SHerm}(n),{\rm  Sym}(n),{\rm SSym}(n)\}$.

\subsection{Doubly structured Hermitian mappings}

This section considers the doubly structured Hermitian mapping (DSHM) problem, i.e., when $\mathbb S={\rm Herm}(n)$ in Problem~\ref{def:dsmprob}. We have the following result that completely solves the existence and characterization problem for DSHMs and provides sufficient conditions for the minimal norm solution to the DSHM problem. 

\begin{theorem} \label{theorem HD}
	Given $x=[x_1^T~x_2^T]^T$ with $x_1 \in \C^{n}$ and $x_2 \in \C^m$, $y,z \in \C^n$, and $w=[w_1^T~w_2^T]^T$ with $w_1 \in \C^n$ and $w_2 \in \C^m$. Define 
\begin{equation} \label{def: double struct mapping set_H }
\mathcal S_d^{{\rm Herm}}  := \left \{ \D :~\D = [\D_1 ~ \D_2],\, \D_1\in \C^{n,n}, \, \D_2 \in \C^{n,m},\,  \D_1^* = \D_1,\, \D x=y, \,\D^*z=w \right \}.
	\end{equation}
	Then $\mathcal S_d^{{\rm Herm}} \neq \emptyset $ if and only if $x^*w=y^*z$ and $z^*w_1 \in \R$. If the later conditions hold true, then
		\begin{equation}\label{eq : char double struct_H}
			\mathcal S_d^{{\rm Herm}}  = \left \{ H + \wt H(K,R)  : ~ K\in \C ^{n,n},\, R \in \C^{n,m},\, K^*=K \right \},
		\end{equation}
		where $H=[H_1 ~ H_2]$ and $\wt H(K,R) = [\wt H_1(K) ~ \wt H_2(K,R) ]$ with $H_1,H_2,\wt H_1(K), \wt H_2(K,R)$ given by 
		\begin{eqnarray}
			H_1 &=& w_1 z^\pls + (w_1 z ^\pls) ^* - (z^\pls w_1)z z^\pls \label{val: H_1 double struct_H}, \\
			H_2 &=& yx_2^\pls- \left( w_1 z^\pls + (w_1 z ^\pls) ^* - (z^\pls w_1)z z^\pls\right)x_1  x_2^\pls + (w_2 z^\pls)^*\mathcal P_{x_2}   \label{val: H_2 double struct_H}, \\
			\wt H_1(K)&=& \mathcal P_z K \mathcal P_z \label{val: wt H_1 double struct_H}, \\
			\wt H_2(K,R) &=& \mathcal P_z R \mathcal P_{x_2} - \mathcal P_z K \mathcal P_z x_1 x_2^\pls \label{val: wt H_2 double struct_H},
		\end{eqnarray}
and
\begin{equation} \label{eq: minimal norm double struct_H}
 \inf_{\D \in \mathcal S_d^{{\rm Herm}}} {\|\D\|}_F^2 ~\geq~ {\| H_1\|}_F^2 + \inf_{K\in \C^{n,n}}  {\big\|H_2 + \wt H_2(K,0)\big \|}_F^2.	
\end{equation}  		
Moreover, if $x_1 = \alpha z$ for some nonzero $\alpha \in \C$, then equality holds in~\eqref{eq: minimal norm double struct_H} and we have
\begin{equation*} \label{eq: minimal norm double struct_H}
	 	\inf_{\D \in \mathcal S_d^{{\rm Herm}}} {\|\D\|}_F^2 ={\|H_1\|}_F^2 + {\|H_2\|}_F^2,
	 \end{equation*}  
where infimum is uniquely attained by the matrix $H$.
\end{theorem}
\proof
Let us suppose that $\mathcal S_d^{{\rm Herm}} \neq \emptyset$. Then there exists $\D = [\D_1 ~ \D_2]$ with $\D_1\in \C^{n,n}$ and $\D_2 \in \C^{n,m}$ such  that  $\D_1^*= \D_1, \D x=y$, and $\D^* z=w$. This implies that $y^*z= (\D x)^* z = x^* \D^* z=x^* w$. Also, $z^* w_1= z^* \D_1^* z=z^* \D_1 z=(\D_1^* z)^*z=w_1^*z$, since $\D_1^*z=w_1$ and $\D_1^*=\D_1$. This implies that 
 $z^*w_1 \in \R $. Conversely, if $x^*w=y^*z, z^*w_1 \in \R $, then $H=[ H_1 ~ H_2] $ satisfies that  $H x=y, H^*z=w$, and $H_1^* = H_1$, which implies that $H \in \mathcal S_d^{{\rm Herm}}$.
	
Next, we prove~\eqref{eq : char double struct_H}. First suppose that 
$\D \in \mathcal S_d^{{\rm Herm}}$, i.e., $\D=[\D_1 ~ \D_2]$, such that $\D x = y, \D^* z = w$, and $\D_1 ^*=\D_1 $. This implies that 	
\begin{eqnarray}\label{val: eq1 dstruct_H} 
		\D_1 x_1 + \D_2 x_2 =y,  \quad\D_1 z = w_1  \quad \text{and}\quad
		\D_2^* z = w_2. 
\end{eqnarray}
Since $\Delta_1$ is a Hermitian matrix taking $z$ to $w_1$, from Theorem~\ref{theorem: herm map} $\D_1$ has the form
\begin{equation} \label{val: delta_1 map_H}
		\D_1 = w_1 z^\pls + (w_1 z ^\pls) ^* - (z^\pls w_1)z z^\pls + \mathcal P_z K \mathcal P_z 
\end{equation} 
for some Hermitian matrix $K \in \C^{n,n}$. By substituting $\D_1$ from~\eqref{val: delta_1 map_H} in~\eqref{val: eq1 dstruct_H}, we get	
\begin{eqnarray}\label{eq: delta2 stage 2_H} 	
		\D_2 x_2 = y- \left ( w_1 z^\pls + (w_1 z ^\pls) ^* - (z^\pls w_1)z z^\pls + \mathcal P_z K \mathcal P_z \right )x_1  \quad \text{and}\quad \D_2 ^* z = w_2,
\end{eqnarray} 
i.e., a mapping of the form $\D_2 x_2=\tilde y$ and $\D_2^*z=w_2$, where $\tilde y=y-  ( w_1 z^\pls + (w_1 z ^\pls) ^* - (z^\pls w_1)z z^\pls + \mathcal P_z K \mathcal P_z )x_1$. The vectors $x_2,\tilde y, z$, and $w_2$ satisfy 
\begin{eqnarray*}
\tilde y^*z&=&\left(y-  ( w_1 z^\pls + (w_1 z ^\pls) ^* - (z^\pls w_1)z z^\pls + \mathcal P_z K \mathcal P_z )x_1\right)^*z\\
	&=&y^*z- x_1^* w_1  \quad (\because (w_1 z^\pls + (w_1 z ^\pls) ^* - (z^\pls w_1)z z^\pls + \mathcal P_z K \mathcal P_z) z=\D_1z=w_1)\\
			&=& x_2^* w_2 \quad (\because x^*w= x_1^*w_1 + x_2^* w_2~\text{and}~x^*w=y^*z).
\end{eqnarray*}
Therefore, from Theorem~\ref{theorem : 2.1 mehl2017}, $\Delta_2$ can be written as 
\begin{equation} \label{val: delta2 map2unstruct_H}
		\D_2 = \tilde y x_2^\pls +  (w_2 z^\pls)^* - (w_2 z^\pls)^*x_2 x_2^\pls + \mathcal P_z R \mathcal P_{x_2},
	\end{equation}
for some $R \in \C^{n,m}$.	

Thus, in view of ~\eqref{val: delta_1 map_H} and~\eqref{val: delta2 map2unstruct_H},  we have 
\begin{eqnarray}
\left [\D_1 ~ \D_2\right ] &=& \big [ w_1 z^\pls + (w_1 z ^\pls) ^* - (z^\pls w_1)z z^\pls + \mathcal P_z K  \mathcal P_z ~~~ \tilde y x_2^\pls +  (w_2 z^\pls)^* - (w_2 z^\pls)^* x_2 x_2^\pls +  \mathcal P_z R  \mathcal P_{x_2}\big ]\nonumber \\
&=& \big[H_1+\tilde H_1(K)~~H_2+\tilde H_2(K,R)\big] \nonumber \\
&=& H + \wt H(K,R).
\end{eqnarray}
This proves ``$\subseteq$" in~\eqref{eq : char double struct_H}.
	
Conversely, let $[\Delta_1~\Delta_2]=[ H_1 + \wt H_1(K) ~~ H_2 + \wt H_2(K,R) ]$, where $H_1,\wt H_1(K),H_2$, and $\wt H_2(K,R)$ are defined by~\eqref{val: H_1 double struct_H}-\eqref{val: wt H_2 double struct_H} for some matrices $R \in \C^{n,m}$ and $K \in \C^{n,n}$ such that $K^*=K$. Then it is easy to check that $[\Delta_1~\Delta_2]x=y$ and 
$[\Delta_1~\Delta_2]^*z=w$ since $x^*w=y^*z$. Also $(H_1 + \wt H_1(K))^*= H_1 + \wt H_1(K)$ since  $z^* w_1 \in \R$ and $K^*=K$. 
Hence $[\Delta_1~\Delta_2] \in \mathcal S_d^{{\rm Herm}} $. This shows 
``$\supseteq$" in~\eqref{eq : char double struct_H}.

In view of~\eqref{eq : char double struct_H}, we have
\begin{eqnarray}
		\inf_{\D \in \mathcal S_d^{{\rm Herm}} } {\|\D\|}_F^2 &=& \inf_{K\in \C^{n,n},R\in \C^{n,m}, K^*=K} {\left \|H + \wt H(K,R)\right \|}_F^2 \nonumber \\ 
&=&  \inf_{K\in \C^{n,n},R\in \C^{n,m}, K^*=K}\left( { \big\| \big [ H_1 ~ H_2 \big]+ \big[ \wt H_1(K) ~ \wt H_2(K,R) \big]\big \|}_F^2 \right) \nonumber \\
&=& \inf_{K\in \C^{n,n},R\in \C^{n,m}, K^*=K}\left( { \big\| H_1 + \wt H_1(K)\big\|}_F^2 + {\big\| H_2 + \wt H_2(K,R)\big\|}_F^2 \right) \nonumber \\
&\geq & \inf_{K\in \C^{n,n}, K^*=K}{\big\| H_1 + \wt H_1(K)\big\|}_F^2  + \inf_{K\in \C^{n,n},R\in \C^{n,m}, K^*=K}  {\big\| H_2 + \wt H_2(K,R)\big\|}_F^2     \nonumber \\ \label {eq:firstineq}  \\
&=&{\| H_1\|}_F^2 + \inf_{K\in \C^{n,n},R\in \C^{n,m}, K^*=K} {\big\| H_2 + \wt H_2(K,R)\big\|}_F^2 \label {eq:secineq}  \\
&=& {\| H_1\|}_F^2 + \inf_{K\in \C^{n,n}, K^*=K} \Big( \inf_{R\in \C^{n,m}}  {\big \|H_2 + \wt H_2(K,R)\big\|}_F^2 \Big)     \nonumber \\
&=& {\| H_1\|}_F^2 + \inf_{K\in \C^{n,n}, K^*=K}  {\big\|H_2 + \wt H_2(K,0)\big \|}_F^2, \label{eq:thiineq} 
\end{eqnarray} 
where the first inequality in~\eqref{eq:firstineq} follows due to the fact that for any two real valued functions $f$ and $g$ defined on the same domain, $\inf(f+g)\geq \inf f+\inf g$. Also equality in~\eqref{eq:secineq} follows since the infimum in the first term is attained when $K=0$. In fact, for any $K\in \C^{n,n}$ such that $K^*=K$, we have $(H_1 + \wt H_1(K))z=w_1$, which implies from Theorem~\ref{theorem: herm map} that the minimum of ${\|H_1 + \wt H_1(K)\|}_F$ is attained when $K=0$. Further, for a fixed $K$ and for any $R \in \C^{n,m}$, $H_2 + \wt H_2(K,R)$ is a matrix satisfying $(H_2 + \wt H_2(K,R))x_2=\tilde y$ and $(H_2 + \wt H_2(K,R))^*z=w_2$. This implies from Theorem~\ref{theorem : 2.1 mehl2017} that for any fixed $K$, the minimum of  ${\|H_2 + \wt H_2(K,R)\|}_F$ over $R$ is attained when $R=0$, which yields~\eqref{eq:thiineq}. This proves~\eqref{eq: minimal norm double struct_H}.

Next suppose if $x_1=\alpha z$ for some nonzero $\alpha \in \C$, then $\wt H_2(K,0)=0$ for every $K \in \C^{n,n}$. This implies from~\eqref{eq:thiineq} that 
\begin{eqnarray}
\inf_{\D \in \mathcal S_d^{{\rm Herm}} } {\|\D\|}_F^2~\geq~ {\| H_1\|}_F^2+{\| H_2\|}_F^2\,= \, {\| H\|}_F^2,
\end{eqnarray}
and in this case the lower bound is attained since $H \in \mathcal S_d^{{\rm Herm}}$. This completes the proof. 
\eproof

\subsection{Doubly structured skew-Hermitian mappings}	

A result analogous to Theorem~\ref{theorem HD} can be obtained for doubly structured skew-Hermitian mappings (DSSHMs), i.e., when $\mathbb S={\rm SHerm}(n)$ in Problem~\ref{def:dsmprob}. In the following, we state the result for DSSHMs and skip its proof as it is similar to the proof of Theorem~\ref{theorem HD}.
\begin{theorem} \label{theorem_SHD}
	Given $x=[x_1^T~x_2^T]^T$ with $x_1 \in \C^{n}$ and $x_2 \in \C^m$, $y,z \in \C^n$, and $w=[w_1^T~w_2^T]^T$ with $w_1 \in \C^n$ and $w_2 \in \C^m$. Define 
\begin{equation*} \label{def: double struct mapping set_SH }
\mathcal S_d^{{\rm SHerm}}  := \left \{ \D :~\D = [\D_1 ~ \D_2],\, \D_1\in \C^{n,n}, \, \D_2 \in \C^{n,m},\,  \D_1^* = -\D_1,\, \D x=y, \,\D^*z=w \right \}.
	\end{equation*}
	Then $\mathcal S_d^{{\rm SHerm}} \neq \emptyset $ if and only if $x^*w=y^*z$ and $z^*w_1 \in i\R$. If the later conditions hold true, then
		\begin{equation*}\label{eq : char double struct_SH}
			\mathcal S_d^{{\rm SHerm}}  = \left \{ H + \wt H(K,R)  : ~ K\in \C ^{n,n},\, R \in \C^{n,m},\, K^*=-K \right \},
		\end{equation*}
		where $H=[H_1 ~ H_2]$ and $\wt H(K,R) = [\wt H_1(K) ~ \wt H_2(K,R) ]$ with $H_1,H_2,\wt H_1(K), \wt H_2(K,R)$ given by 
		\begin{eqnarray*}
			H_1 &=& -w_1 z^\pls + (w_1 z ^\pls) ^* + (z^\pls w_1)z z^\pls \label{val: H_1 double struct_SH}, \\
			H_2 &=& yx_2^\pls- \left( -w_1 z^\pls + (w_1 z ^\pls) ^* + (z^\pls w_1)z z^\pls\right)x_1  x_2^\pls + (w_2 z^\pls)^*\mathcal P_{x_2}  \label{val: H_2 double struct_SH}, \\
			\wt H_1(K)&=& P_z K P_z \label{val: wt H_1 double struct_SH}, \\
			\wt H_2(K,R) &=& P_z R P_{x_2} - P_z K P_z x_1 x_2^\pls \label{val: wt H_2 double struct_SH},
		\end{eqnarray*}
and
\begin{equation} \label{eq: minimal norm double struct_SH}
 \inf_{\D \in \mathcal S_d^{{\rm SHerm}}} {\|\D\|}_F^2 ~\geq~ {\| H_1\|}_F^2 + \inf_{K\in \C^{n,n}}  {\big\|H_2 + \wt H_2(K,0)\big \|}_F^2.	
\end{equation}  		
Moreover, if $x_1 = \alpha z$ for some nonzero $\alpha \in \C$, then equality holds in~\eqref{eq: minimal norm double struct_SH} and we have 
\begin{equation*} 
	 	\inf_{\D \in \mathcal S_d^{{\rm SHerm}}} {\|\D\|}_F^2 ={\|H_1\|}_F^2 + {\|H_2\|}_F^2,
	 \end{equation*}  
where infimum is uniquely attained by the matrix $H$.
\end{theorem}

\subsection{Doubly structured complex symmetric/skew-symmetric mappings}

In this section, we consider the doubly structured symmetric mapping (DSSM) problem, i.e., when $\mathbb S={\rm Sym}(n)$ in Problem~\ref{def:dsmprob}. We have the following result for DSSMs, proof of which is kept in~\ref{app:proofcomsym}.

\begin{theorem} \label{theorem:DSSM}
	Given $x=[x_1^T~x_2^T]^T$ with $x_1 \in \C^{n}$ and $x_2 \in \C^m$, $y,z \in \C^n$, and $w=[w_1^T~w_2^T]^T$ with $w_1 \in \C^n$ and $w_2 \in \C^m$. Define 
	\begin{equation} \label{def: double struct mapping set_T }
		\mathcal S_d^{{\rm Sym}}  := \left \{ \D :~\D = [\D_1 ~ \D_2],\, \D_1\in \C^{n,n}, \, \D_2 \in \C^{n,m},\,  \D_1^T = \D_1,\, \D x=y, \,\D^*z=w \right \}.
	\end{equation}
	Then $\mathcal S_d^{{\rm Sym}} \neq \emptyset $ if and only if $x^*w=y^*z$. If the later conditions hold true, then
	\begin{equation}\label{eq : char double struct_T}
		\mathcal S_d^{{\rm Sym}}  = \left \{ H + \wt H(K,R)  : ~ K\in \C ^{n,n},\, R \in \C^{n,m},\, K^T=K \right \},
	\end{equation}
	where $H=[H_1 ~ H_2]$ and $\wt H(K,R) = [\wt H_1(K) ~ \wt H_2(K,R) ]$ with $H_1,H_2,\wt H_1(K), \wt H_2(K,R)$ given by 
	\begin{eqnarray}
		H_1 &=& \bar{w}_1 \bar{z}^\pls + (\bar{w}_1 \bar{z}^\pls) ^T - \bar{z}^{\pls^T} \bar{z}^T \bar{w}_1 \bar{z}^\pls \label{val: H_1 double struct_T}, \\
		H_2 &=& yx_2^\pls- \left(\bar{w}_1 \bar{z}^\pls + (\bar{w}_1 \bar{z}^\pls) ^T - \bar{z}^{\pls^T} \bar{z}^T \bar{w}_1 \bar{z}^\pls\right )x_1  x_2^\pls + (w_2 z^\pls)^*\mathcal P_{x_2}   \label{val: H_2 double struct_T}, \\
		\wt H_1(K)&=& \mathcal P_{\bar{z}}^T K \mathcal P_{\bar{z}} \label{val: wt H_1 double struct_T}, \\
		\wt H_2(K,R) &=& \mathcal P_z R \mathcal P_{x_2} - (\mathcal P_{\bar{z}})^T K \mathcal P_{\bar{z}} x_1 x_2^\pls \label{val: wt H_2 double struct_T},
	\end{eqnarray}
	and
	\begin{equation} \label{eq: minimal norm double struct_T}
		\inf_{\D \in \mathcal S_d^{{\rm Sym}}} {\|\D\|}_F^2 ~\geq~ {\| H_1\|}_F^2 + \inf_{K\in \C^{n,n}}  {\big\|H_2 + \wt H_2(K,0)\big \|}_F^2.	
	\end{equation}  		
	Moreover, if $x_1 = \alpha \bar z$ for some nonzero $\alpha \in \C$, then equality holds in~\eqref{eq: minimal norm double struct_T} and we have
	\begin{equation*} 
		 	\inf_{\D \in \mathcal S_d^{{\rm Sym}}} {\|\D\|}_F^2 ={\|H_1\|}_F^2 + {\|H_2\|}_F^2,
		 \end{equation*}  
where infimum is uniquely attained by the matrix $H$.
\end{theorem}
\proof See~\ref{app:proofcomsym}.
\eproof

A result analogous to Theorem~\ref{theorem:DSSM} can be obtained for doubly structured skew-symmetric mappings, i.e., when $\mathbb S={\rm SSym}(n)$ in Problem~\ref{def:dsmprob} as follows.

\begin{theorem} \label{theorem STD}
			Given $x=[x_1^T~x_2^T]^T$ with $x_1 \in \C^{n}$ and $x_2 \in \C^m$, $y,z \in \C^n$, and $w=[w_1^T~w_2^T]^T$ with $w_1 \in \C^n$ and $w_2 \in \C^m$. Define 
			\begin{equation} \label{def: double struct mapping set_ST }
				\mathcal S_d^{{\rm SSym}}  := \left \{ \D :~\D = [\D_1 ~ \D_2],\, \D_1\in \C^{n,n}, \, \D_2 \in \C^{n,m},\,  \D_1^T = -\D_1,\, \D x=y, \,\D^*z=w \right \}.
			\end{equation}
			Then $\mathcal S_d^{{\rm SSym}} \neq \emptyset $ if and only if $x^*w=y^*z$ and $z^Tw_1 =0$. If the later conditions hold true, then
			\begin{equation}\label{eq : char double struct_ST}
				\mathcal S_d^{{\rm SSym}}  = \left \{ H + \wt H(K,R)  : ~ K\in \C ^{n,n},\, R \in \C^{n,m},\, K^T=-K \right \},
			\end{equation}
			where $H=[H_1 ~ H_2]$ and $\wt H(K,R) = [\wt H_1(K) ~ \wt H_2(K,R) ]$ with $H_1,H_2,\wt H_1(K), \wt H_2(K,R)$ given by 
			\begin{eqnarray}
				H_1 &=& -\bar{w}_1 \bar{z}^\pls + (\bar{w}_1 \bar{z}^\pls) ^T + \bar{z}^{\pls^T} \bar{z}^T \bar{w}_1 \bar{z}^\pls \label{val: H_1 double struct_ST}, \\
				H_2 &=& yx_2^\pls- \left(-\bar{w}_1 \bar{z}^\pls + (\bar{w}_1 \bar{z}^\pls) ^T + \bar{z}^{\pls^T} \bar{z}^T \bar{w}_1 \bar{z}^\pls\right )x_1  x_2^\pls + (w_2 z^\pls)^*\mathcal P_{x_2}   \label{val: H_2 double struct_ST}, \\
				\wt H_1(K)&=& \mathcal P_{\bar{z}}^T K \mathcal P_{\bar{z}} \label{val: wt H_1 double struct_ST}, \\
				\wt H_2(K,R) &=& \mathcal P_z R \mathcal P_{x_2} - (\mathcal P_{\bar{z}})^T K \mathcal P_{\bar{z}} x_1 x_2^\pls \label{val: wt H_2 double struct_ST},
			\end{eqnarray}
			and
			\begin{equation} \label{eq: minimal norm double struct_ST}
				\inf_{\D \in \mathcal S_d^{{\rm SSym}}} {\|\D\|}_F^2 ~\geq~ {\| H_1\|}_F^2 + \inf_{K\in \C^{n,n}}  {\big\|H_2 + \wt H_2(K,0)\big \|}_F^2.	
			\end{equation}  		
			Moreover, if $x_1 = \alpha \bar z$ for some nonzero $\alpha \in \C$, then equality holds in~\eqref{eq: minimal norm double struct_ST} and  we have
\begin{equation*} 
		 	\inf_{\D \in \mathcal S_d^{{\rm SSym}}} {\|\D\|}_F^2 ={\|H_1\|}_F^2 + {\|H_2\|}_F^2,
		 \end{equation*}  
		where infimum is uniquely attained by the matrix $H$.
	\end{theorem}
	\proof The proof follows on the lines of the proof of Theorem~\ref{theorem:DSSM} (see~\ref{app:proofcomsym}) by using Theorem~\ref{theorem: Complex skew-sym map} in place of Theorem~\ref{theorem: Complex-sym map}.	
	\eproof

\section{Solution to the doubly structured semidefinite mapping problem}
\label{sec:dssmp}

This section considers the doubly structured positive semidefinite mapping (DSPSDM) problem, i.e.,  when $\mathbb S$ is the set of all $n \times n$ positive semidefinite matrices in Problem~\ref{def:dsmprob}. We first prove a lemma that will be needed in characterizing the set of all solutions to the DSPSD mapping problem.

\begin{lemma} \label{lemtracecondA}
	Let $A,B \in \C^{n,n}$ such that $B \succeq 0$. If $\lambda \in \Lambda(A)$ implies that $\real(\lambda) \leq 0$, then 
$\real{(\trace(A B))} \leq 0$.
\end{lemma}
\proof  Let $\lambda_1,\ldots,\lambda_n$ be the eigenvalues of $A$. Then by assumption $\real{(\lambda_i)} \leq 0$ for all $i$. Also, $B \succeq  0$ implies that there exists a unitary matrix $U \in \C^{n,n}$ such that $U^*BU=D$, where $D=\text{diag}(d_1,\ldots,d_n)$ with $d_i \geq 0$ for all $i$. Thus we have
\begin{equation}
\text{trace}(AB)=\text{trace}(AUDU^*)=\text{trace}(U^*AUD)= \sum_{j=1}^n \widetilde{a}_{jj}d_j,
\end{equation}
where $U^*AU = (\widetilde{a}_{ij})$. This implies that 
\begin{eqnarray*}
\real{(\text{trace}(AB))}=\sum_{j=1}^n \real{(\widetilde{a}_{jj})}d_j 
\leq \max_{j} \widetilde{a}_{jj} \cdot \sum_{j=1}^n \real{(\widetilde{a}_{jj})} =  \max_{j} \widetilde{a}_{jj} \cdot \sum_{j=1}^n \real{\lambda_j} \leq 0, 
\end{eqnarray*}
since $A$ and $U^*AU$ are unitary similar and have the same eigenvalues, and 
$d_i\geq 0$ and $\real{(\lambda_i)}\leq 0$ for all $i=1,\ldots ,n$.
\eproof

We have the following result that completely solves the existence and characterization problem for DSPSDMs and provides sufficient conditions for the minimal norm solution to the DSPSDM problem.

\begin{theorem} \label{theorem PSDD}	
Given $x=[x_1^T~x_2^T]^T$ with $x_1 \in \C^{n}$ and $x_2 \in \C^m\setminus \{0\}$, $y,z \in \C^n$, and $w=[w_1^T~w_2^T]^T$ with $w_1 \in \C^n\setminus\{0\}$ and $w_2 \in \C^m\setminus \{0\}$. Define 
\begin{equation} \label{def: double struct mapping set_PSD }
\mathcal S_d^{{\rm PSD}}  := \left \{ \D :~\D = [\D_1 ~ \D_2],\, \D_1\in \C^{n,n}, \, \D_2 \in \C^{n,m},\,  \D_1 \succeq 0,\, \D x=y, \,\D^*z=w \right \}.
	\end{equation}
	Then $\mathcal S_d^{{\rm PSD}} \neq \emptyset $ if and only if $x^*w=y^*z$ and $z^*w_1 > 0$. If the later conditions hold true, then
\begin{equation}\label{eq : char double struct_PSD}
\mathcal S_d^{{\rm PSD}}  = \left \{ H + \wt H(K,R)  : ~ K\in \C ^{n,n},\, R \in \C^{n,m},\, K\succeq 0 \right \},
		\end{equation}	
		where $H=[H_1 ~ H_2]$ and $\wt H(K,R) = [\wt H_1(K) ~ \wt H_2(K,R) ]$ with $H_1,H_2,\wt H_1(K), \wt H_2(K,R)$ given by 
\begin{eqnarray}
			H_1 &=& \frac{w_1 w_1^*}{z^*w_1}  \label{val: H_1 double structPSD} \\
			H_2 &=& \frac{y x_2^*}{\|x_2\|^2}- \frac{w_1 w_1^* x_1 x_2^\dagger}{(z^*w_1)} + \frac{z w_2^*}{\|z\|^2} - \frac{(w_2^* x_2)z x_2^\dagger}{\|z\|^2 } \label{val: H_2 double structPSD} \\
			\wt H_1(K)&=& \mathcal P_z K \mathcal P_z \label{val: wt H_1 double structPSD} \\
			\wt H_2(K,R) &=& \mathcal P_z R \mathcal P_{x_2} - \mathcal P_z K \mathcal P_z x_1 x_2^\dagger\label{val: wt H_2 double structPSD},
\end{eqnarray}
and
\begin{equation} \label{eq: minimal norm double struct_PSD}
 \inf_{\D \in \mathcal S_d^{{\rm PSD}}} {\|\D\|}_F^2 ~\geq~ {\| H_1\|}_F^2 + \inf_{K\in \C^{n,n}}  {\big\|H_2 + \wt H_2(K,0)\big \|}_F^2.	
\end{equation}  	
Moreover, if $x_1 = \alpha z$ for some nonzero $\alpha \in \C$, or, if all eigenvalues of the matrix 
$\mathcal M:=y x_1^*-  \frac{ w_1^*x_1}{z^*w_1}w_1 x_1^* $ lie in the left half of the complex plane, i.e., $\lambda \in \Lambda (\mathcal M)$ implies that $\real{(\lambda)}\leq 0$, then we have 
\begin{equation*}
 \inf_{\D \in \mathcal S_d^{{\rm PSD}}} {\|\D\|}_F^2 ~=~ {\| H_1\|}_F^2 + {\| H_2\|}_F^2 ,
\end{equation*}  	
where the infimum is uniquely attained by the matrix $H$.
\end{theorem}

\proof
Suppose that $\mathcal S_d^{{\rm PSD}} \neq \emptyset$ and let $\D \in \mathcal S_d^{{\rm PSD}} $. Then $\D = [\D_1 ~ \D_2]$ with $\D_1\in \C^{n,n}, \, \D_2 \in \C^{n,m}$ such  that $\D_1 \succeq 0$, $\D x=y$, and $\D^* z=w$. This implies that 
$y^*z= (\D x)^* z = x^* \D^* z=x^* w$. Also, $z^* w_1= z^* \D_1^* z=z^* \D_1 z \geq 0$, since  $\D_1^*z=w_1$ and $\D_1^*=\D_1 \succeq 0$. In fact, we have $z^* w_1 > 0$, since if $z^* w_1=0$, then this implies that $ z^* \D_1^* z=0$ and hence $w_1=\D_1z=0$, which is a contradiction as $w_1 \neq 0$.
Conversely, let $x^*w=y^*z$ and $z^*w_1>0$. Then it is easy to see that the matrix $H=[ H_1 ~ H_2 ] $ satisfies $H x=y, H^*z=w$. Also $H_1 \succeq 0$ for being a rank one symmetric matrix, which implies that $H \in \mathcal S_d^{{\rm PSD}}$.

Next, we prove~\eqref{eq : char double struct_PSD}. For this, let 
$\D \in \mathcal S_d^{{\rm PSD}}$, i.e., $\D=[\D_1 ~ \D_2]$ such that $\D x = y, \D^* z = w$, and  $\D_1 \succeq 0$.  This implies that 
\begin{equation}\label{eq:d1d2psd}
\D_1 x_1 + \D_2 x_2=y, \quad \D_1z=w_1\quad \text{and}\quad 
\D_2^* z = w_2.
\end{equation}
Since $\D_1 \succeq 0$ taking $z$ to $w_1$, from Theorem~\ref{lem:psdmap} $\D_1$ has the form
\begin{equation} \label{val:D1psdmap}
\D_1 = \frac{w_1 w_1^*}{z^*w_1} + \mathcal P_z K \mathcal P_z, 
\end{equation} 
for some positive semidefinite matrix $K \in \C^{n,n}$.	By substituting $\Delta_1$ from~\eqref{val:D1psdmap} in~\eqref{eq:d1d2psd}, we get 
\begin{equation}\label{eq: delta2 stage 2}
\D_2 x_2 =\tilde y \quad \text{and}\quad
\D_2 ^* z  = w_2,
 \end{equation} 
 where $\tilde y=y- \left ( \frac{w_1 w_1^*}{z^*w_1} + \mathcal P_z K \mathcal P_z \right )x_1 $. Again since $\tilde y^*z=x_2^*w_2$, in view of Theorem~\ref{theorem : 2.1 mehl2017} $\Delta_2$ has the form
\begin{equation} \label{val: delta2 map2unstruct }
	\D_2 = \tilde y x_2^\dagger + (w_2 z^\dagger)^*- (w_2 z^ 
\dagger)^*x_2x_2^\dagger +
\mathcal P_zR\mathcal P_{x_2},
\end{equation}
for some $R \in \C^{n,m}$. Thus in view of~\eqref{val:D1psdmap} and~\eqref{val: delta2 map2unstruct }, we have 
\begin{eqnarray}
[\D_1 ~\, \D_2] &=& \mat{cc} \frac{w_1 w_1^*}{z^*w_1} + \mathcal P_z K \mathcal P_z &
	\tilde y x_2^\dagger + (w_2 z^\dagger)^*- (w_2 z^\dagger)^*x_2x_2^\dagger +\mathcal P_zR\mathcal P_{x_2} \rix \nonumber \\
	&=& \mat{cc} \frac{w_1 w_1^*}{z^*w_1} & y x_2^\dagger- 
	(w_1^* x_1)w_1x_2^\dagger + (w_2 z^\dagger)^*- (w_2 z^\dagger)^*x_2x_2^\dagger \rix  \nonumber \\
	&+& \left[\mathcal P_z K \mathcal P_z ~~ \mathcal P_z R \mathcal P_{x_2} - \mathcal P_z K \mathcal P_z x_1 x_2^\dagger  \right]\nonumber \\
&=& \big[H_1 ~ H_2 \big ] + \big[\wt H_1(K) ~\, \wt H_2(K,R) \big] \nonumber \\
		&=& H + \wt H(K,R).
\end{eqnarray}
This proves ``$\subseteq$" in~\eqref{eq : char double struct_PSD}.

Conversely, let $x^*w=y^*z$ and $z^* w_1 >0 $ and consider $[\Delta_1~\Delta_2]=[ H_1 + \wt H_1(K) ~\, H_2 + \wt H_2(K,R) ]$, where $H_1, \wt H_1(K), H_2$, and $\wt H_2(K,R)$ given by~\eqref{val: H_1 double structPSD}-\eqref{val: wt H_2 double structPSD} for some matrices $R \in \C^{n,m}$ and $K \in \C^{n,n}$ such that $K\succeq 0$. Then, it is easy to verify that $[\Delta_1~\Delta_2]x=y$ and 
$[\Delta_1~\Delta_2]^*z=w$. Also, $\Delta_1=H_1 +\wt H_1(K) \succeq 0$, being the sum of two positive semidefinite matrices. This implies $[\Delta_1~\Delta_2] \in \mathcal S_d^{{\rm PSD}}$ and hence shows $``\supseteq"$ in~\eqref{eq : char double struct_PSD}. 

In view of~\eqref{eq : char double struct_PSD} and by following the arguments similar to the proof of~\eqref{eq: minimal norm double struct_H} in Theorem~\ref{theorem HD}, we have that
\begin{equation}\label{eq:thiineqpsd} 
\inf_{\D \in  \mathcal S_d^{{\rm PSD}} } {\|\D\|}_F^2 \geq 
{\| H_1\|}_F^2 + \inf_{K\in \C^{n,n}, K\succeq 0}  {\big\|H_2 + \wt H_2(K,0)\big \|}_F^2.
\end{equation}
Next we show that equality holds in~\eqref{eq:thiineqpsd}  for two cases. First suppose that $x_1=\alpha z$ for some nonzero $\alpha \in \C$. Then  $\mathcal P_zx_1 =0$ and thus $\wt H_2(K,0)=0$ for any $K \succeq 0$. This implies from~\eqref{eq:thiineqpsd}  that 
\[
\inf_{\D \in  \mathcal S_d^{{\rm PSD}} } {\|\D\|}_F^2\geq {\| H_1\|}_F^2 +{\| H_2\|}_F^2 ={\| H\|}_F^2, 
\]
and the lower bound is uniquely attained since $H \in \mathcal S_d^{{\rm PSD}}$. 
Now suppose that $\lambda \in \Lambda (\mathcal M)$, where $\mathcal M:=y x_1^*-  \frac{ w_1^*x_1}{z^*w_1}w_1 x_1^* $ implies that $\real{(\lambda)}\leq 0$. Then for any $K \in \C^{n,n}$ such that $K \succeq 0$, we have
\begin{eqnarray}
	&&{\|H_2 + \wt H_2 (K,0)\|}_F^2 \nonumber\\
&=& {\|H_2\|}_F^2 + {\|\wt H_2(K,0)\|}_F^2 + 2 \real\left(\trace \big(H_2 \wt H_2^*(K,0)\big)\right) \nonumber \\
&=& {\|H_2\|}_F^2 + {\|\wt H_2(K,0)\|}_F^2 + 2 \real\bigg(\trace \Big( \big(  \frac{y x_2^*}{\|x_2\|^2}- \frac{w_1 w_1^* x_1 x_2^*}{(z^*w_1)\|x_2\|^2} + \frac{z w_2^*}{\|z\|^2} - \frac{(w_2^* x_2)z x_2^*}{\|z\|^2 \|x_2\|^2} \big ) \big ( - \frac {\mathcal P_z K \mathcal P_z x_1 x_2^*}{\|x_2\|^2} \big )^*  \Big)\bigg) \nonumber \\
&=& {\|H_2\|}_F^2 + {\|\wt H_2(K,0)\|}_F^2 - 2 \real\bigg(\trace \Big(  \mathcal M  \frac{  \mathcal P_z K \mathcal P_z}{\|x_2\|^2}  \Big) \bigg) \label{condA_psd111} \\
&\geq& {\|H_2\|}_F^2 + {\|\wt H_2(K,0)\|}_F^2 \geq   {\|H_2\|}_F^2, \label{condA_psd1}
\end{eqnarray}
where~\eqref{condA_psd111} follows by repeated use of the identity 
$\text{trace}(AB)=\text{trace}(BA)$ for matrix $A,B \in \C^{n,n}$, and the fact that $\mathcal P_z z=0$. The last inequality~\eqref{condA_psd1} follows because of the fact that  $\real(\text{trace}(\mathcal M  \mathcal P_z K \mathcal P_z ))\leq 0$, this is due to Lemma~\ref{lemtracecondA}, since $\mathcal P_z K \mathcal P_z  \succeq 0$ as $K \succeq 0$, and by assumption that $\lambda \in \Lambda (\mathcal M)$ implies $\real{(\lambda)} \leq 0$.
This implies from~\eqref{condA_psd1}  that 
\begin{equation}\label{condA_psd2}
\inf_{K\in \C^{n,n}, K\succeq 0}  {\big\|H_2 + \wt H_2(K,0)\big \|}_F^2 \geq {\| H_2\|}_F^2. 
\end{equation}
Thus from~\eqref{eq:thiineqpsd} and~\eqref{condA_psd2}, we have that
\[
\inf_{\D \in  \mathcal S_d^{{\rm PSD}} } {\|\D\|}_F^2\geq {\| H_1\|}_F^2 +{\| H_2\|}_F^2 ={\| H\|}_F^2, 
\]
and  again the lower bound is uniquely attained since $H \in \mathcal S_d^{{\rm PSD}}$. This completes the proof. 
\eproof

\begin{remark}{\rm \label{rem:dsnsdm}
We note that although in Theorem~\ref{theorem PSDD}, we considered only the DSPSDM problem, there is a corresponding result for the doubly structured negative semidefinite mapping (DSNSDM) problem, i.e. when $\mathbb S$ is the set of $n \times n$ negative semidefinite matrices in Problem~\ref{def:dsmprob}. The corresponding result for DSNSDMs follows from Theorem~\ref{theorem PSDD} by replacing $w_1$ with $-w_1$ and $x_1$ with $-x_1$.
}
\end{remark}

\section{Solution to the doubly structured dissipative mapping problem}\label{sec:DSDMs}

Let ${\rm Diss}(n)$ denote the set of all $n \times n$ dissipative matrices, i.e., $A \in {\rm Diss}(n)$ implies that $A+A^* \succeq 0$.
In this section, we consider two types of doubly structured dissipative mapping (DSDM) problems: (i) for given vectors $x,w,y,z \in \C^n$, find $\Delta \in {\rm Diss}(n)$ such that $\Delta x=y$, and $\Delta^*z=w$. We call this mapping problem as \emph{Type-1 DSDM problem}; (ii) for given vectors $x,w \in \C^{n+m}$, and $y,z \in \C^n$, find $\Delta=[\Delta_1~\Delta_2]$ with $\Delta_1 \in {\rm Diss}(n)$ and $\Delta_2 \in \C^{n,m} $ such that  $\Delta x=y$ and  $\Delta^*z=w$. We call this mapping problem a \emph{Type-2 DSDM problem}.

\subsubsection{Type-1 doubly structured dissipative mappings}

In the following, we tackle the type-1 DSDM problem in a general case when $X,Y,Z$, and $W$ are matrices of size $n \times m$. The result provides a complete, unified, and explicit solution to the Type-1 DSDM problem when $X$ and $Z$ share the same range space.

\begin{theorem}\label{map result: NSPSD 4vec}
	Let $X, Y, Z, W \in \C^{n,m}$ with $\text{rank}(X)=\text{rank}(Z)=r$, and let $X=U \Sigma V^*,Z=\wt U \wt \Sigma \wt V^*$ be the singular value decompositions of $X$ and $Z$ with $U=\left[U_1 ~ U_2\right], \wt U=\left[\wt U_1 ~ \wt U_2\right]$, where $U_1,\wt U_1 \in \C^{n,r}$.
Define  $\mathcal S_{d_1}^{Diss}:=\{\D \in \C^{n,n}~:~\D + \D^* \succeq 0,~\D X=Y,\D^*Z=W\}$. Suppose that $U_1=\wt U_1$ and $\ker(U_1^*(YX^\pls + (YX^\pls)^*)U_1) \subseteq \ker(U_2^*(YX^\pls + WZ^\pls)U_1) $. Then $\mathcal S_{d_1}^{Diss} \neq \emptyset$ if and only if 
\begin{equation}\label{eq:necsufcond}
YX^\pls X=Y,\quad WZ^\pls Z=W,\quad X^*W=Y^*Z,\quad X^*Y+Y^*X \succeq 0.
\end{equation}
Moreover, if $\mathcal S_{d_1}^{Diss} \neq \emptyset$, then 
\begin{enumerate}
		\item {\rm Characterization:} 
\begin{eqnarray}\label{eq:S_char n}
	\mathcal S_{d_1}^{Diss}  = \left \{U \mat{cc}U_1^* YX^\pls U_1 & U_1^* (WZ^\pls)^*  U_2 \\ U_2^*YX^\pls U_1 &  U_2^*(K+G)U_2 \rix U^*:~K,G \in \C^{n,n}~\text{satisfy}~\eqref{eq:cond1 n} \right \},
\end{eqnarray} 
		where
		\begin{equation}\label{eq:cond1 n}
G^*=-G,\quad K\succeq 0,\quad K-U_2JU_2^* \succeq 0, 
		\end{equation}
		with $J= \frac{1}{2} U_2^*(YX^\pls + WZ^\pls){\left(YX^\pls + (YX^\pls)^* \right)}^{\pls}(YX^\pls + WZ^\pls)^*U_2$.
\item {\rm Minimal norm mapping:}~
		\begin{equation}\label{eq:minres n}
			\inf_{A \in \mathcal S_{d_1}^{Diss} } {\|A\|}_F^2 \; = \; {\|YX^\pls\|}_F^2 + {\|WZ^\pls\|}_F^2-{\rm trace}\left(WZ^\pls(WZ^\pls)^*XX^\pls   \right) + {\|J\|}_F^2,
		\end{equation}
		where the infimum is uniquely attained by the matrix 
		\begin{eqnarray}\label{eq: min norm 4vec}
				\mathcal H :=YX^\pls + (WZ^\pls)^* -{(WZ^\pls)}^* XX^\pls + \mathcal P_Z U_2 JU_2^* \mathcal P_X  
				=U \mat{cc} U_1^* YX^\pls U_1 & U_1^*(WZ^\pls)^* U_2 \\ U_2^* YX^\pls U_1 & J \rix U^* \nonumber
		\end{eqnarray}
		which is obtained by setting $K=U_2JU_2^*$ and $G=0$ in~\eqref{eq:S_char n}.
	\end{enumerate}
\end{theorem}
\begin{proof}
	First suppose that  $\D \in \mathcal S_{d_1}^{Diss} $, i.e., $\D+\D^* \succeq 0$, $\D X=Y$, and $\D^* Z=W$. Then clearly 
	$YX^\pls X= \D X X^\pls  X= \D X=Y$ and  $WZ^\pls Z= \D^* Z Z^\pls  Z= \D^* Z=W$. Also, $	X^*W=X^*\D^*Z=(\D X)^*Z=Y^*Z$ and 
$ X^*Y+Y^*X= X^* \D X + X^* \D^* X= X^*(\D + \D ^*)X \succeq 0$.
Conversely, suppose that  $X,Y,Z$, and $W$ satisfy~\eqref{eq:necsufcond}. Then the matrix $\mathcal H$ defined in \eqref{eq: min norm 4vec} satisfies $\mathcal H X=Y$ and $\mathcal H^* Z=W$. Further, we have
	\begin{eqnarray*}
		\mathcal H +{\mathcal H}^*&=& U \mat{cc} U_1^* \left(YX^\pls+(YX^\pls)^*\right) U_1 & U_1^*\left((WZ^\pls)^*+(YX^\pls)^*\right) U_2 \\  U_2^*\left(WZ^\pls+YX^\pls\right) U_1 & 2J  \rix U^*,
	\end{eqnarray*}
with $J= \frac{1}{2} U_2^*(YX^\pls + WZ^\pls){\left(YX^\pls + (YX^\pls)^* \right)}^{\pls}((YX^\pls)^* + (WZ^\pls)^*)U_2$. 
Clearly $J=J^*$ and in fact $J \succeq 0$ from Lemma~\ref{lem8}, since $X^*Y+Y^*X \succeq 0$. Thus
in view of 
Lemma~\ref{lem:psd}, we have $\mathcal H +{\mathcal H}^* \succeq 0$, since from Lemma~\ref{lem8}
$U_1^* \left(YX^\pls+(YX^\pls)^*\right) U_1 \succeq 0$, by assumption 
 $\ker(U_1^*(YX^\pls + (YX^\pls)^*)U_1) \subseteq \ker(U_2^*(YX^\pls + WZ^\pls)U_1) $, and
$J\succeq 0$.  
This implies that $\mathcal H \in \mathcal S_{d_1}^{Diss}$.
	 
Next, we prove~\eqref{eq:S_char n}. First suppose that $\D \in \mathcal S_{d_1}^{Diss}$, i.e., $\D+\D^* \succeq 0$, $\D X=Y$, and $\D^* Z=W$.
Let $\Sigma=\mat{cc}\Sigma_1 & 0\\ 0 &0\rix, \wt \Sigma=\mat{cc}\wt \Sigma_1 & 0\\ 0 &0\rix, V=\mat{cc}V_1 & V_2 \rix,\wt V=\mat{cc}\wt V_1 & \wt V_2 \rix$, where $V_1, \wt V_1 \in \C^{m,r}$, $V_2,\wt V_2 \in \C^{m, m-r}$, and $\Sigma_1, \wt \Sigma_1 \in \C^{r,r}$ such that $X=U_1 \Sigma_1 V_1^*$ and $Z=\wt U_1 \wt \Sigma_1 \wt V_1^*$ become the reduced SVDs of $X$ and $Z$, respectively. Now consider
\begin{equation}\label{equation: delta N wt delta}
\wh{\D} = U^*\D U= \wh \D_H + \wh \D_S,
\end{equation} 
 where
	\[
	\wh \D_H=U^*\D_H U=\mat{cc}H_{11}&H_{12}\\H_{12}^* & H_{22}\rix\quad \text{and}\quad
	\wh \D_S=U^*\D_S U=\mat{cc}S_{11}&S_{12}\\-S_{12}^* & S_{22}\rix.
	\]
	Clearly, ${\|\D\|}_{F}={\|\wh \D \|}_{F}$, since Frobenius norm is unitarily invariant, and also $\D_H \succeq 0$ if and only if $\wh \D_H \succeq 0$.	As $\D X=Y$, we have $U^* \Delta U U^* X = U^*Y$ which implies that 
\begin{eqnarray}
			 &\wh \D \mat{c}U_1^*  \\U_2^* \rix X = \mat{c}U_1^*Y \\U_2^*Y \rix \Longrightarrow
			\mat{cc} H_{11}+S_{11} & H_{12}+S_{12}\\ H_{12}^*-S_{12}^* & H_{22}+S_{22}\rix \mat{c} \Sigma_1 V_1^* \\  0\rix=\mat{c} U_1^*Y \\  U_2^*Y \rix	.
\end{eqnarray}
This implies that
\begin{equation}\label{eq:h11}
		\left(H_{11}+S_{11} \right)\Sigma_1 V_1^*=U_1^*Y
\quad \text{and}\quad 
		\left(H_{12}^*-S_{12}^* \right)\Sigma_1 V_1^* =  U_2^* Y.
	\end{equation}
	Thus from~\eqref{eq:h11}, we have 
	\begin{equation}\label{h11+s11}
		H_{11} + S_{11}=U_1^* Y V_1 \Sigma_1^{-1}=U_1^* Y X^\pls U_1
	\end{equation}
	and
	\begin{equation}\label{h12*-s12*}
		H_{12}^*- S_{12}^* =  U_2^* Y V_1 \Sigma_1^{-1}=U_2^* YX^\pls U_1,
	\end{equation}
	since $X^\pls=V_1 \Sigma_1^{-1}U_1^*$ and $X^\pls U_1=V_1 \Sigma_1^{-1}$. Similarly, $\D^* Z=W$ implies that $U^*\D^* U U^* Z=U^*W$ and we have 
	\begin{eqnarray*}
 \wh \D^* \mat{c}  U_1^* \\  U_2^* \rix Z = \mat{c}  U_1^* W \\  U_2^* W \rix \Longrightarrow
		\mat{cc} H_{11}-S_{11} & H_{12}-S_{12}\\ H_{12}^*+S_{12}^* & H_{22}-S_{22}\rix \mat{c}\wt \Sigma_1 \wt V_1^* \\ 0 \rix=\mat{c} U_1^* W \\  U_2^* W \rix,
	\end{eqnarray*}
since $\wt U_1=U_1$ and $Z=U_1 \wt \Sigma_1 \wt V_1^*$. This implies that
\begin{equation*}
	\left(H_{11}-S_{11} \right) \wt \Sigma_1 \wt V_1^*=  U_1^* W \quad \text{and}\quad 
		\left(H_{12}^*+S_{12}^* \right) \wt \Sigma_1 \wt V_1^*= U_2^* W.
	\end{equation*}
Thus, we have 
\begin{equation}\label{h11-s11}
	H_{11}-S_{11}=U_1^* W \wt V_1 \wt \Sigma_1^{-1}=U_1^* WZ^\pls U_1
\end{equation}
and
\begin{equation}\label{h12*+s12*}
H_{12}^* + S_{12}^*= U_2^* W \wt V_1 \wt \Sigma_1^{-1}= U_2^* WZ^\pls U_1,
\end{equation}
since $Z^\pls=\wt V_1 \wt \Sigma_1^{-1}  U_1^*$ and $Z^\pls  U_1= \wt V_1  \wt \Sigma_1^{-1}$.
Thus from \eqref{h11+s11} and \eqref{h11-s11},
	\begin{eqnarray}\label{eq:h11form1}
		H_{11}= U_1^*\left(\frac{YX^\pls+WZ^\pls}{2}\right)U_1=U_1^*\left(\frac{YX^\pls+(YX^\pls)^*}{2}\right)U_1 
\end{eqnarray} 
and
\begin{eqnarray}\label{eq:h11form2}
		S_{11}= U_1^*\left(\frac{YX^\pls-WZ^\pls}{2}\right)U_1 = U_1^*\left(\frac{YX^\pls-(YX^\pls)^*}{2}\right)U_1,
	\end{eqnarray} 
where in~\eqref{eq:h11form1} and~\eqref{eq:h11form2}, we have used Lemma~\ref{lem: mohit}.
Similarly, from~\eqref{h12*-s12*} and~\eqref{h12*+s12*}
	\begin{equation}\label{eq:h11form}
		H_{12}^*= U_2^*\left(\frac{YX^\pls+WZ^\pls}{2}\right)U_1 \quad \text{and} \quad
		S_{12}^*= U_2^*\left(\frac{WZ^\pls - YX^\pls}{2}\right)U_1.
	\end{equation}
Note that since $X^*Y+Y^*X \succeq 0$, in view of Lemma~\ref{lem8} , we have that $H_{11} \succeq 0$. Therefore
\begin{equation}
		\wh \D = \mat{cc} U_1^* YX^\pls U_1 & U_1^*(WZ^\pls)^* U_2 \\  U_2^* YX^\pls U_1 & H_{22}+ S_{22} \rix, 
	\end{equation}
where $H_{22},S_{22} \in \C^{n-r,n-r}$ are such that $\wh \D_H \succeq 0$ and $\wh \D_S^*=-\wh \D_S$.
	That means, $S_{22}$ satisfies that $S_{22}^*=-S_{22}$, and in view of Lemma~\ref{lem:psd} $H_{22}$ satisfies the constraints  $H_{22}\succeq 0$ and $H_{22}- J \succeq 0$ with $J= \frac{1}{2} U_2^*(YX^\pls + WZ^\pls){\left(YX^\pls + (YX^\pls)^* \right)}^{\pls}(YX^\pls + WZ^\pls)^*U_2$.
Thus from~\eqref{equation: delta N wt delta}, we have
\begin{eqnarray}\label{eq:const}
		\D &=& \mat{cc}U_1 & U_2 \rix \mat{cc} U_1^* YX^\pls U_1 & U_1^*(WZ^\pls)^* U_2 \\  U_2^* YX^\pls U_1 & H_{22} + S_{22} \rix
		\mat{c}U_1^* \\ U_2^*\rix.
	\end{eqnarray}
By setting $K=U_2 H_{22} U_2^*$ and $G=U_2 S_{22} U_2^*$ in~\eqref{eq:const}, we obtain that 
\begin{equation}\label{eq:dandhatd}
		\D = U \mat{cc} U_1^* YX^\pls U_1 & U_1^*(WZ^\pls)^*U_2 \\ U_2^*  YX^\pls U_1 & U_2^* (K+G) U_2 \rix U^*,
	\end{equation}
where $G$ and $K$ satisfy the conditions~\eqref{eq:cond1 n}. 	This proves $``\subseteq"$ in~\eqref{eq:S_char n}.
	
For the other side inclusion in~\eqref{eq:S_char n}, let $A$ be any matrix of the form
\begin{equation*}
		A=U \mat{cc} U_1^* YX^\pls U_1 & U_1^*(WZ^\pls)^* U_2 \\ U_2^* YX^\pls U_1 & U_2^* (K+G) U_2 \rix U^*,
\end{equation*}
where $K$ and $G$ satisfy the conditions~\eqref{eq:cond1 n}. Then 
using the fact that $U_1 U_1^*+ U_2 U_2^*=U U^*=I_n$, $U_1 U_1^*=XX^\pls=ZZ^\pls$, $U_1^* U_1=I_r$, and $U_2^* U_2=I_{n-r}$, $A$ can be written as 
\begin{eqnarray}
A=YX^\pls + (WZ^\pls)^* -{(WZ^\pls)}^* XX^\pls + \mathcal P_Z (K+G) \mathcal P_X.
\end{eqnarray}
Clearly  $A$ satisfies that $AX=Y$ and $A^*Z=W$, since $YX^\pls X=Y$, $WZ^\pls Z=W$, $\mathcal P_X X=0$ and $\mathcal P_Z Z=0$. Further, in view of~\eqref{eq:cond1 n} and Lemma~\ref{lem:psd} we have that
\[
U^*(A+A^*)U=\mat{cc} U_1^* \left(YX^\pls+(YX^\pls)^* \right) U_1 & U_1^* \left(YX^\pls+WZ^\pls \right)^* U_2\\
	U_2^* \left(YX^\pls+WZ^\pls \right) U_1 & U_2^* (2 K) U_2 
	\rix \succeq 0,
	\]
since $ U_1^* \left(YX^\pls+(YX^\pls)^* \right) U_1 \succeq 0$ from Lemma~\ref{lem8} as $X^*Y+Y^*X \succeq 0$,  by assumption $\ker(U_1^*(YX^\pls + (YX^\pls)^*)U_1) \subseteq \ker(U_2^*(YX^\pls + WZ^\pls)U_1)$, and $U_2^*KU_2-J \succeq 0$, since $K$ satisfies that $K-U_2JU_2^* \succeq 0$.
This implies that $A+A^* \succeq 0$ and hence $A \in \mathcal S_{d_1}^{Diss}$. This proves $``\supseteq"$ in~\eqref{eq:S_char n}.

Suppose that $\mathcal S_{d_1}^{Diss}\neq  \emptyset$ and  let $\D \in \mathcal S_{d_1}^{Diss}$, then from~\eqref{eq:S_char n} we have that
\begin{eqnarray}
{\|\D\|}_F^2 &=&  \left \|\mat{cc}U_1^* YX^\pls U_1 & U_1^*(WZ^\pls)^*U_2\\ U_2^*YX^\pls U_1 & U_2^*(K+G)U_2  \rix \right \| _F^2 \nonumber\\ 
&=& \left \|  U^* YX^\pls U_1  \right \|_F^2 + {\|U_1^*(WZ^\pls)^*U_2\|}_F^2 + {\| U_2^*(K+G)U_2 \|}_F^2  \nonumber\\ 
&=& \left \|  YX^\pls U_1  \right \|_F^2 + {\|U_1^*(WZ^\pls)^*U_2\|}_F^2 + {\| U_2^*(K+G)U_2 \|}_F^2   \nonumber\\ 
		&=&  {\| YX^\pls \|}_F^2 + {\|U_1^*(WZ^\pls)^*U\|}_F^2 + {\| U_2^*(K+G)U_2 \|}_F^2  - {\| YX^\pls U_2  \|}_F^2 
		- {\|U_1^*(WZ^\pls)^* U_1\|}_F^2    \nonumber\\ 
		&=& {\| YX^\pls \|}_F^2 + {\|(WZ^\pls)^*\|}_F^2 + {\| U_2^*(K+G)U_2 \|}_F^2  - {\| YX^\pls U_2  \|}_F^2 
		- {\|U_1^*(WZ^\pls)^*U_1\|}_F^2 - {\|U_2^*(WZ^\pls)^*\|}_F^2  \nonumber\\ 
		&=& {\| YX^\pls \|}_F^2 + {\|(WZ^\pls)^*\|}_F^2 + {\| U_2^*(K+G)U_2 \|}_F^2   
		- {\|U_1^*(WZ^\pls)^*U_1\|}_F^2   \quad (\because X^\pls U_2=0, \,Z^\pls U_2=0)  \nonumber\\ 
	&=& {\| YX^\pls \|}_F^2 + {\|(WZ^\pls)^*\|}_F^2 + {\|U_2^*(K+G)U_2 \|}_F^2  - {\rm trace}(U_1^* WZ^\pls U_1 U_1^*(WZ^\pls)^*U_1)    \nonumber\\ 
		&=& {\| YX^\pls \|}_F^2 + {\|(WZ^\pls)^*\|}_F^2 + {\| U_2^*(K+G)U_2 \|}_F^2   
		- {\rm trace}(  WZ^\pls Z Z^\pls (WZ^\pls)^* X X^\pls) \quad (\because U_1 U_1^*= X X^\pls= Z Z^\pls )   \nonumber\\ 
		&=& {\| YX^\pls \|}_F^2 + {\|(WZ^\pls)^*\|}_F^2 + {\| U_2^*(K+G)U_2 \|}_F^2   
		- {\rm trace}(  WZ^\pls (WZ^\pls)^* X X^\pls)   \nonumber\\ 
		&=& {\| YX^\pls \|}_F^2 + {\|(WZ^\pls)^*\|}_F^2 + {\| U_2^*KU_2\|}_F^2 + {\|U_2^* G U_2 \|}_F^2   
		- {\rm trace}(  WZ^\pls (WZ^\pls)^* X X^\pls)  \nonumber
\end{eqnarray}
where the last equality follows as for any square matrix $A=A_H+A_S$ we have ${\|A\|}_F^2={\|A_H\|}_F^2+{\|A_S\|}_F^2$. This implies that 
\begin{eqnarray}
&&\inf_{\D \in \mathcal S_{d_1}^{Diss}}{\|\D\|}_F^2 \nonumber \\
&&=  \inf_{K,G \in \C^{n,n},G^*=-G, K-U_2JU_2^* \succeq 0}
 {\| YX^\pls \|}_F^2 + {\|(WZ^\pls)^*\|}_F^2 + {\| U_2^*KU_2\|}_F^2   + {\|U_2^* G U_2 \|}_F^2 
		- {\rm trace}(  WZ^\pls (WZ^\pls)^* X X^\pls) \nonumber \\  
		&&\geq \inf_{K \in \C^{n,n},K-U_2JU_2^* \succeq 0}
 {\| YX^\pls \|}_F^2 + {\|(WZ^\pls)^*\|}_F^2 + {\| U_2^*KU_2\|}_F^2   
		- {\rm trace}(  WZ^\pls (WZ^\pls)^* X X^\pls) \label{eq:type11} \\
		&& \geq  {\| YX^\pls \|}_F^2 + {\|(WZ^\pls)^*\|}_F^2 + {\| J\|}_F^2   - {\rm trace}(  WZ^\pls (WZ^\pls)^* X X^\pls), \label{eq:type12} \
\end{eqnarray}
where the first inequality in~\eqref{eq:type11} is obvious since for any $G \in \C^{n,n}$ ${\|U_2^* G U_2 \|}_F \geq 0$ and the second inequality is due to Lemma~\ref{lemma: ineq_psd}, since ${\|\cdot\|}_F$ is unitarily invariant and $J \succeq 0$ implies that for any $K \in \C^{n,n}$
 such that $K-U_2JU_2^*\succeq 0$ we have ${\|K\|}_F \geq {\|U_2JU_2^*\|}_F={\|J\|}_F$. Thus by setting $K=U_2JU_2^*$ and $G=0$, we obtain a unique matrix $\mathcal H$ that attains the lower bound in~\eqref{eq:type12}, i.e., 
	\begin{eqnarray*}
		\inf_{\D \in \mathcal S_{d_1}^{Diss}} {\|\D\|}_F^2= {\|\mathcal H\|}_F^2= {\|YX^\pls\|}_F^2 + {\|WZ^\pls\|}_F^2 - {\rm{trace}}(W Z^\pls(W Z^\pls)^*XX^\pls) + {\|J\|}_F^2. 
	\end{eqnarray*}
This competes the proof. 
\end{proof}

The vector case Type-1 DSDMs, i.e., when $m=1$ in Theorem~\ref{map result: NSPSD 4vec}, is particularly interesting. This is because  (i) the conditions on the free matrices in~\eqref{eq:S_char n} are more simplified,  and (ii) it will be used in computing the structured eigenpair backward errors for pencil $L(z)$ defined in~\eqref{eq:defpenciL}, see Section~\ref{sec:app_strbacerr}. For future reference, we state the vector case separately.

\begin{theorem}\label{map:type1veccase} 
	Let $x,y,w \in \C^{n}\setminus\{0\}$ and let  $z=\alpha x$ for some nonzero $\alpha \in \C$. Suppose that $\real{(x^*y)} \neq 0$. 
 Then $\mathcal S_{d_1}^{\rm Diss} \neq \emptyset$ if and only if  $x^*w=y^*z$ and  $\real{(x^*y)} > 0$.

 Moreover, if $\mathcal S_{d_1}^{Diss} \neq \emptyset$, then 
\begin{eqnarray}\label{eq:typevecchar1}
		\mathcal S_{d_1}^{\rm Diss}=\Big\{yx^\pls + (wz^\pls)^*-(wz^\pls)xx^\dagger + \mathcal P_x (K+G)\mathcal P_x:~ K,G \in \C^{n,n}~\text{satisfy}~\eqref{eq:cond1t1vec a} \Big\},
\end{eqnarray} 
		where
		\begin{eqnarray}\label{eq:cond1t1vec a}
G^*=-G,\quad K\succeq 0, \quad \text{and}\quad K-\mathcal P_x J \mathcal P_x \succeq 0,
		\end{eqnarray}
where $J= \frac{1}{4\real{(x^*y)}}\left(y+\frac{\bar \alpha}{{|\alpha|}^2}w\right)\left(y+\frac{\bar \alpha}{{|\alpha|}^2}w\right)^*$.
Further,
\[
\inf_{\Delta \in \mathcal S_{d_1}^{\rm Diss}} {\|\Delta\|}_F^2 \; = \; \frac{{\|y\|}^2}{{\|x\|}^2}-\frac{{\|w\|}^2}{{\|z\|}^2}-\frac{|w^*x|^2}{{\|x\|}^2{\|z\|}^2}+{\|J\|}_F^2,
\]
where the infimum is uniquely attained by  $\mathcal H:=yx^\pls +{(wz^\pls)}^*  \mathcal P_x + \mathcal P_x J \mathcal P_x$,
		which is obtained by setting $K=\mathcal P_x J \mathcal P_x$ and  $G=0$ in~\eqref{eq:typevecchar1}.
\end{theorem}

\begin{remark}{\rm \label{rem:type-1dsdm}
If we consider the DSM Problem~\ref{def:dsmprob} with $\D_1+\D_1 \preceq 0$, then the corresponding result follows from Theorem~\ref{map result: NSPSD 4vec} by replacing everywhere the positive semidefinite constraint ``$\succeq$" with the negative semidefinite constraint ``$\preceq$".
}
\end{remark}

\subsubsection{Type-2 doubly structured dissipative mappings}

In this section, we consider the Type-2 DSDM problem, i.e.  when $\mathbb S={\rm Diss}(n)$ in Problem~\ref{def:dsmprob}. We have the following result that completely solves the existence and characterization problem of the Type-2 DSDM problem and derives sufficient conditions for computing the minimal norm  solution to the Type-2 DSDM problem.

\begin{theorem} \label{thm:Type-2DSDM}
	Given $x=[x_1^T ~\, x_2^T ]^T$ with $x_1 \in \C ^{n}$ and $x_2 \in \C^m \setminus \{0\}$, $y,z \in \C^n$, and $w=[w_1^T ~\,  w_2^T ]^T$ with $w_1 \in \C ^{n}\setminus\{0\}$ and  $w_2 \in \C^m \setminus \{0\}$. Define
\begin{equation} \label{def: double struct mapping set NSPSD}
		\lS_{d_2}^{\rm Diss} := \left \{ \D~:~ [\D_1 ~\, \D_2],\, \D_1 \in \C^{n,n},\,\D_2 \in \C^{n,m},\, \D_1 + \D_1^* \succeq 0,\, \D x=y,\, \D^*z=w \right \}.
	\end{equation}
	Then $\lS_{d_2}^{\rm Diss} \neq \emptyset $ if and only if $x^*w=y^*z$ and $\real{(z^*w_1)} \geq 0$. If $x^*w=y^*z$ and $\real{(z^*w_1)} > 0$, then 
\begin{equation}\label{eq:chardiss2}
\lS_{d_2}^{\rm Diss} = \left \{ H + \wt H(Z,K,G,R)  :\,R \in \C^{n,m},\,K,G,Z \in \C^{n,n}~\text{satisfy}~\eqref{1cond vec}\right \},
\end{equation}
where 
\begin{eqnarray}
G^*=-G,\quad  K \succeq 0,\quad \text{and}\quad  K- \frac{1}{4\real{(z^*w_1)}} 
\left(2w_1  + Z^*z\right) \left(2w_1  + Z^*z\right) ^* \succeq 0,  \label{1cond vec}
\end{eqnarray}  
$H=[H_1 ~\, H_2 ]$ and $\wt H(Z,K,G,R) =[\wt H_1(Z,K,G) ~\,\wt H_2(Z,K,G,R) ]$ with $H_1$, $H_2$, $\wt H_1(Z,K,G)$, $\wt H_2(Z,K,G,R)$ given by 
\begin{eqnarray}
H_1 &=& (w_1 z^\pls)^* + \mathcal P_z w_1 z^\pls \label{val: H_1 double struct NSPSD} \\
H_2 &=& y x_2^\pls -  (w_1 z^\pls)^*x_1  x_2^\pls - \mathcal P_z w_1 z^\pls x_1  x_2^\pls + (w_2 z^\pls)^* - (w_2 z^\pls)^* x_2 x_2^\pls  \label{val: H_2 double struct NSPSD} \\
\wt H_1(Z,K,G)&=& \mathcal P_z Z^* z z^\pls + \mathcal P_z K \mathcal P_z - \mathcal P_zG \mathcal P_z \label{val: wt H_1 double struct NSPSD} \\
\wt H_2(Z,K,G,R) &=& -\mathcal P_z Z^* z z^\pls x_1 x_2^\pls - \mathcal P_z K \mathcal P_z x_1 x_2^\pls + \mathcal P_z G \mathcal P_z x_1 x_2^\pls + \mathcal P_z R \mathcal P_{x_2}  \label{val: wt H_2 double struct NSPSD},
\end{eqnarray}	
and 
\begin{equation}\label{eq:lbounddiss2}
\inf_{\Delta \in \lS_{d_2}^{\rm Diss} }
{\|\Delta\|}_F^2 \geq {\|\hat H_1\|}_F^2+ \inf_{Z,K,G \, \text{satisfying}~\eqref{1cond vec}}{\|H_2+\wt H_2(Z,K,G,0)\|}_F^2,
\end{equation}
 where 
 \begin{equation}\label{eq:diss2hatmaps}
 \hat{H}_1 := H_1 -  2 P_z w_1 z^\pls.  
 \end{equation}
 Moreover, 
if  $y= \beta z$ for some $ \beta \in \C$ and $z$ is orthogonal to $x_1$, then 
	\begin{equation}\label{eq:extremal_diss2}
\inf_{\Delta \in \lS_{d_2}^{\rm Diss} }
{\|\Delta\|}_F^2= {\| \hat{H}_1 \|}_F^2 + {\| \hat{H}_2 \|}_F^2,
		\end{equation}
where $\hat H_1$ is defined by~\eqref{eq:diss2hatmaps} and $\hat{H}_2 := y x_2^\pls -  (w_1 z^\pls)^*x_1  x_2^\pls + (w_2 z^\pls)^* - (w_2 z^\pls)^* x_2 x_2^\pls $, and the infimum in~\eqref{eq:extremal_diss2} is uniquely attained by the matrix $\hat{H}= [ \hat{H}_1 ~\, \hat{H}_2]$.	  
\end{theorem} 
\proof
	First suppose that  $\lS_{d_2}^{\rm Diss}\neq \emptyset$ and let $\Delta \in \lS_{d_2}^{\rm Diss}$. Then $\D = [\D_1 ~\,\D_2]$ with $\Delta_1 \in \C^{n,n}$, $\Delta_2 \in \C^{n,m}$ such that $\D_1 + \D_1^* \succeq 0$, $\D x=y$, and $\D^* z=w$. This implies that $y^*z= (\D x)^* z = x^* \D^* z=x^* w$. Since $\Delta^*z=w$, we have $\D_1^*z=w_1$ which implies that 
 \[2\real{(z^*w_1)} =z^* w_1 + w_1^*z = z^* \D_1^* z + z^* \D_1 z= z^*(\D_1 + \D_1^*)z \geq 0,
 \]
  since $ \D_1+\D_1^* \succeq 0 $. 
Conversely, let $x^*w=y^*z$ and $\real{(z^*w_1)}  \geq 0$. Then it is easy to check that $\hat{H}=[\hat{H}_1 ~\, \hat{H}_2 ] $ with $\hat{H}_1 := H_1 -  2 P_z w_1 z^\pls$ and $\hat{H}_2 := H_2 + 2 P_z w_1 z^\pls x_1  x_2^\pls $ satisfies $\hat{H} x=y$ and $\hat{H}^*z=w$. Further,
\begin{eqnarray*}
\hat{H}_1 + \hat{H}_1^*&=&(w_1 z^\pls)^* - \mathcal P_z w_1 z^\pls + (w_1 z^\pls) -  (w_1 z^\pls)^*\mathcal P_z= (z^*w_1+w_1^*z)zz^\dagger \succeq 0,
\end{eqnarray*}
since $\real{(z^*w_1)} \geq 0$ and $zz^\dagger \succeq 0$ being a rank one symmetric matrix. This implies that $\hat H \in \lS_{d_2}^{\rm Diss}$.

Next we prove \eqref{eq:chardiss2}. For this, let $\D \in \lS_{d_2}^{\rm Diss}$, i.e., $\D=[\D_1 ~\, \D_2]$ such that $\D x = y$, $\D^* z = w$, and $\D_1 + \D_1^* \succeq 0$. This implies that 
\begin{equation}\label{eq:equimaps}
\D_1x_1+\D_2x_2=y,\quad \D_1^*z=w_1,\quad \text{and}\quad \D_2^*z=w_2.
\end{equation}
Since $\D_1+\D_1^* \succeq 0$ with $\D_1^*$ taking $z$ to $w_1$, from Theorem~\ref{map result : MAIN BhaGS21a} 
\begin{equation*}
		\D_1^* = w_1 z^\pls + (w_1 z^\pls)^* \mathcal P_z + z z^\pls Z \mathcal P_z + \mathcal P_z K\mathcal P_z + \mathcal P_z G \mathcal P_z,
	\end{equation*}
or, equivalently,
\begin{equation}\label{val:del_1diss2}
\D_1 = (w_1 z^\pls)^* + \mathcal P_z w_1 z^\pls + \mathcal P_z Z^* z z^\pls + \mathcal P_z K \mathcal P_z - \mathcal P_z G \mathcal P_z,
\end{equation}
for some $Z,K,G \in \C^{n,n}$ such that 
\[
G^*=-G,\quad   K \succeq 0,\quad \text{and} \quad  K- \frac{1}{4\real{(z^*w_1)}} 
\left(2w_1  + Z^*z\right) \left(2w_1  + Z^*z\right)^* \succeq 0.
\]
By substituting $\D_1$ from~\eqref{val:del_1diss2} in~\eqref{eq:equimaps}, we get
\begin{eqnarray}\label{eq: delta2 stage 2 NSPSD}
\D_2 x_2 = \tilde y   \quad \text{and}\quad  \D_2 ^* z = w_2,
\end{eqnarray} 
where $\tilde y=y-\D_1x_1=
y-  \left( (w_1 z^\pls)^* + P_z w_1 z^\pls + P_z Z^* z z^\pls + P_z K P_z - P_z G P_z  \right)x_1$. Again since $\tilde y^*z=x_2^*w_2$, in view of Theorem~\ref{theorem : 2.1 mehl2017} $\D_2$ has the form
\begin{eqnarray} \label{val:del_2diss2}
\D_2 = \tilde y x_2^\pls  +  (w_2 z^\pls)^* - (w_2 z^\pls)^* x_2 x_2^\pls + \mathcal P_z R \mathcal P_{x_2},
\end{eqnarray}
for some $R \in \C^{n,m}$. In view of~\eqref{val:del_1diss2} and~\eqref{val:del_2diss2}, we have
\begin{eqnarray}
[\D_1~\, \D_2]=[H_1~\,H_2]+[\wt H_1(Z,K,G)~\,\wt H_2(Z,K,G,R)]=H+\wt H(Z,K,G,R).
\end{eqnarray}
This proves ``$\subseteq$" in~\eqref{eq:chardiss2}.

Conversely, let $x^*w=y^*z$ and $\real{(z^* w_1)} \geq 0 $, consider $[\D_1~\,\D_2]=\big[H_1+\wt H_1(Z,K,G)~~H_2+\wt H_2(Z,K,G,R)\big]$, where 
$H_1$, $H_2$, $\wt H_1(Z,K,G)$, and $\wt H_2(Z,K,G,R)$ be 
given by~\eqref{val: H_1 double struct NSPSD}--\eqref{val: wt H_2 double struct NSPSD} 
for some matrices $R \in \C^{n,m}$ and $Z,K,G \in \C^{n,n}$ satisfying~\eqref{1cond vec}. Then a straight-forward calculation shows that $[\D_1~\,\D_2]x=y$ and $[\D_1~\,\D_2]^*z=w$. Also $\D_1+\D_1^* \succeq 0$. To see this, let $u_1=\frac{z}{\|z\|}$ and $U_2 \in \C^{n,n-1}$ be such that 
$U=[u_1~\,U_2]$ becomes unitary. This implies that 
\begin{eqnarray}
\D_1^*&=&(H_1+\wt H_1(Z,K,G))^*=w_1 z^\pls + (w_1 z^\pls)^* \mathcal P_z + z z^\pls Z \mathcal P_z + \mathcal P_z K\mathcal P_z + \mathcal P_z G \mathcal P_z \nonumber\\
&=& w_1 z^\pls + (w_1 z^\pls)^* U_2U_2^* + z z^\pls Z U_2U_2^*  + U_2U_2^*  KU_2U_2^*  + U_2U_2^*  G U_2U_2^* \nonumber \\
&=& (u_1u_1^*+U_2U_2^*)(w_1z^\dagger)u_1u_1^*+ u_1u_1^*(w_1 z^\pls)^* U_2U_2^* + u_1u_1^* Z U_2U_2^*  + U_2U_2^*  (K+ G) U_2U_2^* \nonumber \\
&=& U\mat{cc}u_1^*(w_1z^\pls)u_1 & u_1^*(w_1z^\pls)^*U_2 + u_1^*ZU_2 \\ U_2^*(w_1z^\pls)u_1 & U_2^*(K+G)U_2 
\rix U^* \nonumber \\
&=& U\mat{cc}z^\pls w_1 & \frac{w_1^*U_2}{\|z\|} + \frac{z^*ZU_2}{\|z\|} \\ \frac{U_2^*w_1}{\|z\|}  & U_2^*(K+G)U_2 
\rix U^*, \label{eq:proveD1nspsd1}
\end{eqnarray}
where we have used the fact that $UU^*=I_n$ and $u_1=\frac{z}{\|z\|}$.  Thus in view of Lemma~\ref{lem:psd} and~\eqref{eq:proveD1nspsd1}, we have that $\D_1+\D_1^* \succeq 0$, since $Z,K,G$ satisfy~\eqref{1cond vec}. This proves ``$\supseteq$" in~\eqref{eq:chardiss2}.

In view of~\eqref{eq:chardiss2}, we have
\begin{eqnarray}
\inf_{\D \in \lS_{d_2}^{\rm Diss}} {\|\D\|}_F^2&=& \inf_{R\in \C^{n,m}, K,G,Z \in \C^{n,n}~\text{satisfying}~\eqref{1cond vec}} {\|H + \wt H(Z,K,G,R)\|}_F^2\nonumber \\
&=& \inf_{R\in \C^{n,m}, K,G,Z \in \C^{n,n}~\text{satisfying}~\eqref{1cond vec}} \left({\left \|\big[H_1~\,H_2\big]+\big[\wt H_1(Z,K,G)~\,\wt H_2(Z,K,G,R)\big]\right\|}_F^2\right) \nonumber \\
&=& \inf_{R\in \C^{n,m}, K,G,Z \in \C^{n,n}~\text{satisfying}~\eqref{1cond vec}} \left({\left \|H_1+\wt H_1(Z,K,G)\right\|}_F^2 +{\left\|H_2+,\wt H_2(Z,K,G,R)\right\|}_F^2\right) \nonumber \\
&\geq & \inf_{ K,G,Z \in \C^{n,n}~\text{satisfying}~\eqref{1cond vec}} {\left \|H_1+\wt H_1(Z,K,G)\right\|}_F^2 \nonumber \\
&+&\inf_{R\in \C^{n,m}, K,G,Z \in \C^{n,n}~\text{satisfying}~\eqref{1cond vec}}{\left\|H_2+,\wt H_2(Z,K,G,R)\right\|}_F^2 
\label{eq:proofdissnorm1}\\
&= &  {\left \|H_1+\wt H_1(-2(w_1z^\dagger)^*,0,0)\right\|}_F^2 +\inf_{R\in \C^{n,m}, K,G,Z \in \C^{n,n}~\text{satisfying}~\eqref{1cond vec}}{\left\|H_2+,\wt H_2(Z,K,G,R)\right\|}_F^2\nonumber \\ \label{eq:proofdissnorm2}\\
&= &  {\left \|\hat H_1\right\|}_F^2 +\inf_{ K,G,Z \in \C^{n,n}~\text{satisfying}~\eqref{1cond vec}}\left(\inf_{R\in \C^{n,m}}{\left\|H_2+,\wt H_2(Z,K,G,R)\right\|}_F^2\right)\nonumber \\
&= &  {\left \|\hat H_1\right\|}_F^2 +\inf_{ K,G,Z \in \C^{n,n}~\text{satisfying}~\eqref{1cond vec}}{\left\|H_2+\wt H_2(Z,K,G,0)\right\|}_F^2, \label{eq:proofdissnorm3}
\end{eqnarray}
where the first inequality in~\eqref{eq:proofdissnorm1} follows due to the fact that for any two real valued functions $f$ and $g$ defined on the same domain, $\inf(f+g)\geq \inf f+\inf g$. Also equality in~\eqref{eq:proofdissnorm2} follows since the infimum in the first term is attained when $K=0$, $G=0$, and $Z=-2(w_1z^\dagger)^*$. In fact, for any $K,G,Z \in \C^{n,n}$ satisfying~\eqref{1cond vec}, the matrix $H_1+\wt H_1(Z,K,G)$ is a dissipative map taking $z$ to $w_1$, which implies from Theorem~\ref{map result : MAIN BhaGS21a} that the minimum of ${\|H_1+\wt H_1(Z,K,G)\|}_F$ is attained when $K=0$, $G=0$, and $Z=-2(w_1z^\dagger)^*$, i.e., for $\hat H_1=H_1+\wt H_1(-2(w_1z^\dagger)^*,0,0)=H_1 -  2 P_z w_1 z^\pls$ defined by~\eqref{eq:diss2hatmaps}. Similarly, for a fixed 
$K,G,Z \in \C^{n,n}$ satisfying~\eqref{1cond vec} and for any $R \in \C^{n,,m}$, the matrix $H_2+\wt H_2(Z,K,G,R)$ satisfies that $(H_2+\wt H_2(Z,K,G,R))x_2=\tilde y$ and $(H_2+\wt H_2(Z,K,G,R))^*z=w_2$. This implies from Theorem~\ref{theorem : 2.1 mehl2017} that for any fixed $Z,K,G$, the minimum of ${\|H_2+\wt H_2(Z,K,G,R)\|}_F$ over $R$ is attained when $R=0$. This justifies~\eqref{eq:proofdissnorm3} and hence proves~\eqref{eq:lbounddiss2}. 

Next we prove~\eqref{eq:extremal_diss2} under the assumption that $y=\beta z$ for some $\beta \in \C$ and $z$ is orthogonal to $x_1$. For this, let us first estimate the infimum in the right hand side of~\eqref{eq:proofdissnorm3}. For any $Z,K,G \in \C^{n,n}$ satisfying~\eqref{1cond vec}, we have 
\begin{eqnarray}
&&{\|H_2 + \wt H_2 (Z,K,G,0)\|}_F^2 \nonumber \\
	&&= {\|H_2 \|}_F^2 + {\| \wt H_2(Z,K,G,0)  \|}_F^2  +2\real\left (\trace\big((H_2 )(\wt H_2(Z,K,G,0) )^*\big)\right) \nonumber \\
	&&= {\|\hat H_2 \|}_F^2 + {\| \wt H_2(Z,K,G,0)  \| }_F^2  \nonumber \\
	&&-2\real\left(\trace\big((y x_2^\pls - (w_1 z^\pls)^* x_1 x_2^\pls + (w_2 z^\pls)^*P_{x_2} ) (x_1 x_2^\pls)^*( P_z K + P_z G )P_z\big)\right)\label{eq:proofdissnorm11}\\
	&&= {\|\hat H_2  \|}_F^2 + {\| \wt H_2(Z,K,G,0) \| }_F^2  \nonumber \\
	&&-2\real\left(\trace\big(P_z(y x_2^\pls - (w_1 z^\pls)^* x_1 x_2^\pls + (w_2 z^\pls)^*P_{x_2}) (x_1 x_2^\pls)^*( P_z K + P_z G) \big)\right) \nonumber \\
	&&= {\|\hat H_2 \|}_F^2 + {\| \wt H_2(Z,K,G,0)  \| }_F^2  \quad (\because P_z y=0,~\text{and}~P_z (z^\dagger)^*=0) \nonumber\\
	&&\geq {\|\hat H_2 \|}_F^2, \nonumber
\end{eqnarray}
where the equality in~\eqref{eq:proofdissnorm11} follows from~\eqref{val: wt H_2 double struct NSPSD} and~\eqref{val:  H_2 double struct NSPSD} using the fact that $z$ is orthogonal to $x_1$ and by setting $\hat{H}_2 := y x_2^\pls -  (w_1 z^\pls)^*x_1  x_2^\pls + (w_2 z^\pls)^* \mathcal P_{x_2} $.
This implies that 
\begin{equation}\label{eq:proofdissnorm22}
\inf_{ K,G,Z \in \C^{n,n}~\text{satisfying}~\eqref{1cond vec}}{\left\|H_2+\wt H_2(Z,K,G,0)\right\|}_F^2 \geq {\|\hat{H}_2\|}_F^2
\end{equation}
In view of~\eqref{eq:proofdissnorm3} and~\eqref{eq:proofdissnorm22}, when $y=\beta z$ for some $\beta \in \C$ and $z\perp x_1$, we have that 
\begin{equation}\label{eq:proofdissnorm33}
\inf_{\D \in \lS_{d_2}^{\rm Diss}} {\|\D\|}_F^2 ~\geq~ {\|\hat{H}_1\|}_F^2+{\|\hat{H}_2\|}_F^2.
\end{equation}
Notice that the lower bound in~\eqref{eq:proofdissnorm33} is uniquely attained for $\D=[\hat H_1~\,\hat H_2]$ which is obtained by taking $Z=-2(w_1 z^\pls)^*$, $K=0$, $G=0$, and $R=0$ in~\eqref{eq:chardiss2}. This completes the proof. 	
\eproof

\begin{remark}{\rm \label{rem:type-2dsdm}
A remark similar to Remark~\ref{rem:type-1dsdm} also holds for Type-2 DSDMs. 
}
\end{remark}

\section{DSM's in computing structured eigenpair backward errors of pencils $L(z)$ }\label{sec:app_strbacerr}

Consider the pencil $L(z)$ in the form~\eqref{eq:defpenciL} that arises in passivity analysis of port-Hamiltonian systems. In this section, we exploit the minimal-norm DSMs from Section~\ref{sec:DSDMs} to develop eigenpair backward error estimates under block- and symmetry-structure-preserving perturbations. These results extend the work done in~\cite{MehMS18}, where the pencil $L(z)$ was considered without semidefinite structure on the block $R$. 
Let us introduce the perturbation $\Delta_M+z \Delta_N$ of the pencil $L(z)=M+zN$, where 
\begin{equation}\label{eq:pertumn}
\Delta_M = \mat{ccc} 0 & \Delta_J -\Delta_R & \Delta_B \\ \Delta_J^*-\Delta_R^* & 0 & 0 \\ \Delta^*_B & 0 & 0 \rix \quad \text{and}\quad 
	\D_N = \mat{ccc} 0 & \D_E & 0 \\ -\D_E^* & 0 & 0 \\ 0 & 0 & 0 \rix,
\end{equation}	
for $\Delta_J,\Delta_R,\Delta_E \in \C^{n,n}$ and $\Delta_B \in \C^{n,m}$, that affect the blocks $J,R,E,B$ of $L(z)$. Thus motivated by~\cite{MehMS18}, we define various structured eigenpair backward errors of $L(z)$ for a given pair $(\lambda,u) \in \C \times \C^{2n+m}\setminus \{0\}$ as follows:
\begin{enumerate}
\item \emph{the block-structure-preserving} eigenpair backward error $\eta^{\mathcal B}(J,R,E,B,\lambda,u)$ of $L(z)$ with respect to  perturbations from the set
\begin{eqnarray}\label{eq:persetb}
\mathcal B(J,R,E,B):=&\big\{ \Delta_M+z\Delta_N:~\Delta_M,\Delta_N~ \text{defined\, by}~\eqref{eq:pertumn}~\text{for} \nonumber\\
&~\Delta_J,\Delta_R,\Delta_E \in \C^{n,n},~\Delta_B \in \C^{n,m}
\big\}
\end{eqnarray}
is defined by
\begin{eqnarray}\label{def:errblc}
\eta^{\mathcal B}(J,R,E,B,\lambda,u)=&\big \{ {\|\left[\Delta_M~\Delta_N\right]\|}_F:~ \left((M-\Delta_M)+\lambda (N-\Delta_N)\right)u=0, \nonumber \\
&\Delta_M+z\Delta_N \in \mathcal B(J,R,E,B)
\big \};
\end{eqnarray}
\item  \emph{the symmetry-structure-preserving} eigenpair backward error $\eta^{\mathcal S}(J,R,E,B,\lambda,u)$ of $L(z)$ with respect to  perturbations from the set
\begin{eqnarray}\label{eq:persets}
\mathcal S(J,R,E,B):=&\big\{ \Delta_M+z\Delta_N:~\Delta_M,\Delta_N~ \text{defined\, by}~\eqref{eq:pertumn}~\text{for}~\Delta_J,\Delta_R,\Delta_E \in \C^{n,n}, \nonumber\\
&~\Delta_B \in \C^{n,m},~\Delta_J^*=-\Delta_J, \Delta_R^*=\Delta_R,\Delta_E^*=\Delta_E
\big\}
\end{eqnarray}
is defined by
\begin{eqnarray}\label{def:errstr}
\eta^{\mathcal S}(J,R,E,B,\lambda,u)=&\big \{ {\|\left[\Delta_M~\Delta_N\right]\|}_F:~ \left((M-\Delta_M)+\lambda (N-\Delta_N)\right)u=0, \nonumber \\
&\Delta_M+z\Delta_N \in \mathcal S(J,R,E,B)
\big \};
\end{eqnarray}
\item  \emph{the semidefinite-structure-preserving} eigenpair backward error $\eta^{\mathcal S_d}(J,R,E,B,\lambda,u)$ of $L(z)$ with respect to  perturbations from the set
\begin{eqnarray}\label{eq:persetsd}
\mathcal S_d(J,R,E,B):=&\big\{ \Delta_M+z\Delta_N:~\Delta_M,\Delta_N~ \text{defined\, by}~\eqref{eq:pertumn}~\text{for}~\Delta_J,\Delta_R,\Delta_E \in \C^{n,n}, \nonumber\\
&~\Delta_B \in \C^{n,m},~\Delta_J^*=-\Delta_J, \Delta_R\succeq 0,\Delta_E^*=\Delta_E
\big\}
\end{eqnarray}
is defined by
\begin{eqnarray}\label{def:errpsd}
\eta^{\mathcal S_d}(J,R,E,B,\lambda,u)=&\big \{ {\|\left[\Delta_M~\Delta_N\right]\|}_F:~ \left((M-\Delta_M)+\lambda (N-\Delta_N)\right)u=0, \nonumber \\
&\Delta_M+z\Delta_N \in \mathcal S_d(J,R,E,B)
\big \}.
\end{eqnarray}
\end{enumerate}

We note that by choosing different perturbation sets in~\eqref{eq:persetb},~\eqref{eq:persets}, and~\eqref{eq:persetsd}, the corresponding backward errors can be defined by allowing perturbations only to specific blocks $J,R,E,B$ of $L(z)$. For example, if we chose 
$\mathcal B(J,R):=\mathcal B(J,R,0,0)$, where $\Delta_E=0$ and $\Delta_B=0$ in~\eqref{eq:persetb} to allow perturbation only in blocks $J$ and $R$ of $L(z)$, then the corresponding backward error is given by $\eta^{\mathcal B}(J,R,\lambda,u):=\eta^{\mathcal B}(J,R,0,0,\lambda,u)$. Similarly, the backward errors $\eta^{\mathcal S}(J,R,\lambda,u)$ and $\eta^{\mathcal S_d}(J,R,\lambda,u)$ can be defined by restricting the perturbation sets as $\mathcal S(J,R):=\mathcal S(J,R,0,0)$ and $\mathcal S_d(J,R):=\mathcal S_d(J,R,0,0)$, where $\Delta_E=0$ and $\Delta_B=0$ in~\eqref{eq:persets} and~\eqref{eq:persetsd}, respectively. 

The block- and symmetry-structure-preserving eigenpair backward errors 
$\eta^{\mathcal B}(\cdot,\cdot,\cdot,\cdot,\lambda,u)$ and  $\eta^{\mathcal S}(\cdot,\cdot,\cdot,\cdot,\lambda,u)$ were studied in~\cite{MehMS18} for different combinations of perturbation blocks $J,R,E$, and $B$ of $L(z)$. In the following, we obtain results only for the semidefinite-structure-preserving backward error $\eta^{\mathcal S_d}(J,R,E,B,\lambda,u)$. The backward errors for other combination of the blocks $J,R,E,B$ of $L(z)$ can be obtained analogously, see~\ref{sec:appe2}.

\begin{remark}\label{rem:appdsw}{\rm
Let $L(z)$ be a pencil in the form~\eqref{eq:defpenciL}, $\lambda \in i\R$, and 
$u=\big[u_1^T~u_2^T~u_3^T\big]^T$ be such that $u_1,u_2 \in \C^n$, and $u_3 \in \C^m$. Then for any $\Delta L(z)=\Delta_M + z \Delta_N$, where 
$\Delta_M$ and $\Delta_N$ are defined by~\eqref{eq:pertumn} for $\Delta_J,\Delta_R,\Delta_E \in \C^{n,n}$ and $\Delta_B \in \C^{n,m}$, we have
$\left(L(\lambda)-\Delta L (\lambda)\right)u=0$ if and only if 
\begin{eqnarray}
(\D_J - \D_R + \lm \D_E) u_2 + \D_B u_3	 &=& (J-R + \lm E) u_2 + B u_3 \nonumber \\
	(\D_J^*- \D_R^* - \lm \D_E^*) u_1  &=&  ((J-R)^* - \lm E^*)u_1 \nonumber \\
	\D_B^* u_1  &=& B^*u_1 + S u_3 \nonumber
\end{eqnarray}
if and only if 
\begin{eqnarray}
	\mat{cc}  \underbrace{\D_J -\D_R + \lm \D_E}_{=:\Delta_1} & \underbrace{\D_B}_{=:\Delta_2} \rix \underbrace{\mat{c} u_2 \\ u_3 \rix}_{=:x} &=& \underbrace{(J-R + \lm E) u_2 + B u_3}_{=:y} \label{dsmequi1} \\
	\mat{cc}  \D_J -\D_R + \lm \D_E& \D_B \rix ^* \underbrace{u_1}_{=:z} &=& \underbrace{\mat{c}-(J+R + \lm E)u_1  \\ B^*u_1 + S u_3 \rix}_{=:w}, \label{dsmequi2}
\end{eqnarray}
since $\lambda \in i\R$. Thus for any $\Delta L(z)=\Delta_M + z \Delta_N$  and $(\lambda,u) \in i\R \times \C^{2n+m}\setminus\{0\} $,  $\left(L(\lambda)-\Delta L (\lambda)\right)u=0$ is equivalent to solving the doubly structured mapping defined by~\eqref{dsmequi1}-\eqref{dsmequi2}, where the structure on $\Delta_1$ depends on the structures imposed on the perturbations $\Delta_J$, $\Delta_R$, and $\Delta_E$.
}
\end{remark}

In view of~\eqref{dsmequi1} and~\eqref{dsmequi2}, the following lemma is analogous to~\cite[Lemma 6.2]{MehMS18} that will be useful in preserving the semidefinite structure on $R$ in the backward error $\eta^{\mathcal S_d}(J,R,E,B,\lambda,u)$. 
\begin{lemma}\label{lem:JpsdREB}
	Let $L(z) $ be a pencil as in~\eqref{eq:defpenciL}, and let $\lm \in i \R$ and $u \in \C^{2n +m} \setminus \{ 0\}$. Partition $u= [u_1^T ~ u_2^T ~ u_3^T]^T$ such that $u_1,u_2 \in \C^{n}$ and $u_3 \in \C^m$, and let $x,y,z$ and $w$ be defined by~\eqref{dsmequi1} and~\eqref{dsmequi2}. Then the following statements are equivalent.
	\begin{enumerate}
\item There exists $\D_J,\D_R, \D_E \in \C^{n,n}$ and $\D_B \in \C^{n,m}$ such that $\D_J \in {\rm SHerm}(n)$, $\D_R \succeq 0$, and $\D_E \in {\rm Herm}(n)$ satisfying~\eqref{dsmequi1} and~\eqref{dsmequi2}.
		\item There exists $\D =[\D_1 ~ \D_2]$, $\D_1 \in \C^{n,n}$, $\D_2 \in \C^{n,m}$  such that $\D_1+ \D_1^* \preceq 0$, $\D x =y$, and  $\D^* z=w$.
		\item $u_3 = 0$.
	\end{enumerate}
	Moreover, we have 
\begin{eqnarray}
		&\inf \Big \{ {\left \|\big[\D_J ~ \D_R ~ \D_E ~ \D_B\big] \right \|}_F^2  : ~\D_J \in {\rm SHerm}(n),\D_R,  \D_E\in {\rm Herm}(n), \Delta_R\succeq 0,\\
		&~\D_B \in \C^{n,m}~\text{satisfying}~\eqref{dsmequi1} ~\text{and}~\eqref{dsmequi2} \Big \} \nonumber \\
		&=\inf \bigg \{  {\left \| \frac{\D_1 + \D_1^*}{2} \right \|}_F^2 + \frac{1}{1+|\lm|^2}{\left \| \frac{\D_1 - \D_1^*}{2} \right \|}_F^2 + {\|\D_2\|}_F^2 : ~ \D =\big[\D_1 ~ \D_2\big],~\D_1 \in \C^{n,n}, \\
		& \D_2 \in \C^{n,m},~\D_1+\D_1^* \preceq 0,~\D x =y, \D ^* z =w \bigg \} \nonumber.
	\end{eqnarray}
\end{lemma}
\proof The proof is similar to the proof of~\cite[Lemma 6.2]{MehMS18} due to Type-2 doubly structured dissipative mapping from Theorem~\ref{thm:Type-2DSDM}.
\eproof

\begin{theorem}\label{bacerrJpsdREB}
	Let $L(z)$ be a pencil as in~\eqref{eq:defpenciL}, let $\lm \in i\R \setminus \{0\}$ and $u \in \C^{2n+m} \setminus \{0\}$. Partition $u=[u_1^T~ u_2^T ~ u_3^T]^T$ such that $u_1,u_2 \in \C^n$, and $u_3 \in \C^n$. Set $\tilde y=  (J-R + \lm E) u_2$ and $w_1= -(J+R +\lm E)u_1$. Then $\eta^{\mathcal S_d}(J,R,E,B,\lm,u)$ is finite if and only if $u_3=0$. If the later condition holds and if $u_2$ satisfies that 
	$Ru_2 \neq 0$ and $u_2 =\alpha u_1$ for some nonzero $\alpha \in \C$, then 
\begin{equation}\label{eq:thmJpsdREB_1}
\frac{1}{1+ |\lm|^2} {\|H_1\|}_F^2 + {\|H_2\|}_F^2 \leq  
\left(\eta^{\mathcal S_d}(J,R,E,B,\lm,u)\right)^2 \leq {\|H_1\|}_F^2+{\|H_2\|}_F^2,
\end{equation}	
where 
\begin{equation}\label{eq:thmJpsdREB_2}
H_1=\tilde yu_2^\dagger + (w_1 u_1^\dagger)^*\mathcal P_{u_2}+\mathcal P_{u_2}J\mathcal P_{u_2} \quad \text{with}\quad J=\frac{1}{4\real{(u_2^*{\tilde y}})} \big(\tilde y+\frac{\alpha}{|\alpha|^2}w_1\big)\big(\tilde y+\frac{\alpha}{|\alpha|^2}w_1\big)^*
\end{equation}
and $H_2=u_1u_1^\dagger B$.
\end{theorem}
\proof In view of Remark~\ref{rem:appdsw} and Lemma~\ref{lem:JpsdREB}, we obtain that $\eta^{\mathcal S_d}(J,R,E,B,\lm,u)$ is finite if and only if $u_3= 0$. Thus by substituting $u_3=0$ in~\eqref{dsmequi1} and~\eqref{dsmequi2}, and using Lemma~\ref{lem:JpsdREB} in~\eqref{def:errpsd}, we have that
\begin{eqnarray}\label{eq:thmJpsdREB_3}
&\left(\eta^{\mathcal S_d}(J,R,E,B,\lm,u)\right)^2 = 
\inf\Big\{  {\left \| \frac{\D_1 + \D_1^*}{2} \right \|}_F^2 + \frac{1}{1+|\lm|^2}{\left \| \frac{\D_1 - \D_1^*}{2} \right \|}_F^2 + {\|\D_2\|}_F^2:~
\D_1 \in \C^{n,n}, \nonumber \\
		&  \hspace{2cm}\D_2 \in \C^{n,m},~\D_1+\D_1^* \preceq 0,~\D_1 u_2 =\tilde y, \D_1 ^* u_1 =w_1,~\Delta_2^*u_1=B^*u_1
\Big\}. 
\end{eqnarray}
If $u_2=\alpha u_1$ for some nonzero $\alpha \in \C$ and $u_2  \notin \text{ker}(R)$, then from Theorem~\ref{map result: NSPSD 4vec} and Remark~\ref{rem:type-1dsdm} there always exists a Type-1 doubly structured dissipative mapping $\Delta_1$ such that  $\D_1+\D_1^* \preceq 0$ and
\begin{equation}\label{eq:thmJpsdREB_3}
\D_1 u_2 =\tilde y,\quad \text{and}\quad \D_1 ^* u_1 =w_1.
\end{equation}
This is because of Theorem~\ref{map:type1veccase}  as the necessary and sufficient conditions $u_1^*w_1={\tilde y}^*u_2$ and $\real{(u_2^*\tilde y)}\leq 0$ for the existence of such a $\Delta_1$ are satisfied, since $R \succeq 0$, $J^*=-J$, $E^*=E$, and $\lambda \in i\R$. 
Further, the minimal Frobenius norm of such a $\Delta_1$ is attained by the unique matrix $H_1$ defined in~\eqref{eq:thmJpsdREB_2}. Similarly, from Theorem~\ref{map result: unstruct}, for any $u_1$ there always exists $\Delta_2 \in \C^{n,m}$ such that 
$\Delta_2^*u_1=B^*u_1$ and the minimal Frobenius norm of such a $\Delta_2$ is attained by $H_2:=u_1u_1^\dagger B$.

Next, observe that for any $\Delta_1 \in \C^{n,n}$, we have 
${\|\Delta_1\|}_F^2={\left\|\frac{\Delta_1+\Delta_1^*}{2}\right\|}_F^2+{\left\|\frac{\Delta_1-\Delta_1^*}{2}\right\|}_F^2$. This implies that for any $\Delta_1 \in \C^{n,n}$ and $\Delta_2 \in \C^{n,m}$, we have
\begin{equation}\label{eq:thmJpsdREB_4}
\frac{1}{1+|\lambda|^2}{\|\Delta_1\|}_F^2+{\|\Delta_2\|}_F^2\leq
{\left\|\frac{\Delta_1+\Delta_1^*}{2}\right\|}_F^2+\frac{1}{1+|\lambda|^2}{\left\|\frac{\Delta_1-\Delta_1^*}{2}\right\|}_F^2+ {\|\Delta_2\|}_F^2 \leq {\|\Delta_1\|}_F^2+{\|\Delta_2\|}_F^2.
\end{equation}
Thus by taking the infimum over all $\Delta_1\in \C^{n,n}$ and $\Delta_2 \in \C^{n,m}$ satisfying the mappings in the right hand side of~\eqref{eq:thmJpsdREB_3}, and by using the minimal Frobenius norm mappings $H_1$ and $H_2$, we obtain~\eqref{eq:thmJpsdREB_1}. This completes the proof.
\eproof

\subsection{Numerical experiments}

In Table~\ref{tab:results}, we present some numerical experiments to illustrate the results of this section. We generate a random pencil $L(z)$ of the form~\eqref{eq:defpenciL} with no eigenvalues on the imaginary axis and compare the various eigenpair backward errors for perturbations to all the blocks $J$, $R$, $E$, and $B$ of $L(z)$. The $\lambda$-values are chosen randomly on the imaginary axis, and $u \in \C^{2n+m}$ is chosen to satisfy the conditions of Theorem~\ref{bacerrJpsdREB}. 
The block structured backward error $\eta^{\mathcal B} (J,R,E,B,\lambda,u)$ and the symmetry structured eigenpair backward error $\eta^{\mathcal S} (J,R,E,B,\lambda,u)$ were obtained in~\cite[Theorem 6.3]{MehMS18}.
The semidefinite structure-preserving backward error $\eta^{\mathcal S_{d}} (J,R,E,B,\lambda,u)$ is obtained in Theorem~\ref{bacerrJpsdREB}. We observe that the eigenpair backward error is significantly larger when semidefinite structure-preserving perturbations are considered instead of  block structure-preserving ones or symmetry structure-preserving ones. 
The tightness of the lower and upper bounds for $\eta^{\mathcal S_{d}} (J,R,E,B,\lambda,u)$ depends on the value of $\lambda$, as shown in Theorem~\ref{bacerrJpsdREB}.

\begin{center}  
	\begin{table}[h!] 
		\begin{center} 
			\caption{Comparison of various  block-/symmetry-/semidefinite-structure-preserving eigenpair backward errors of $L(z)$ under perturbations to the blocks $J,R,E$, and $B$ of $L(z)$. Here, l.b.  and u.b. respectively stand for the terms lower and upper bound. } 
			\label{tab:results}
			\begin{tabular}{|c|c|c|c|c|c|} 
				\hline $\lm$  &  $\eta^{\mathcal B}$ & l.b. of $\eta^{\mathcal S}$ & u.b. of $\eta^{\mathcal S}$ & l.b. of $\eta^{\mathcal S_d}$ & u.b. of $\eta^{\mathcal S_d}$   \\  
  &\cite{MehMS18}   & \cite[Theorem 6.3]{MehMS18} &  \cite[Theorem 6.3]{MehMS18} &  Theorem~\ref{bacerrJpsdREB} &  Theorem~\ref{bacerrJpsdREB}  \\ \hline 
				$0.1380i$  & $23.9305$ & $28.2248$ & $28.4919$ & $30.6476$ & $30.9366$ \\ 
				\hline$0.5100i$  & $23.5909$ & $25.7498$ & $29.0013$ & $27.9347$ & $31.3382$ \\ 
				\hline$0.8950i$  & $23.0355$ & $22.2134$ & $30.7630$ & $24.0542$ & $32.2223$ \\ 
				\hline$1.0480i$  & $22.8160$ & $20.9155$ & $32.0553$ & $22.6284$ & $32.6987$ \\ 
				\hline$1.3210i$  & $22.4764$ & $18.9091$ & $35.5143$ & $20.4225$ & $33.7134$ \\ 
				\hline$1.9080i$ & $22.0725$ & $15.8640$ & $48.9597$ & $17.0707$ & $36.5371$ \\ 
				\hline$2.5080i$ &  $22.1087$ & $13.9975$ & $71.4942$ & $15.0184$ & $40.1794$ \\ 
				\hline\end{tabular} 
		\end{center} 
	\end{table} 
\end{center}

The eigenpair backward errors of $L(z)$ when only specific blocks in the pencil $L(z)$ are perturbed also follow similar lines and have been kept in~\ref{sec:appe2} for future reference. In Table~\ref{tab2:thmused}, we 
summarize the results for symmetry and semidefinite structure-preserving backward errors with respect to other combinations of the perturbation blocks $J$, $R$, $E$, and $B$ of $L(z)$. 
%
Table~\ref{tab2:thmused} also covers the cases of symmetry structure-preserving backward errors left open in~\cite{MehMS18} . 

	\begin{table}[h!]
		\begin{center} 
			\caption{An overview of the results for the symmetry- or semidefinite-structure-preserving eigenpair backward error when only specific blocks in the pencil $L(z)$ are perturbed.} 
			\label{tab2:thmused}
			\begin{tabular}{|c|c|c|} 
			 	\hline perturbation blocks & $\eta^{\mathcal S}(\cdot,\cdot,\cdot,\cdot,\lambda,u)$ & $\eta^{\mathcal S_d}(\cdot,\cdot,\cdot,\cdot,\lambda,u)$ \\
			 	\hline J and R  &\cite[Theorem 4.14]{MehMS18}& Theorem~\ref{bacerrJpsdR} \\
			 	\hline J and E & \cite[Theorem 4.6]{MehMS18}& \cite[Theorem 4.6]{MehMS18} \\
			 	\hline J and B & Theorem~\ref{bacerrJB}& Theorem~\ref{bacerrJB} \\
			 	\hline R and E & \cite[Theorem 4.10]{MehMS18}& Theorem~\ref{bacerrpsdRE} \\
			 	\hline R and B & Theorem~\ref{bacerrRB}& Theorem~\ref{bacerrpsdRB} \\
			 	\hline E and B & Theorem~\ref{bacerrEB}& Theorem~\ref{bacerrEB} \\
			 	\hline J,R and E &  \cite[Theorem 5.11]{MehMW18}&  Theorem~\ref{bacerrJpsdRE} \\
			 	\hline J,R and B & \cite[Theorem 5.4]{MehMS18} & Theorem~\ref{bacerrJpsdRB} \\
			 	\hline R,E and B &  \cite[Theorem 5.7]{MehMW18} &   Theorem~\ref{bacerrpsdREB} \\
			 	\hline J,E and B & Theorem~\ref{bacerrJEB} & Theorem~\ref{bacerrJEB}  \\
			 	\hline J,R,E and B &  \cite[Theorem 6.3]{MehMW18}&   Theorem~\ref{bacerrJpsdREB}\\  \hline
			\end{tabular}
			\end{center}
	\end{table}

\section*{References}	
\bibliographystyle{siam}
\bibliography{ab1}

\appendix

\section{Proof of Theorem~\ref{theorem:DSSM}}	\label{app:proofcomsym}

\proof
Let us suppose that $\mathcal S_d^{{\rm Sym}} \neq \emptyset$. Then there exists $\D = [\D_1 ~ \D_2]$ with $\D_1\in \C^{n,n}$ and $\D_2 \in \C^{n,m}$ such  that  $\D_1^T= \D_1, \D x=y$, and $\D^* z=w$. This implies that $y^*z= (\D x)^* z = x^* \D^* z=x^* w$. Conversely, if $x^*w=y^*z, z^Tw_1 \in \R $, then $H=[ H_1 ~ H_2] $ satisfies that  $H x=y, H^*z=w$, and $H_1^T = H_1$, which implies that $H \in \mathcal S_d^{{\rm Sym}}$.

Next, we prove~\eqref{eq : char double struct_T}. First suppose that 
$\D \in \mathcal S_d^{{\rm Sym}}$, i.e., $\D=[\D_1 ~ \D_2]$, such that $\D x = y, \D^* z = w$, and $\D_1 ^T=\D_1 $. This implies that 	
\begin{eqnarray}\label{val: eq1 dstruct_T} 
\D_1 x_1 + \D_2 x_2 =y,  \quad \bar{\D}_1 z = w_1  \quad \text{and}\quad
\D_2^* z = w_2. 
\end{eqnarray}
Since $\Delta_1$ is a Complex-Symmetric matrix taking $\bar{z}$ to $\bar{w_1}$, from Theorem~\ref{theorem: Complex-sym map} $\D_1$ has the form
\begin{equation} \label{val: delta_1 map_T}
\D_1 = \bar{w}_1 \bar{z}^\pls + (\bar{w}_1 \bar{z}^\pls) ^T - \bar{z}^{\pls^T} \bar{z}^T \bar{w}_1 \bar{z}^\pls + (\mathcal P_{\bar{z}})^T K \mathcal P_{\bar{z}} 
\end{equation} 
for some Complex-symetric matrix $K \in \C^{n,n}$. By substituting $\D_1$ from~\eqref{val: delta_1 map_T} in~\eqref{val: eq1 dstruct_T}, we get	
\begin{eqnarray}\label{eq: delta2 stage 2_T} 	
\D_2 x_2 = y- \left ( \bar{w}_1 \bar{z}^\pls + (\bar{w}_1 \bar{z}^\pls) ^T - \bar{z}^{\pls^T} \bar{z}^T \bar{w}_1 \bar{z}^\pls + (\mathcal P_{\bar{z}})^T K \mathcal P_{\bar{z}} \right )x_1  \quad \text{and}\quad \D_2 ^* z = w_2,
\end{eqnarray} 
i.e., a mapping of the form $\D_2 x_2=\tilde y$ and $\D_2^*z=w_2$, where 
\[
\tilde y=\left(y-  ( \bar{w}_1 \bar{z}^\pls + (\bar{w}_1 \bar{z}^\pls) ^T - \bar{z}^{\pls^T} \bar{z}^T \bar{w}_1 \bar{z}^\pls + (\mathcal P_{\bar{z}})^T K \mathcal P_{\bar{z}} )x_1\right ).
\]
 The vectors $x_2,\tilde y, z$, and $w_2$ satisfy 
\begin{eqnarray*}
\tilde y^*z&=&\left(y-  ( \bar{w}_1 \bar{z}^\pls + (\bar{w}_1 \bar{z}^\pls) ^T - \bar{z}^{\pls^T} \bar{z}^T \bar{w}_1 \bar{z}^\pls + \mathcal P_{\bar{z}}^T K \mathcal P_{\bar{z}} )x_1\right)^*z\\
&=&\left(y^*-  x_1^*( {w}_1 {z}^\pls + ({w}_1 {z}^\pls) ^T - {z}^{\pls^T} {z}^T {w}_1 {z}^\pls + \mathcal P_{{z}}^T \bar K \mathcal P_{{z}} )^T\right)z\\
&=&\left(y^*-  x_1^*( {w}_1 {z}^\pls + ({w}_1 {z}^\pls) ^T - {z}^{\pls^T} {z}^T {w}_1 {z}^\pls + \mathcal P_{{z}}^T \bar K \mathcal P_{{z}} )\right)z \quad ( \because K^T=K)\\
&=&y^*z- x_1^* w_1  \quad (\because ({w}_1 {z}^\pls + ({w}_1 {z}^\pls) ^T - {z}^{\pls^T} {z}^T {w}_1 {z}^\pls + \mathcal P_{{z}}^T K \mathcal P_{{z}}) z=w_1 \\\
&=& x_2^* w_2 \quad (\because x^*w= x_1^*w_1 + x_2^* w_2~\text{and}~x^*w=y^*z).
\end{eqnarray*}
Therefore, from Theorem~\ref{theorem : 2.1 mehl2017}, $\Delta_2$ can be written as 
\begin{equation} \label{val: delta2 map2unstruct_T}
\D_2 = \tilde y x_2^\pls +  (w_2 z^\pls)^* - (w_2 z^\pls)^*x_2 x_2^\pls + \mathcal P_z R \mathcal P_{x_2},
\end{equation}
for some $R \in \C^{n,m}$.	

Thus, in view of ~\eqref{val: delta_1 map_T} and~\eqref{val: delta2 map2unstruct_T},  we have 
\begin{eqnarray}
\left [\D_1 ~ \D_2\right ] &=& \big [ \bar{w}_1 \bar{z}^\pls + (\bar{w}_1 \bar{z}^\pls) ^T - \bar{z}^{\pls^T} \bar{z}^T \bar{w}_1 \bar{z}^\pls + \mathcal P_{\bar{z}}^T K \mathcal P_{\bar{z}} ~~~ \tilde y x_2^\pls +  (w_2 z^\pls)^* - (w_2 z^\pls)^* x_2 x_2^\pls +  \mathcal P_z R  \mathcal P_{x_2}\big ]\nonumber \\
&=& \big[H_1+\tilde H_1(K)~~H_2+\tilde H_2(K,R)\big] \nonumber \\
&=& H + \wt H(K,R).
\end{eqnarray}
This proves ``$\subseteq$" in~\eqref{eq : char double struct_T}.

Conversely, let $[\Delta_1~\Delta_2]=[ H_1 + \wt H_1(K) ~~ H_2 + \wt H_2(K,R) ]$, where $H_1,\wt H_1(K),H_2$, and $\wt H_2(K,R)$ are defined by~\eqref{val: H_1 double struct_T}-\eqref{val: wt H_2 double struct_T} for some matrices $R \in \C^{n,m}$ and $K \in \C^{n,n}$ such that $K^T=K$. Then it is easy to check that $[\Delta_1~\Delta_2]x=y$ and 
$[\Delta_1~\Delta_2]^*z=w$ since $x^*w=y^*z$. Also $(H_1 + \wt H_1(K))^T= H_1 + \wt H_1(K)$ since $K^T=K$. 
Hence $[\Delta_1~\Delta_2] \in \mathcal S_d^{{\rm Sym}} $. This shows 
``$\supseteq$" in~\eqref{eq : char double struct_T}.

In view of~\eqref{eq : char double struct_T}, we have
\begin{eqnarray}
\inf_{\D \in \mathcal S_d^{{\rm Sym}} } {\|\D\|}_F^2 &=& \inf_{K\in \C^{n,n},R\in \C^{n,m}, K^T=K} {\left \|H + \wt H(K,R)\right \|}_F^2 \nonumber \\ 
&=&  \inf_{K\in \C^{n,n},R\in \C^{n,m}, K^T=K}\left( { \big\| \big [ H_1 ~ H_2 \big]+ \big[ \wt H_1(K) ~ \wt H_2(K,R) \big]\big \|}_F^2 \right) \nonumber \\
&=& \inf_{K\in \C^{n,n},R\in \C^{n,m}, K^T=K}\left( { \big\| H_1 + \wt H_1(K)\big\|}_F^2 + {\big\| H_2 + \wt H_2(K,R)\big\|}_F^2 \right) \nonumber \\
&\geq & \inf_{K\in \C^{n,n}, K^T=K}{\big\| H_1 + \wt H_1(K)\big\|}_F^2  + \inf_{K\in \C^{n,n},R\in \C^{n,m}, K^T=K}  {\big\| H_2 + \wt H_2(K,R)\big\|}_F^2     \nonumber \\ \label {eq:firstineq_T}  \\
&=&{\| H_1\|}_F^2 + \inf_{K\in \C^{n,n},R\in \C^{n,m}, K^T=K} {\big\| H_2 + \wt H_2(K,R)\big\|}_F^2 \label {eq:secineq_T}  \\
&=& {\| H_1\|}_F^2 + \inf_{K\in \C^{n,n}, K^T=K} \Big( \inf_{R\in \C^{n,m}}  {\big \|H_2 + \wt H_2(K,R)\big\|}_F^2 \Big)     \nonumber \\
&=& {\| H_1\|}_F^2 + \inf_{K\in \C^{n,n}, K^T=K}  {\big\|H_2 + \wt H_2(K,0)\big \|}_F^2, \label{eq:thiineq_T} 
\end{eqnarray} 
where the first inequality in~\eqref{eq:firstineq_T} follows due to the fact that for any two real valued functions $f$ and $g$ defined on the same domain, $\inf(f+g)\geq \inf f+\inf g$. Also equality in~\eqref{eq:secineq_T} follows since the infimum in the first term is attained when $K=0$. In fact, for any $K\in \C^{n,n}$ such that $K^T=K$, we have $(H_1 + \wt H_1(K))z=w_1$, which implies from 
Theorem~\ref{theorem: Complex-sym map} that the minimum of ${\|H_1 + \wt H_1(K)\|}_F$ is attained when $K=0$. Further, for a fixed $K$ and for any $R \in \C^{n,m}$, $H_2 + \wt H_2(K,R)$ is a matrix satisfying $(H_2 + \wt H_2(K,R))x_2=\tilde y$ and $(H_2 + \wt H_2(K,R))^*z=w_2$. This implies from Theorem~\ref{theorem : 2.1 mehl2017} that for any fixed $K$, the minimum of  ${\|H_2 + \wt H_2(K,R)\|}_F$ over $R$ is attained when $R=0$, which yields~\eqref{eq:thiineq_T}. This proves~\eqref{eq: minimal norm double struct_T}.

Next suppose if $x_1=\alpha z$ for some nonzero $\alpha \in \C$, then $\wt H_2(K,0)=0$ for every $K \in \C^{n,n}$. This implies from~\eqref{eq:thiineq_T} that 
\begin{eqnarray}
\inf_{\D \in \mathcal S_d^{{\rm Sym}} } {\|\D\|}_F^2~\geq~ {\| H_1\|}_F^2+{\| H_2\|}_F^2\,= \, {\| H\|}_F^2,
\end{eqnarray}
and in this case the lower bound is attained since $H \in \mathcal S_d^{{\rm Sym}}$. This completes the proof. 
\eproof

\section{Estimation of $\eta^{\mathcal S}(\cdot,\cdot,\cdot,\cdot,\lm,u)$ and $\eta^{\mathcal S_d}(\cdot,\cdot,\cdot,\cdot,\lm,u)$ when  perturbing any two/three of the blocks $J$, $R$, $E$ and $B$ of the pencil $L(z)$}	\label{sec:appe2}

Let $L(z)$ be a pencil of the form~\eqref{eq:defpenciL}, $\lm \in i\R$ and $u=\big[u_1^T~ u_2^T ~ u_3^T \big]^T$ with $u_1,u_2 \in \C^{n} \setminus\{0\}$ and $u_3 \in \C^{m}$. 

\subsection{Perturbing only $J$ and $R$}

Suppose that only $J$ and $R$ blocks of $L(z)$ are subject to perturbation. Then in view of~\eqref{def:errblc},~\eqref{def:errstr} and~\eqref{def:errpsd}, the corresponding  backward errors are denoted by 
$\eta^{\mathcal B}(J,R,\lm,u):= \eta^{\mathcal B}(J,R,0,0,\lm,u)$, $\eta^{\mathcal S}(J,R,\lm,u):= \eta^{\mathcal S}(J,R,0,0,\lm,u)$, and $\eta^{\mathcal S_d}(J,R,\lm,u):= \eta^{\mathcal S_d}(J,R,0,0,\lm,u)$.  In this case, the block-structured and   the symmetry-structured backward errors $\eta^{\mathcal B}(J,R,\lm,u)$ and $\eta^{\mathcal S}(J,R,\lm,u)$ were obtained in~\cite[Theorem 4.14]{MehMS18}. Thus,  we provide estimation only for the semidefinte-structured backward error $\eta^{\mathcal S_d}(J,R,\lm,u)$.
In view of \eqref{dsmequi1} and \eqref{dsmequi2} , when $\D_E=0,\D_B=0$ we obtain
\begin{eqnarray}
	(\D_J - \D_R) \underbrace{u_2}_{:=x} &=& \underbrace{(J-R+ \lm E)u_2 + B u_3}_{:=y} \label{dsmequiJR1}\\ 
	(\D_J - \D_R)^*\underbrace{u_1}_{:=z} &=& \underbrace{-(J+R+\lm E)u_1}_{:=w} \label{dsmequiJR2} \\ 
	 B^*u_1 + S u_3&=& 0. \label{dsmequiJR3}
\end{eqnarray}
This gives us the following lemma which is analogous to Lemma~\ref{lem:JpsdREB}.
\begin{lemma}\label{lem: JpsdR}
	Let $L(z) $ be a pencil as in~\eqref{eq:defpenciL}, and let $\lm \in i \R$ and $u \in \C^{2n +m} \setminus \{ 0\}$. Partition $u= [u_1^T ~ u_2^T ~ u_3^T]^T$ such that $u_1,u_2 \in \C^{n}$ and $u_3 \in \C^m$, and let $x,y,z$ and $w$ be defined as in~\eqref{dsmequiJR1} and~\eqref{dsmequiJR2}. If $z=\alpha x$, then the following statements are equivalent. 
	\begin{enumerate}
		\item There exists $\D_J,\D_R \in \C^{n,n}$ such that $\D_J \in {\rm SHerm}(n) , \D_R \succeq 0$ satisfying \eqref{dsmequiJR1} and \eqref{dsmequiJR2}.
		
		\item There exists $\D \in \C^{n,n}$ such that $\D + \D^* \preceq 0, \D x=y,\D^*z=w$.
		
		\item $u_3^*B^*u_1 = 0$.
	\end{enumerate}
	Moreover, we have 
	\begin{eqnarray}
		&& \inf \left \{ {\|[\D_J ~ \D_R]\|}_F^2:~ \D_J \in {\rm SHerm}(n),\, \D_R \in {\rm Herm}(n),\, \D_R \succeq 0, ~\text{satisfying}~ \eqref{dsmequiJR1} ~ \text{and}~ \eqref{dsmequiJR2} \right \} \nonumber \\	
		&& =~  \inf \left \{ {\|\D\|}_F^2 ~:~  \Delta \in \C^{n,n},\,\D+ \D^* \preceq 0,\,\D x= y,\,\D^*z=w \right \}. \nonumber
	\end{eqnarray}	
\end{lemma}
\proof The proof is analogous to \cite[Lemma 4.13]{MehMS18} due to Type-1 doubly structured dissipative mapping from Theorem~\ref{map:type1veccase}.
\eproof

\begin{theorem}\label{bacerrJpsdR}
	Let $L(z)$ be a pencil as in~\eqref{eq:defpenciL}, let $\lm \in i\R \setminus \{0\}$ and $u \in \C^{2n+m} \setminus \{0\}$. Partition $u=\big[u_1^T~ u_2^T ~ u_3^T \big]^T$ such that $u_1,u_2 \in \C^n$, and $u_3 \in \C^m$. Set $\tilde y=  (J-R + \lm E) u_2$ and $w_1= -(J+R +\lm E)u_1$. Let $u_2 =\alpha u_1$ for some nonzero $\alpha \in \C$. Then $\eta^{\mathcal S_d}(J,R,\lm,u)$ is finite if and only if $u_3=0$ and $B^*u_1=0$. If the later condition holds and if $u_2$ satisfies that 
	$Ru_2 \neq 0$, then 
	\begin{equation}\label{eq:thmJpsdR_1} 
		\eta^{\mathcal S_d}(J,R,\lm,u) = {\|H\|}_F,
	\end{equation}	
	where 
	\begin{equation*}\label{eq:thmJpsdR_2}
		H=\tilde yu_2^\dagger + (w_1 u_1^\dagger)^*\mathcal P_{u_2}+\mathcal P_{u_2}J\mathcal P_{u_2}, \quad \text{and}\quad 
		J=\frac{1}{4\real{(u_2^*{\tilde y}})} \big(\tilde y+\frac{\alpha}{|\alpha|^2}w_1\big)\big(\tilde y+\frac{\alpha}{|\alpha|^2}w_1\big)^*.
	\end{equation*}
\end{theorem}
\proof The proof is anologous to the proof of~\cite[Theorem 4.14]{MehMS18} due to Lemma~\ref{lem: JpsdR} and Theorem~\ref{map:type1veccase}.
\eproof


\subsection{Perturbing only $J $ and $B$}

In this section, suppose that only $J$ and $B$ blocks of $L(z)$ are subject to perturbation.  Then in view of~\eqref{def:errblc},~\eqref{def:errstr} and~\eqref{def:errpsd},  the  corresponding backward  errors are denoted by $\eta^{\mathcal B}(J,B,\lm,u):= \eta^{\mathcal B}(J,0,0,B,\lm,u),~\eta^{\mathcal S}(J,B,\lm,u):= \eta^{\mathcal S}(J,0,0,B,\lm,u)$, and $\eta^{\mathcal S_d}(J,B,\lm,u):= \eta^{\mathcal S_d}(J,0,0,B,\lm,u)$. 
Note that  the block-structured backward error $\eta^{\mathcal B}(J,B,\lm,u)$ was obtained in~\cite[Theorem 4.17]{MehMS18}, but the symmetry-structured backward error $\eta^{\mathcal S}(J,B,\lm,u)$ were not known  in~\cite{MehMS18} due to  unavailability  of the doubly  structured  skew-Hermitian mappings. Also note that $\eta^{\mathcal S_d}(J,B,\lm,u)=\eta^{\mathcal S}(J,B,\lm,u)$, because we are not perturbing $R$ and there  is no semidefinite structure  on $J$ and $B$. 
To estimate $\eta^{\mathcal S}(J,B,\lm,u)$ from~\eqref{dsmequi1} and~\eqref{dsmequi2}, when $\D_R=0$ and $\D_E=0$, we obtain
\begin{eqnarray}
	\mat{cc} \underbrace{\D_J }_{=:\Delta_1} & \underbrace{\D_B}_{=:\Delta_2} \rix \underbrace{\mat{c} u_2 \\ u_3 \rix}_{=:x} &=& \underbrace{(J-R + \lm E) u_2 + B u_3}_{=:y} \label{dsmequiJB1} \\
	\mat{cc}  \D_J& \D_B \rix ^* \underbrace{u_1}_{=:z} &=& \underbrace{\mat{c}-(J+R + \lm E)u_1=:w_1  \\ B^*u_1 + S u_3=:w_2 \rix}_{=:w}, \label{dsmequiJB2}
\end{eqnarray}
This leads to the following lemma.
\begin{lemma}\label{lem: JB}
	Let $L(z) $ be a pencil as in~\eqref{eq:defpenciL}, and let $\lm \in i \R$ and $u \in \C^{2n +m} \setminus \{ 0\}$. Partition $u= [u_1^T ~ u_2^T ~ u_3^T]^T$ such that $u_1,u_2 \in \C^{n}$ and $u_3 \in \C^m$, and let $x,y,z$ and $w$ be defined as in~\eqref{dsmequiJB1} and~\eqref{dsmequiJB2}. Then the following statements are equivalent. 
	\begin{enumerate}
		\item There exists $\D_J\in \C^{n,n}$ and $\D_B \in \C^{n,m}$ such that $\D_J \in {\rm SHerm}(n)$ satisfying~\eqref{dsmequiJB1} and~\eqref{dsmequiJB2}.
		\item $u_3 = 0$ and $R u_1=0$.
	\end{enumerate}
	\end{lemma}
	\proof The proof is immediate from Theorem~\ref{theorem_SHD} since 
	$J^*=-J$, $R\succeq 0$, and $E^*=E$.
	\eproof
	
	\begin{theorem}\label{bacerrJB}
		Let $L(z)$ be a pencil as in~\eqref{eq:defpenciL}, let $\lm \in i\R \setminus \{0\}$ and $u \in \C^{2n+m} \setminus \{0\}$. Partition $u=\big[u_1^T~ u_2^T ~ u_3^T \big]^T$ such that $u_1,u_2 \in \C^n$, and $u_3 \in \C^m$. Set $\tilde y=  (J-R + \lm E) u_2$ and $w_1= -(J+R +\lm E)u_1$, $X=\big[u_2 ~ u_1\big],Y=\big[\tilde{y} ~ -w_1 \big]$. Then $\eta^{\mathcal S_d}(J,B,\lm,u)$ is finite if and only if $u_3=0$ and $R u_1 = 0$. If the later condition holds and if $YX^\pls X=Y,~Y^*X=-X^*Y$  and if $u_2 =\alpha u_1$ for some nonzero $\alpha \in \C$, then 
		\begin{equation}\label{eq:thmJB_1}  
			\eta^{\mathcal S}(J,B,\lm,u) =\sqrt{ {\|H_1\|}_F^2+{\|H_2\|}_F^2},
		\end{equation}	
		where 
		\begin{equation}\label{eq:thmJB_2}
			H_1=YX^\pls - (YX^\pls)^* - X X^\pls Y X^\pls \quad \text{and}\quad H_2= u_1 u_1^\pls B.
		\end{equation}
	\end{theorem}
 	\proof In view of Remark~\ref{rem:appdsw} and Lemma~\ref{lem: JB}, we obtain that $\eta^{\mathcal S}(J,B,\lm,u)$ is finite if and only if $u_3=0$ and $R u_1 = 0$. Thus by using $u_3=0$ and $R u_1 = 0$ in \eqref{dsmequiJB1} and \eqref{dsmequiJB2}, and using Lemma~\ref{lem: JB} in the definition of $\eta^{\mathcal S}(J,B, \lambda,u)$ from \eqref{def:errstr}, we have that 
 	\begin{eqnarray}\label{eq:thmJB_3}
 		&\eta^{\mathcal S}(J,B,\lm,u) = 
 		\inf \Big \{ {\|\D\|}_F^2  ~: ~ \D =\big[\D_J ~ \D_B\big],\,\D_J \in \C^{n,n},\, \D_B \in \C^{n,m},\,\D_J^*=-\D_J, \nonumber \\
 		&\D_J u_2 =\tilde{y},\, \D_J ^* u_1 =w_1,\, \D_B^* u_1 = B^* u_1 \Big \} \nonumber \\ 
 		& = \inf \Big \{ {\|\D\|}_F^2  ~: ~ \D =\big[\D_J ~ \D_B\big],\,\D_J \in \C^{n,n},\,\D_B \in \C^{n,m},\,\D_J^*=-\D_J, \nonumber \\
 		& \D_J\big[u_2 ~ u_1\big]= \big[\tilde{y} ~ -w_1 \big],\, \D_B^* u_1 = B^* u_1 \Big \} \nonumber \\
 		& = \inf \Big \{ {\|\D\|}_F^2  ~: ~ \D =\big[\D_J ~ \D_B\big],\, \D_J \in \C^{n,n},\,\D_B \in \C^{n,m},\,\D_J^*=-\D_J, \nonumber \\
 		& \D_J X= Y, \,\D_B^* u_1 = B^* u_1 \Big \}.
 	\end{eqnarray}
 	If $Y X^\pls X=Y,~Y^*X=-X^*Y$ and $u_2 =\alpha u_1$ for some nonzero $\alpha \in \C$, then from~\cite[Theorem 2.2.3]{adhikari2008backward}, there always exists a skew-Hermitian mapping $\Delta_J$ such that  $\D_J^*=-\D_J$ and $\D_J X=Y$.
The minimal Frobenius norm of such a $\Delta_J$ is attained by the unique matrix $H_1$ defined in~\eqref{eq:thmJB_2}. Similarly, from Theorem~\ref{map result: unstruct}, for any $u_1$ there always exists $\Delta_B \in \C^{n,m}$ such that  $\Delta_B^*u_1=B^*u_1$ and the minimal Frobenius norm of such a $\Delta_B$ is attained by $H_2:=u_1u_1^\dagger B$. Using the minimal Frobenius norm mappings $H_1$ and $H_2$, we obtain~\eqref{eq:thmJB_1}. This completes the proof.
 	
 	
\subsection{Perturbing only $R$ and $B$}
 	
Here, suppose that only $R$ and $B$ blocks of $L(z)$ are subject to perturbation. Then in view of~\eqref{def:errblc},~\eqref{def:errstr} and~\eqref{def:errpsd}, the corresponding backward errors are denoted by $\eta^{\mathcal B}(R,B,\lm,u):= \eta^{\mathcal B}(0,R,0,B,\lm,u)$, $\eta^{\mathcal S}(R,B,\lm,u):= \eta^{\mathcal S}(0,R,0,B,\lm,u)$, and $\eta^{\mathcal S_d}(R,B,\lm,u):= \eta^{\mathcal S_d}(0,R,0,B,\lm,u)$. 
Note that the backward error $\eta^{\mathcal B}(R,B,\lm,u)$ was obtained in~\cite[Remark 4.18]{MehMS18}. In this section, we compute the eigenpair backward errors $\eta^{\mathcal S}(R,B,\lm,u)$ and $\eta^{\mathcal S_d}(R,B,\lm,u)$. 
 	
For this, observe from~\eqref{dsmequi1} and~\eqref{dsmequi2} that when $\D_J=0$ and  $\D_E=0$,  we have 
\begin{eqnarray}
\mat{cc} \underbrace{-\D_R }_{=:\Delta_1} & \underbrace{\D_B}_{=:\Delta_2} \rix \underbrace{\mat{c} u_2 \\ u_3 \rix}_{=:x} &=& \underbrace{(J-R + \lm E) u_2 + B u_3}_{=:y} \label{dsmequiRB1} \\
 		\mat{cc}  -\D_R& \D_B \rix ^* \underbrace{u_1}_{=:z} &=& \underbrace{\mat{c}-(J+R + \lm E)u_1=:w_1  \\ B^*u_1 + S u_3=:w_2 \rix}_{=:w}, \label{dsmequiRB2}
\end{eqnarray}
In view of Theorem~\ref{theorem HD}, there exists $\D_R\in \C^{n,n}$ and $\D_B \in \C^{n,m}$ such that $\D_R \in {\rm Herm}(n)$ satisfying~\eqref{dsmequiRB1} and~\eqref{dsmequiRB2} if and only if  $u_3 = 0$ and $u_1^*(J+\lambda E)u_1 =0$.  We have the following result for $\eta^{\mathcal S}(R,B,\lm,u)$.

\begin{theorem}\label{bacerrRB}
Let $L(z)$ be a pencil as in~\eqref{eq:defpenciL}, let $\lm \in i\R \setminus \{0\}$ and $u \in \C^{2n+m} \setminus \{0\}$. Partition $u=\big[u_1^T~ u_2^T ~ u_3^T \big]^T$ such that $u_1,u_2 \in \C^n$, and $u_3 \in \C^m$. Set $\tilde y=  (J-R + \lm E) u_2$ and $w_1= -(J+R +\lm E)u_1$, $X=\big[u_2 ~ u_1\big],Y=\big[\tilde{y} ~ w_1 \big]$. Then $\eta^{\mathcal S_d}(R,B,\lm,u)$ is finite if and only if $u_3=0$ and $u_1^*(J+\lambda E)u_1 =0$. If the later condition holds and if $YX^\pls X=Y,~Y^*X=X^*Y$ and if $u_2 =\alpha u_1$ for some nonzero $\alpha \in \C$, then 
 		\begin{equation}\label{eq:thmRB_1}  
 			\eta^{\mathcal S}(R,B,\lm,u) = \sqrt{{\|H_1\|}_F^2+{\|H_2\|}_F^2},
 		\end{equation}	
 		where 
 		\begin{equation}\label{eq:thmRB_2}
 			H_1=YX^\pls + (YX^\pls)^* - X X^\pls Y X^\pls \quad \text{and}\quad  H_2= u_1 u_1^\pls B.
 		\end{equation}
 	\end{theorem}
\proof In view of Remark~\ref{rem:appdsw} and~\eqref{dsmequiRB1}-\eqref{dsmequiRB2},  we obtain that $\eta^{\mathcal S}(R,B,\lm,u)$ is finite if and only if $u_3=0$ and $u_1^*(J+\lambda E)u_1 =0$. Thus by using $u_3=0$~\eqref{dsmequiRB1} and \eqref{dsmequiRB2}, we have from~\eqref{def:errstr} that
 	\begin{eqnarray}\label{eq:thmRB_3}
 		&\eta^{\mathcal S}(R,B,\lm,u) = 
 		\inf \Big \{ {\|\D\|}_F^2  ~: ~ \D =\big[\D_R ~ \D_B\big],\,\D_R \in \C^{n,n},\, \D_B \in \C^{n,m},\,\D_R^*=\D_R, \nonumber \\
 		&\D_R u_2 =\tilde{y},\, \D_R ^* u_1 =w_1, \, \D_B^* u_1 = B^* u_1 \Big \} \nonumber \\ 
 		& = \inf \Big \{ {\|\D\|}_F^2  ~: ~ \D =\big[\D_R ~ \D_B\big],\,\D_R \in \C^{n,n},\,\D_B \in \C^{n,m},\,\D_R^*=\D_R, \nonumber \\
 		& \D_R\big[u_2 ~ u_1\big]= \big[\tilde{y} ~ w_1 \big], \,\D_B^* u_1 = B^* u_1 \Big \} \nonumber \\
 		& = \inf \Big \{ {\|\D\|}_F^2  ~: ~ \D =\big[\D_R ~ \D_B\big],\,\D_R \in \C^{n,n},\,\D_B \in \C^{n,m},\,\D_R^*=\D_R, \nonumber \\
 		& \D_R X= Y,\, \D_B^* u_1 = B^* u_1 \Big \}
 	\end{eqnarray}
 	If $Y X^\pls X=Y,~Y^*X=X^*Y$ and $u_2 =\alpha u_1$ for some nonzero $\alpha \in \C$, then from~\cite[Theorem 2.2.3]{adhikari2008backward}, there always exists a Hermitian mapping $\Delta_R$ such that  $\D_R^*=\D_R$ and $\D_R X=Y$. From~\cite[Theorem 2.2.3]{adhikari2008backward}, the minimal Frobenius norm of such a $\Delta_R$ is attained by the unique matrix $H_1$ defined in~\eqref{eq:thmRB_2}. Similarly, from Theorem~\ref{map result: unstruct}, for any $u_1$ there always exists $\Delta_B \in \C^{n,m}$ such that 
 	$\Delta_B^*u_1=B^*u_1$ and the minimal Frobenius norm of such a $\Delta_B$ is attained by $H_2:=u_1u_1^\dagger B$.
Using minimal Frobenius norm mappings $H_1$ and $H_2$, we obtain~\eqref{eq:thmRB_1}. This completes the proof.	
 	\eproof
 	
Next, we estimate the semidefinite structured backward error $\eta^{\mathcal S_d}(R,B,\lm,u) $. For this, we need the following lemma.
\begin{lemma}\label{lem: psdRB}
 		Let $L(z) $ be a pencil as in~\eqref{eq:defpenciL}, and let $\lm \in i \R$ and $u \in \C^{2n +m} \setminus \{ 0\}$. Partition $u= [u_1^T ~ u_2^T ~ u_3^T]^T$ such that $u_1,u_2 \in \C^{n} \setminus \{0\}$ and $u_3 \in \C^m$, and let $x,y,z$ and $w$ be defined as in~\eqref{dsmequiRB1} and~\eqref{dsmequiRB2}. Then the following statements are equivalent. 
 		\begin{enumerate}
 			\item There exists $\D_R\in \C^{n,n}$ and $\D_B \in \C^{n,m}$ such that $\D_R \succeq 0$ satisfying~\eqref{dsmequiRB1} and~\eqref{dsmequiRB2}.
 			\item There exists $\D =[\D_1 ~ \D_2]$, $\D_1 \in \C^{n,n}$, $\D_2 \in \C^{n,m}$  such that $\D_1 \preceq 0$, $\D x =y$, and  $\D^* z=w$.
 			\item $u_3 = 0$, $u_1^*(J+\lambda E)u_1=0$, and $Ru_1 \neq 0$.
 		\end{enumerate}
 		Moreover, we have 
 		\begin{eqnarray}
 			&\inf \Big \{ {\left \|\big[\D_R ~ \D_B\big] \right \|}_F^2  : ~\D_R \in {\rm Herm}(n)~\D_R \succeq 0 ~\D_B \in \C^{n,m}~\text{satisfying}~\eqref{dsmequiRB1} ~\text{and}~\eqref{dsmequiRB2} \Big \} \nonumber \\
 			&=\inf \bigg \{ {\|\D\|}_F^2  : ~ \D =\big[\D_1 ~ \D_2\big],~\D_1 \in \C^{n,n}, \D_2 \in \C^{n,m},~\D_1^*=\D_1 \preceq 0,~\D x =y, \D ^* z =w \bigg \} \nonumber.
 		\end{eqnarray}
 	\end{lemma}
 	\proof The proof is immediate from the doubly  structured semidefinite mapping from Theorem~\ref{theorem PSDD}.
 	\eproof
 	
 	\begin{theorem}\label{bacerrpsdRB}
 		Let $L(z)$ be a pencil as in~\eqref{eq:defpenciL}, let $\lm \in i\R \setminus \{0\}$ and $u \in \C^{2n+m} \setminus \{0\}$. Partition $u=\big[u_1^T~ u_2^T ~ u_3^T \big]^T$ such that $u_1,u_2 \in \C^n$, and $u_3 \in \C^m$. Set $\tilde y=  (J-R + \lm E) u_2$ and $w_1= -(J+R +\lm E)u_1$, $X=\big[u_2 ~ u_1\big],Y=\big[\tilde{y} ~ w_1 \big]$. Then $\eta^{\mathcal S_d}(R,B,\lm,u)$ is finite if and only if $u_3 = 0$, $u_1^*(J+\lambda E)u_1=0$, and $Ru_1 \neq 0$. If the later condition holds and if $YX^\pls X=Y,~X^*Y \prec 0$ and if  $u_2 =\alpha u_1$ for some nonzero $\alpha \in \C$, then 
 		\begin{equation}\label{eq:thmpsdRB_1}  
 			\eta^{\mathcal S_d}(R,B,\lm,u) =\sqrt{ {\|H_1\|}_F^2+{\|H_2\|}_F^2},
 		\end{equation}	
 		where 
 		\begin{equation}\label{eq:thmpsdRB_2}
H_1=Y (Y^* X)^{-1} Y^* \quad \text{and}\quad H_2= u_1 u_1^\pls B.
 		\end{equation}
 	\end{theorem}
 	\proof In view of Remark~\ref{rem:appdsw} and Lemma~\ref{lem: psdRB}, we obtain that $\eta^{\mathcal S_d}(R,B,\lm,u)$ is finite if and only if $u_3 = 0$, $u_1^*(J+\lambda E)u_1=0$, and $Ru_1 \neq 0$. Thus by using $u_3=0$ in \eqref{eq:thmRB_1}  and \eqref{eq:thmRB_2}, and using Lemma~\ref{lem: psdRB} in~\eqref{def:errstr}, we have that 
 	\begin{eqnarray}\label{eq:thmpsdRB_3}
 		&\eta^{\mathcal S_d}(R,B,\lm,u) = 
 		\inf \bigg \{ {\|\D\|}_F^2  ~: ~ \D =\big[\D_1 ~ \D_2\big],~\D_1 \in \C^{n,n}, \D_2 \in \C^{n,m},~\D_1^*=\D_1 \preceq 0, \nonumber \\
 		&\D_1 u_2 =\tilde{y}, \D_1 ^* u_1 =w_1, \D_2^* u_1 = B^* u_1 \bigg \} \nonumber \\ 
 		& = \inf \bigg \{ {\|\D\|}_F^2  ~: ~ \D =\big[\D_1 ~ \D_2\big],~\D_1 \in \C^{n,n},\D_2 \in \C^{n,m},~\D_1^*=\D_1 \preceq 0, \nonumber \\
 		& \D_1\big[u_2 ~ u_1\big]= \big[\tilde{y} ~ w_1 \big], \D_2^* u_1 = B^* u_1 \bigg \} \nonumber \\
 		& = \inf \bigg \{ {\|\D\|}_F^2  ~: ~ \D =\big[\D_1 ~ \D_2\big],~\D_1 \in \C^{n,n},~\D_2 \in \C^{n,m},~\D_1^*=\D_1 \preceq 0, \nonumber \\
 		& \D_1 X= Y, \D_2^* u_1 = B^* u_1 \bigg \}.
 	\end{eqnarray}
 	If $Y X^\pls X=Y,~X^*Y \prec 0$, and $u_2 =\alpha u_1$ for some nonzero $\alpha \in \C$, then from~\cite[Theorem 2.2]{MehMS17}, there always exists a negative definite mapping $\Delta_1$ such that  $\D_1^*=\D_1 \preceq 0$ and $\D_1 X=Y$. From~\cite[Theorem 2.2]{MehMS17} the minimal Frobenius norm of such a $\Delta_1$ is attained by the unique matrix $H_1$ defined in~\eqref{eq:thmpsdRB_2}. Similarly, from Theorem~\ref{map result: unstruct}, for any $u_1$ there always exists $\Delta_2 \in \C^{n,m}$ such that 
 	$\Delta_2^*u_1=B^*u_1$ and the minimal Frobenius norm of such a $\Delta_2$ is attained by $H_2:=u_1u_1^\dagger B$.
Thus using the minimal Frobenius norm mappings $H_1$ and $H_2$, we obtain~\eqref{eq:thmpsdRB_1}. 
 	\eproof
 	
%
 	\subsection{Perturbing only $E$ and $B$}

In this section, suppose that only $E$ and $B$ blocks of $L(z)$ are subject to perturbation. Then in view of~\eqref{def:errblc},~\eqref{def:errstr} and~\eqref{def:errpsd}, the corresponding  backward  errors are denoted by $\eta^{\mathcal B}(E,B,\lm,u):= \eta^{\mathcal B}(0,0,E,B,\lm,u)$, $\eta^{\mathcal S}(E,B,\lm,u):= \eta^{\mathcal S}(0,0,E,B,\lm,u)$, and $\eta^{\mathcal S_d}(E,B,\lm,u):= \eta^{\mathcal S_d}(0,0,E,B,\lm,u)$.
Again note that $\eta^{\mathcal B}(E,B,\lm,u)$ was obtained in~\cite[Theorem 4.19]{MehMS18}, and we have $\eta^{\mathcal S_d}(E,B,\lm,u)=\eta^{\mathcal S}(E,B,\lm,u)$ because we are not perturbing $R$ and there is no semidefinite structure on $E$ or $B$. 

 	In view of~\eqref{dsmequi1} and~\eqref{dsmequi2}, when $\D_J=0$ and $\D_R=0$, we have 
 	\begin{eqnarray}
 		\mat{cc} \underbrace{\lm \D_E }_{=:\Delta_1} & \underbrace{\D_B}_{=:\Delta_2} \rix \underbrace{\mat{c} u_2 \\ u_3 \rix}_{=:x} &=& \underbrace{(J-R + \lm E) u_2 + B u_3}_{=:y} \label{dsmequiEB1} \\
 		\mat{cc}  \lm \D_E& \D_B \rix ^* \underbrace{u_1}_{=:z} &=& \underbrace{\mat{c}-(J+R + \lm E)u_1=:w_1  \\ B^*u_1 + S u_3=:w_2 \rix}_{=:w}, \label{dsmequiEB2}
 	\end{eqnarray}
As $\lambda \in i\R$ and $\Delta_E^*=\Delta_E$, we have that $\Delta_1=\lambda \Delta_E$ is skew-Hermitian. Then a direct application of the doubly structured skew-Hermitian mapping from Theorem~\ref{theorem_SHD} yields the following emma.  	
\begin{lemma}\label{lem: EB}
 		Let $L(z) $ be a pencil as in~\eqref{eq:defpenciL}, and let $\lm \in i \R$ and $u \in \C^{2n +m} \setminus \{ 0\}$. Partition $u= [u_1^T ~ u_2^T ~ u_3^T]^T$ such that $u_1,u_2 \in \C^{n}$ and $u_3 \in \C^m$, and let $x,y,z$ and $w$ be defined as in~\eqref{dsmequiEB1} and~\eqref{dsmequiEB2}. Then the following statements are equivalent. 
 		\begin{enumerate}
 			\item There exists $\D_E\in \C^{n,n}$ and $\D_B \in \C^{n,m}$ such that $\D_E \in {\rm Herm}(n)$ satisfying~\eqref{dsmequiEB1} and~\eqref{dsmequiEB2}.
 			\item There exists $\D =[\D_1 ~ \D_2]$, $\D_1 \in \C^{n,n}$, $\D_2 \in \C^{n,m}$  such that $\D_1^*=-\D_1$, $\D x =y$, and  $\D^* z=w$.
 			\item $u_3 = 0$ and $R u_1=0$.
 		\end{enumerate}
 		Moreover, we have 
 		\begin{eqnarray}
 			&\inf \Big \{ {\left \|\big[\D_E ~ \D_B\big] \right \|}_F^2  : ~\D_E \in {\rm Herm}(n), ~\D_B \in \C^{n,m}~\text{satisfying}~\eqref{dsmequiEB1} ~\text{and}~\eqref{dsmequiEB2} \Big \} \nonumber \\
 			&=\inf \bigg \{ \frac{1}{|\lm|^2}{\|\D_1\|}_F^2 + {\|\D_2\|}_F^2 : ~ \D =\big[\D_1 ~ \D_2\big],~\D_1 \in \C^{n,n}, \D_2 \in \C^{n,m},~\D_1^*=-\D_1,~\D x =y, \D ^* z =w \bigg \} \nonumber.
 		\end{eqnarray}
 	\end{lemma}

 The following result provides bounds for the backward error $\eta^{\mathcal S}(E,B,\lm,u)$. 
 	\begin{theorem}\label{bacerrEB}
 		Let $L(z)$ be a pencil as in~\eqref{eq:defpenciL}, let $\lm \in i\R \setminus \{0\}$ and $u \in \C^{2n+m} \setminus \{0\}$. Partition $u=\big[u_1^T~ u_2^T ~ u_3^T \big]^T$ such that $u_1,u_2 \in \C^n$, and $u_3 \in \C^m$. Set $\tilde y=  (J-R + \lm E) u_2$ and $w_1= -(J+R +\lm E)u_1$, $X=\big[u_2 ~ u_1\big],Y=\big[\tilde{y} ~ -w_1 \big]$. Then $\eta^{\mathcal S}(E,B,\lm,u)$ is finite if and only if $u_3=0$ and $R u_1 = 0$. If the later conditions holds and if $YX^\pls X=Y,~Y^*X=-X^*Y$ and if $u_2 =\alpha u_1$ for some nonzero $\alpha \in \C$, then 
 		\begin{equation}\label{eq:thmEB_1}  
 			\eta^{\mathcal S}(E,B,\lm,u) = \sqrt{\frac{1}{|\lm|^2}{\|H_1\|}_F^2+{\|H_2\|}_F^2},
 		\end{equation}	
 		where 
 		\begin{equation}\label{eq:thmEB_2}
 			H_1=YX^\pls - (YX^\pls)^* - X X^\pls Y X^\pls \quad H_2= u_1 u_1^\pls B.
 		\end{equation}
 	\end{theorem}
 	\proof In view of Lemma~\ref{lem: EB}, the proof is similar to the proof of Theorem~\ref{bacerrJB}.
 	\eproof

 	\subsection{Perturbing only $R$ and $E$}

Here suppose that only $R$ and $E$ blocks of $L(z)$ are subject to perturbation. Then in view of~\eqref{def:errblc},~\eqref{def:errstr} and~\eqref{def:errpsd}, the corresponding backward errors are denoted by $\eta^{\mathcal B}(R,E,\lm,u):= \eta^{\mathcal B}(0,R,E,0,\lm,u)$, $\eta^{\mathcal S}(R,E,\lm,u):= \eta^{\mathcal S}(0,R,E,0,\lm,u)$, and $\eta^{\mathcal S_d}(R,E,\lm,u):= \eta^{\mathcal S_d}(0,R,E,0,\lm,u)$. We note that the backward errors $\eta^{\mathcal B}(R,E,\lm,u)$and $\eta^{\mathcal S}(R,E,\lm,u)$ were considered in~\cite[Theorem 4.10]{MehMS18}.  Thus, in this section, we consider only $\eta^{\mathcal S_d}(R,E,\lm,u)$. 
 	From~\eqref{dsmequi1} and~\eqref{dsmequi2}, when $\D_J=0$ and $\D_B=0$ we get 
 	\begin{eqnarray}
 		( - \D_R + \lm \D_E) \underbrace{u_2}_{:=x} &=& \underbrace{(J-R+ \lm E)u_2 + B u_3}_{:=y} \label{dsmequiRE1}\\
 		( - \D_R + \lm \D_E)^*\underbrace{u_1}_{:=z} &=& \underbrace{-(J+R+\lm E)u_1}_{:=w} \label{dsmequiRE2} \\
 		B^*u_1 + S u_3&=& 0. \label{dsmequiRE3}
 	\end{eqnarray}
 	This leads to the following lemma which will be useful in estimating 
 	$\eta^{\mathcal S_d}(R,E,\lm,u)$. 
\begin{lemma}\label{lem: psdRE}
 		Let $L(z) $ be a pencil as in~\eqref{eq:defpenciL}, and let $\lm \in i \R$ and $u \in \C^{2n +m} \setminus \{ 0\}$. Partition $u= [u_1^T ~ u_2^T ~ u_3^T]^T$ such that $u_1,u_2 \in \C^{n}$ and $u_3 \in \C^m$, and let $x,y,z$ and $w$ be defined as in~\eqref{dsmequiRE1} and~\eqref{dsmequiRE2}. If $z=\alpha x$ for some nonzero $\alpha \in \C$,  then the following statements are equivalent.
 		\begin{enumerate}
 			\item There exists $\D_R,\Delta_E \in \C^{n,n}$ such that $\D_E \in {\rm Herm}(n)$ and $ \D_R \succeq 0$ satisfying \eqref{dsmequiRE1} and \eqref{dsmequiRE2}.
 			\item There exists $\D \in \C^{n,n}$ such that $\D + \D^* \preceq 0$, $\D x=y$, $\D^*z=w$.
\item $u_3^* B^* u_1 = 0$.
 		\end{enumerate}
 		Moreover, we have
 		\begin{eqnarray*}
 			& \inf \left \{ {\left \|[ \D_R ~  \D_E]\right \|}_F^2: \D_E,\,\D_R \in {\rm Herm}(n),\, \D_R \succeq 0~\text{satisfying}~ \eqref{dsmequiRE1} ~ \text{and}~ \eqref{dsmequiRE2} \right \} \nonumber \\
 			& =  \inf \left  \{ {\left \|\frac{ \D + \D^*}{2}\right \|}_F^2 + \frac{1}{|\lm|^2} {\left \|\frac{ \D- \D^*}{2}\right \|}_F^2 ~:~\Delta \in \C^{n,n},\,  \D+ \D^* \preceq 0,\,\D x= y,\,\D^*z=w \right \}. \nonumber
 		\end{eqnarray*}
 	\end{lemma}
 	\proof The proof is analogous to \cite[Lemma 4.9]{MehMS18}, due to Type-1 doubly structured dissipative mapping from Theorem~\ref{map:type1veccase}.
 	\eproof
 	
\begin{theorem}\label{bacerrpsdRE}
Let $L(z)$ be a pencil as in~\eqref{eq:defpenciL}, let $\lm \in i\R \setminus \{0\}$ and $u \in \C^{2n+m} \setminus \{0\}$. Partition $u=\big[u_1^T~ u_2^T ~ u_3^T \big]^T$ such that $u_1,u_2 \in \C^n$, and $u_3 \in \C^m$. Set $\tilde y=  (J-R + \lm E) u_2$ and $w_1= -(J+R +\lm E)u_1$. If $u_2 =\alpha u_1$ for some nonzero $\alpha \in \C$, then $\eta^{\mathcal S_d}(R,E,\lm,u)$ is finite if and only if $u_3=0$ and $B^* u_1=0$. If the later condition holds and if $u_2$ satisfies that $Ru_2 \neq 0$, then 	
 		\begin{equation*}
 			\frac{\| H\|_F}{|\lm|} \leq \eta^{\mathcal S_d}(R,E,\lm,u) \leq \sqrt{ {\left \|\frac{ H + H^*}{2}\right \|}_F^2 + \frac{1}{|\lm|^2} {\left \|\frac{ H- H^*}{2}\right \|}_F^2},\quad \text{if} ~|\lambda| \geq 1,
 		\end{equation*}
 		and
 		\begin{equation*}
 			\| H\|_F \leq \eta^{\mathcal S_d}(R,E,\lm,u) \leq \sqrt{ {\left \|\frac{ H + H^*}{2}\right \|}_F^2 + \frac{1}{|\lm|^2} {\left \|\frac{ H- H^*}{2}\right \|}_F^2},\quad \text{if} ~|\lambda| \leq 1,
 		\end{equation*}
 		where
 		\begin{equation}\label{eq:thmpsdRE_2}
 			H=\tilde yu_2^\dagger + (w_1 u_1^\dagger)^*\mathcal P_{u_2}+\mathcal P_{u_2}J\mathcal P_{u_2} \quad \text{and}\quad J=\frac{1}{4\real{(u_2^*{\tilde y}})} \big(\tilde y+\frac{\alpha}{|\alpha|^2}w_1\big)\big(\tilde y+\frac{\alpha}{|\alpha|^2}w_1\big)^*.
 			\end{equation}
 	\end{theorem}
 	\proof The proof is anologous to the proof of~\cite[Theorem 4.10]{MehMS18} using Lemma~\ref{lem: psdRE} and Theorem~\ref{map:type1veccase}.
 	\eproof
 		\subsection{Perturbing only $J$, $R$ and $E$}
 	Suppose that the blocks $J$, $R$ and $E$ of $L(z)$ are subject to perturbation. Then in view of~\eqref{def:errblc},~\eqref{def:errstr} and~\eqref{def:errpsd}, the corresponding backward errors are denoted by  $\eta^{\mathcal B}(J,R,E,\lm,u):= \eta^{\mathcal B}(J,R,E,0,\lm,u)$, $\eta^{\mathcal S}(J,R,E,\lm,u):= \eta^{\mathcal S}(J,R,E,0,\lm,u)$, and $\eta^{\mathcal S_d}(J,R,E,\lm,u):= \eta^{\mathcal S_d}(J,R,E,0,\lm,u)$. The block- and symmetry-structured backward errors  $\eta^{\mathcal B}(J,R,E,\lm,u)$ and $\eta^{\mathcal S}(J,R,E,\lm,u)$ were obtained in ~\cite[Theorem~5.11]{MehMS18}. In this section, we focus on estimating the backward error $\eta^{\mathcal S_d}(J,R,E,\lm,u)$.
 	
From~\eqref{dsmequi1} and~\eqref{dsmequi2}, when $\D_B=0$ we have
\begin{eqnarray}
(\D_J - \D_R + \lm \D_E) \underbrace{u_2}_{:=x} &=& \underbrace{(J-R+ \lm E)u_2 + B u_3}_{:=y} \label{dsmequiJRE1}\\
(\D_J - \D_R + \lm \D_E)^*\underbrace{u_1}_{:=z} &=& \underbrace{-(J+R+\lm E)u_1}_{:=w} \label{dsmequiJRE2} \\
B^*u_1 + S u_3&=& 0. \label{dsmequiJRE3}
\end{eqnarray}
In view of~\eqref{dsmequiJRE1} and~\eqref{dsmequiJRE3}, we have the following lemma which is analogous to ~\ref{lem:JpsdREB}  for estimting 
$\eta^{\mathcal S_d}(J,R,E,\lm,u)$.

\begin{lemma}\label{lem: JpsdRE}
Let $L(z) $ be a pencil as in~\eqref{eq:defpenciL}, and let $\lm \in i \R$ and $u \in \C^{2n +m} \setminus \{ 0\}$. Partition $u= [u_1^T ~ u_2^T ~ u_3^T]^T$ such that $u_1,u_2 \in \C^{n}$ and $u_3 \in \C^m$, and let $x,y,z$ and $w$ be defined as in~\eqref{dsmequiJRE1} and~\eqref{dsmequiJRE2}. If $z=\alpha x$, then the following statements are equivalent.
\begin{enumerate}
\item There exists $\D_J,\D_R,\Delta_E \in \C^{n,n}$ such that $\D_J \in {\rm SHerm}(n) $, $\D_E \in {\rm Herm}(n)$, and $ \D_R \succeq 0$ satisfying \eqref{dsmequiJRE1} and \eqref{dsmequiJRE2}.
\item There exists $\D \in \C^{n,n}$ such that $\D + \D^* \preceq 0, \D x=y,\D^*z=w$.
\item $u_3^* B^* u_1 = 0$.
\end{enumerate}
Moreover, we have
\begin{eqnarray}
&& \inf \Big \{ {\|[\D_J ~ \D_R ~  \D_E]\|}_F^2: \D_J \in {\rm SHerm}(n), \D_E,~\D_R \in {\rm Herm}(n), \D_R \succeq 0, \nonumber \\
&& \hspace{5cm}~\text{satisfying}~ \eqref{dsmequiJRE1} ~ \text{and}~ \eqref{dsmequiJRE2} \Big \} \nonumber \\
&& =  \inf \left \{ {\left \|\frac{ \D + \D^*}{2}\right \|}_F^2 + \frac{1}{1+|\lm|^2} {\left \|\frac{ \D - \D^*}{2}\right \|}_F^2~:~ \Delta \in \C^{n,n},\, \D+ \D^* \preceq 0,\,\D x= y,\,\D^*z=w \right \}. \nonumber
\end{eqnarray}
\end{lemma}
\proof The proof is analogous to \cite[Lemma 5.10]{MehMS18}, due to Type-1 doubly structured dissipative mapping from Theorem~\ref{map:type1veccase}.
\eproof

\begin{theorem}\label{bacerrJpsdRE}
Let $L(z)$ be a pencil as in~\eqref{eq:defpenciL}, let $\lm \in i\R \setminus \{0\}$ and $u \in \C^{2n+m} \setminus \{0\}$. Partition $u=\big[u_1^T~ u_2^T ~ u_3^T \big]^T$ such that $u_1,u_2 \in \C^n$, and $u_3 \in \C^m$. Set $\tilde y=  (J-R + \lm E) u_2$ and $w_1= -(J+R +\lm E)u_1$. If $u_2 =\alpha u_1$ for some nonzero $\alpha \in \C$, then $\eta^{\mathcal S_d}(J,R,\lm,u)$ is finite if and only if $u_3=0$ and $B^* u_1=0$. If the later condition holds and if $u_2$ satisfies that
$Ru_2 \neq 0$, then 	
\begin{equation*}
\frac{\| H\|_F}{\sqrt{1+|\lm|^2}} \leq \eta^{\mathcal S_d}(J,R,E,\lm,u) \leq \sqrt{{\left \|\frac{ H + H^*}{2}\right \|}_F^2 + \frac{1}{1+|\lm|^2} {\left \|\frac{ H -H^*}{2}\right \|}_F^2},
\end{equation*}
where
\begin{equation*}
H=\tilde yu_2^\dagger + (w_1 u_1^\dagger)^*\mathcal P_{u_2}+\mathcal P_{u_2}J\mathcal P_{u_2} \quad \text{and}\quad J=\frac{1}{4\real{(u_2^*{\tilde y}})} \big(\tilde y+\frac{\alpha}{|\alpha|^2}w_1\big)\big(\tilde y+\frac{\alpha}{|\alpha|^2}w_1\big)^*.
\end{equation*}
\end{theorem}
\proof The proof is anologous to the proof of \cite[Theorem~5.11]{MehMS18} due to Lemma~\ref{lem: JpsdRE} and Theorem~\ref{map:type1veccase}.
\eproof
\subsection{Perturbing only J,E and B}

In this section, suppose that the blocks $J$, $E$ and $B$ of $L(z)$ are subject to perturbation. Then in view of~\eqref{def:errblc},~\eqref{def:errstr} and~\eqref{def:errpsd}, the corresponding backward errors are denoted by $\eta^{\mathcal B}(J,E,B,\lm,u):= \eta^{\mathcal B}(J,0,E,B,\lm,u)$, $\eta^{\mathcal S}(J,E,B,\lm,u):= \eta^{\mathcal S}(J,0,E,B,\lm,u)$, and $\eta^{\mathcal S_d}(J,E,B,\lm,u):= \eta^{\mathcal S_d}(J,E,B,\lm,u)$. 
The backward error $\eta^{\mathcal B}(J,E,B,\lm,u)$ was given in~\cite[Remark 5.8]{MehMS18}, and  we have $\eta^{\mathcal S_d}(J,E,B,\lm,u)=\eta^{\mathcal S}(J,E,B,\lm,u)$ because there is no semidefinite structure on  $J$ or $E$ or $B$. Thus we focus on computing $\eta^{\mathcal S}(J,E,B,\lm,u)$. 

From~\eqref{dsmequi1} and~\eqref{dsmequi2}, when $\D_R=0$ we have 
\begin{eqnarray}
	\mat{cc} \underbrace{\D_J + \lm \D_E }_{=:\Delta_1} & \underbrace{\D_B}_{=:\Delta_2} \rix \underbrace{\mat{c} u_2 \\ u_3 \rix}_{=:x} &=& \underbrace{(J-R + \lm E) u_2 + B u_3}_{=:y} \label{dsmequiJEB1} \\
	\mat{cc}  \D_J + \lm \D_E& \D_B \rix ^* \underbrace{u_1}_{=:z} &=& \underbrace{\mat{c}-(J+R + \lm E)u_1=:w_1  \\ B^*u_1 + S u_3=:w_2 \rix}_{=:w}, \label{dsmequiJEB2}
\end{eqnarray}
Thus using doubly structured skew-Hermitian mapping from Theorem~\ref{theorem_SHD} in~\eqref{dsmequiJEB1} and~\eqref{dsmequiJEB2} gives the following lemma.
\begin{lemma}\label{lem: JEB}
	Let $L(z) $ be a pencil as in~\eqref{eq:defpenciL}, and let $\lm \in i \R$ and $u \in \C^{2n +m} \setminus \{ 0\}$. Partition $u= [u_1^T ~ u_2^T ~ u_3^T]^T$ such that $u_1,u_2 \in \C^{n}$ and $u_3 \in \C^m$, and let $x,y,z$ and $w$ be defined as in~\eqref{dsmequiJEB1} and~\eqref{dsmequiJEB2}. Then the following statements are equivalent. 
	\begin{enumerate}
		\item There exists $\D_J,~\D_E\in \C^{n,n}$ and $\D_B \in \C^{n,m}$ such that $\D_J \in {\rm SHerm}(n),~\D_E \in {\rm Herm}(n)$ satisfying~\eqref{dsmequiJEB1} and~\eqref{dsmequiJEB2}.
		\item There exists $\D =[\D_1 ~ \D_2]$, $\D_1 \in \C^{n,n}$, $\D_2 \in \C^{n,m}$  such that $\D_1^*=-\D_1$, $\D x =y$, and  $\D^* z=w$.
		\item $u_3 = 0$ and $R u_1=0$.
	\end{enumerate}
	Moreover, we have 
	\begin{eqnarray}
		&\inf \Big \{ {\left \|\big[\D_J ~ \D_E ~ \D_B\big] \right \|}_F^2  : ~\D_J \in {\rm SHerm}(n),~ \D_E \in {\rm Herm}(n), ~\D_B \in \C^{n,m}~\text{satisfying}~\eqref{dsmequiJEB1} ~\text{and}~\eqref{dsmequiJEB2} \Big \} \nonumber \\
		&=\inf \bigg \{ \frac{1}{1+|\lm|^2}{\|\D_1\|}_F^2 + {\|\D_2\|}_F^2 : ~ \D =\big[\D_1 ~ \D_2\big],~\D_1 \in \C^{n,n}, \D_2 \in \C^{n,m},~\D_1^*=-\D_1,~\D x =y, \D ^* z =w \bigg \} \nonumber.
	\end{eqnarray}
\end{lemma}
\proof The proof is similar to~\cite[Lemma 5.6]{MehMS18} due to doubly structured skew Hermitian mapping from Theorem~\ref{theorem_SHD}.
\eproof

\begin{theorem}\label{bacerrJEB}
	Let $L(z)$ be a pencil as in~\eqref{eq:defpenciL}, let $\lm \in i\R \setminus \{0\}$ and $u \in \C^{2n+m} \setminus \{0\}$. Partition $u=\big[u_1^T~ u_2^T ~ u_3^T \big]^T$ such that $u_1,u_2 \in \C^n$, and $u_3 \in \C^m$. Set $\tilde y=  (J-R + \lm E) u_2$ and $w_1= -(J+R +\lm E)u_1$, $X=\big[u_2 ~ u_1\big],Y=\big[\tilde{y} ~ -w_1 \big]$. Then $\eta^{\mathcal S}(E,B,\lm,u)$ is finite if and only if $u_3=0$ and $R u_1 = 0$. If the later condition holds and if $YX^\pls X=Y,~Y^*X=-X^*Y$, if  $u_2 =\alpha u_1$ for some nonzero $\alpha \in \C$, then 
	\begin{equation}\label{eq:thmJEB_1}  
		\eta^{\mathcal S}(J,E,B,\lm,u) = \sqrt{\frac{1}{1+|\lm|^2}{\|H_1\|}_F^2+{\|H_2\|}_F^2},
	\end{equation}	
	where 
	\begin{equation}\label{eq:thmJEB_2}
		H_1=YX^\pls - (YX^\pls)^* - X X^\pls Y X^\pls \quad \text{and}\quad H_2= u_1 u_1^\pls B.
	\end{equation}
\end{theorem}
\proof In view of Remark~\ref{rem:appdsw} and Lemma~\ref{lem: JEB}, we have that $\eta^{\mathcal S}(J, E,B,\lm,u)$ is finite if and only if $u_3=0$ and $R u_1 = 0$. Thus by using $u_3=0$ in~\eqref{dsmequiJEB1} and~\eqref{dsmequiJEB2}, and using Lemma~\ref{lem: JEB} in~\eqref{def:errstr}, we have that 
\begin{eqnarray}\label{eq:thmJEB_3}
	&\eta^{\mathcal S}(J,E,B,\lm,u)^2 = 
	\inf \bigg \{ \frac{1}{1+|\lm|^2}{\|\D_1\|}_F^2 + {\|\D_2\|}_F^2 ~: ~\D_1 \in \C^{n,n}, \D_2 \in \C^{n,m},~\D_1^*=-\D_1, \nonumber \\
	&\D_1 u_2 =\tilde{y}, \D_1 ^* u_1 =w_1, \D_2^* u_1 = B^* u_1 \bigg \} \nonumber \\ 
	& = \inf \bigg \{ \frac{1}{1+|\lm|^2}{\|\D_1\|}_F^2 + {\|\D_2\|}_F^2  ~:~\D_1 \in \C^{n,n},\D_2 \in \C^{n,m},~\D_1^*=-\D_1, \nonumber \\
	& \D_1 \big[u_2 ~ u_1\big]= \big[\tilde{y} ~ -w_1 \big], \D_2^* u_1 = B^* u_1 \bigg \} \nonumber \\
	& = \inf \bigg \{ \frac{1}{1+|\lm|^2}{\|\D_1\|}_F^2 + {\|\D_2\|}_F^2  ~:~\D_1 \in \C^{n,n},~\D_2 \in \C^{n,m},~\D_1^*=-\D_1, \nonumber \\
	& \D_1 X= Y, \D_2^* u_1 = B^* u_1 \bigg \}.
\end{eqnarray}
If $Y X^\pls X=Y,~Y^*X=-X^*Y$ and $u_2 =\alpha u_1$ for some nonzero $\alpha \in \C$, then from~\cite[Theorem 2.2.3]{adhikari2008backward}, there always exists a skew-Hermitian mapping $\Delta_1$ such that  $\D_1^*=-\D_1$ and $\D_1 X=Y$. The minimal Frobenius norm of such a $\Delta_1$ is attained by the unique matrix $H_1$ defined in~\eqref{eq:thmJEB_2}. Similarly, from Theorem~\ref{map result: unstruct}, for any $u_1$ there always exists $\Delta_2 \in \C^{n,m}$ such that 
$\Delta_2^*u_1=B^*u_1$ and the minimal Frobenius norm of such a $\Delta_2$ is attained by $H_2:=u_1u_1^\dagger B$.
Thus using the minimal Frobenius norm mappings $H_1$ and $H_2$, we obtain~\eqref{eq:thmJEB_1}. This completes the proof.
\eproof
\subsection{Perturbing only R,E and B}
Here, suppose that the blocks $R$, $E$ and $B$ of $L(z)$ are subject to perturbation. Then in view of~\eqref{def:errblc},~\eqref{def:errstr} and~\eqref{def:errpsd}, the corresponding backward errors are denoted by 
 $\eta^{\mathcal B}(R,E,B,\lm,u):= \eta^{\mathcal B}(0,R,E,B,\lm,u)$, $\eta^{\mathcal S}(R,E,B,\lm,u):= \eta^{\mathcal S}(0,R,E,B,\lm,u)$, and $\eta^{\mathcal S_d}(R,E,B,\lm,u):= \eta^{\mathcal S_d}(0,R,E,B,\lm,u)$.
 The block and symmetry structured backward errors  $\eta^{\mathcal B}(R,E,B,\lm,u)$ and $\eta^{\mathcal S}(R,E,B,\lm,u)$ were obtained in~\cite[Theorem 5.7]{MehMS18}.  In this section we compute bounds for semidefinite structured backward error $\eta^{\mathcal S_d}(R,E,B,\lm,u)$.

For this, from~\eqref{dsmequi1} and~\eqref{dsmequi2}, when $\D_J=0$ we have 
\begin{eqnarray}
	\mat{cc} \underbrace{ - \D_R + \lm \D_E }_{=:\Delta_1} & \underbrace{\D_B}_{=:\Delta_2} \rix \underbrace{\mat{c} u_2 \\ u_3 \rix}_{=:x} &=& \underbrace{(J-R + \lm E) u_2 + B u_3}_{=:y} \label{dsmequiREB1} \\
	\mat{cc}   - \D_R + \lm \D_E& \D_B \rix ^* \underbrace{u_1}_{=:z} &=& \underbrace{\mat{c}-(J+R + \lm E)u_1=:w_1  \\ B^*u_1 + S u_3=:w_2 \rix}_{=:w}, \label{dsmequiREB2}.
\end{eqnarray} 
 The following lemma is analogous to~\cite[Lemma 6.2]{MehMS18} that will be useful in computing $\eta^{\mathcal S_d}(R,E,B,\lambda,u)$. 
\begin{lemma}\label{lem:psdREB}
	Let $L(z) $ be a pencil as in~\eqref{eq:defpenciL}, and let $\lm \in i \R$ and $u \in \C^{2n +m} \setminus \{ 0\}$. Partition $u= [u_1^T ~ u_2^T ~ u_3^T]^T$ such that $u_1,u_2 \in \C^{n}$ and $u_3 \in \C^m$, and let $x,y,z$ and $w$ be defined as in~\eqref{dsmequiREB1} and~\eqref{dsmequiREB2}. Then the following statements are equivalent.
	\begin{enumerate}
		\item There exists $\D_R, \D_E \in \C^{n,n}$ and $\D_B \in \C^{n,m}$ such that $\D_R \succeq 0$, and $\D_E \in {\rm Herm}(n)$ satisfying~\eqref{dsmequiREB1} and~\eqref{dsmequiREB2}.
		\item There exists $\D =[\D_1 ~ \D_2]$, $\D_1 \in \C^{n,n}$, $\D_2 \in \C^{n,m}$  such that $\D_1+ \D_1^* \preceq 0$, $\D x =y$, and  $\D^* z=w$.
		\item $u_3 = 0$.
	\end{enumerate}
	Moreover, we have 
	\begin{eqnarray*}
		&&\inf \left \{ {\left \|\big[\D_R ~ \D_E ~ \D_B\big] \right \|}_F^2  : \D_R,  \D_E\in {\rm Herm}(n), \D_R\succeq 0,\,\D_B \in \C^{n,m}~\text{satisfying}~\eqref{dsmequiREB1} ~\text{and}~\eqref{dsmequiREB2} \right \} \nonumber \\
		&&=\inf \bigg \{  {\left \| \frac{\D_1 + \D_1^*}{2} \right \|}_F^2 + \frac{1}{|\lm|^2}{\left \| \frac{\D_1 - \D_1^*}{2} \right \|}_F^2 + {\|\D_2\|}_F^2 : ~ \D =\big[\D_1 ~ \D_2\big],~\D_1 \in \C^{n,n},\,\D_2 \in \C^{n,m}, \\
		&& \hspace{6cm}\D_1+\D_1^* \preceq 0,~\D x =y, \D ^* z =w \bigg \} \nonumber.
	\end{eqnarray*}
\end{lemma}
\proof The proof is similar to the proof of~\cite[Lemma 5.6]{MehMS18} due to Type-2 doubly structured dissipative mapping from Theorem~\ref{thm:Type-2DSDM}.
\eproof

\begin{theorem}\label{bacerrpsdREB}
	Let $L(z)$ be a pencil as in~\eqref{eq:defpenciL}, let $\lm \in i\R \setminus \{0\}$ and $u \in \C^{2n+m} \setminus \{0\}$. Partition $u=[u_1^T~ u_2^T ~ u_3^T]^T$ such that $u_1,u_2 \in \C^n$, and $u_3 \in \C^m$. Set $\tilde y=  (J-R + \lm E) u_2$ and $w_1= -(J+R +\lm E)u_1$. Then $\eta^{\mathcal S_d}(J,R,B,\lm,u)$ is finite if and only if $u_3=0$. If the later condition holds and if $u_2$ satisfies that 
	$Ru_2 \neq 0$ and $u_2 =\alpha u_1$ for some nonzero $\alpha \in \C$, then 
	\begin{equation*}\label{eq:thmpsdREB_1}
		\sqrt{{\|H_1\|}_F^2 + {\|H_2\|}_F^2} \leq  
		\eta^{\mathcal S_d}(R,E,B,\lm,u) \leq \sqrt{{\left \| \frac{H_1 + H_1^*}{2} \right \|}_F^2 + \frac{1}{|\lm|^2}{\left \| \frac{H_1 - H_1^*}{2} \right \|}_F^2 + {\|H_2\|}_F^2},
	\end{equation*}	
	when $|\lambda| \leq 1$,
	and
\begin{equation*}\label{eq:thmpsdREB_1}
		\sqrt{\frac{1}{|\lambda|^2}{\|H_1\|}_F^2 + {\|H_2\|}_F^2} \leq  
		\eta^{\mathcal S_d}(R,E,B,\lm,u) \leq \sqrt{{\left \| \frac{H_1 + H_1^*}{2} \right \|}_F^2 + \frac{1}{|\lm|^2}{\left \| \frac{H_1 - H_1^*}{2} \right \|}_F^2 + {\|H_2\|}_F^2},
	\end{equation*}	
when $|\lambda| > 1$,
	where 
	\begin{equation*}\label{eq:thmpsdREB_2}
		H_1=\tilde yu_2^\dagger + (w_1 u_1^\dagger)^*\mathcal P_{u_2}+\mathcal P_{u_2}J\mathcal P_{u_2} \quad \text{with} \quad J=\frac{1}{4\real{(u_2^*{\tilde y}})} \big(\tilde y+\frac{\alpha}{|\alpha|^2}w_1\big)\big(\tilde y+\frac{\alpha}{|\alpha|^2}w_1\big)^*,
	\end{equation*}
	where 
and $H_2=u_1u_1^\dagger B$.
\end{theorem}
\proof In view of Remark~\ref{rem:appdsw} and Lemma~\ref{lem:psdREB}, the proof is similar to the proof of~\cite[Theorem 5.7]{MehMS18}.
\eproof
\subsection{Perturbing only J,R and B}

Finally, suppose that the blocks $J$, $R$ and $B$ of $L(z)$ are subject to perturbation. Then in view of~\eqref{def:errblc},~\eqref{def:errstr} and~\eqref{def:errpsd}, the corresponding backward errors are denoted by 
$\eta^{\mathcal B}(J,R,B,\lm,u):= \eta^{\mathcal B}(J,R,0,B,\lm,u)$, $\eta^{\mathcal S}(J,R,B,\lm,u):= \eta^{\mathcal S}(J,R,0,B,\lm,u)$, and $\eta^{\mathcal S_d}(J,R,B,\lm,u):= \eta^{\mathcal S_d}(J,R,0,B,\lm,u)$. 
Again note that the block and symmetry structured backward errors $\eta^{\mathcal B}(J,R,B,\lm,u)$ and $\eta^{\mathcal S}(J,R,B,\lm,u)$ were respectively obtained in~\cite[Theorem 5.3]{MehMS18} and ~\cite[Theorem 5.4]{MehMS18}. Thus, in this section, we focus only on computing the semidefinite structured backward error $\eta^{\mathcal S_d}(J,R,B,\lm,u)$. 

From~\eqref{dsmequi1} and~\eqref{dsmequi2}, when $\D_E=0$ we have
\begin{eqnarray}
	\mat{cc} \underbrace{\D_J - \D_R  }_{=:\Delta_1} & \underbrace{\D_B}_{=:\Delta_2} \rix \underbrace{\mat{c} u_2 \\ u_3 \rix}_{=:x} &=& \underbrace{(J-R + \lm E) u_2 + B u_3}_{=:y} \label{dsmequiJRB1} \\
	\mat{cc}  \D_J - \D_R & \D_B \rix ^* \underbrace{u_1}_{=:z} &=& \underbrace{\mat{c}-(J+R + \lm E)u_1=:w_1  \\ B^*u_1 + S u_3=:w_2 \rix}_{=:w}. \label{dsmequiJRB2}
\end{eqnarray} 
Then a direct use of Type-2 doubly structured  dissipative mapping from Theorem~\ref{thm:Type-2DSDM} in~\eqref{dsmequiJRB1} and~\eqref{dsmequiJRB2}, gives the following lemma.
\begin{lemma}\label{lem:JpsdRB}
	Let $L(z) $ be a pencil as in~\eqref{eq:defpenciL}, and let $\lm \in i \R$ and $u \in \C^{2n +m} \setminus \{ 0\}$. Partition $u= [u_1^T ~ u_2^T ~ u_3^T]^T$ such that $u_1,u_2 \in \C^{n}$ and $u_3 \in \C^m$, and let $x,y,z$ and $w$ be defined as in~\eqref{dsmequiJRB1} and~\eqref{dsmequiJRB2}. Then the following statements are equivalent.
	\begin{enumerate}
		\item There exists $\D_J,\D_R \in \C^{n,n}$ and $\D_B \in \C^{n,m}$ such that $\D_J \in {\rm SHerm}(n)$, $\D_R \succeq 0$ satisfying~\eqref{dsmequiJRB1} and~\eqref{dsmequiJRB2}.
		\item There exists $\D =[\D_1 ~ \D_2]$, $\D_1 \in \C^{n,n}$, $\D_2 \in \C^{n,m}$  such that $\D_1+ \D_1^* \preceq 0$, $\D x =y$, and  $\D^* z=w$.
		\item $u_3 = 0$.
	\end{enumerate}
	Moreover, we have 
	\begin{eqnarray*}
		&\inf \Big \{ {\left \|\big[\D_J ~ \D_R ~ \D_B\big] \right \|}_F^2  : ~\D_J \in {\rm SHerm}(n),\D_R \in {\rm Herm}(n), \D_R \succeq 0,\,\D_B \in \C^{n,m},\\
		&\text{satisfying}~\eqref{dsmequiJRB1}  ~ \text{and}~\eqref{dsmequiJRB2} \Big \} \nonumber \\
		&=\inf \Big \{   {\|\D_1\|}_F^2 + {\|\D_2\|}_F^2 : ~ \D =\big[\D_1 ~ \D_2\big],~\D_1 \in \C^{n,n}, \D_2 \in \C^{n,m},~\D_1+\D_1^* \preceq 0,\\
		& ~\D x =y, \D ^* z =w \Big \} \nonumber.
	\end{eqnarray*}
\end{lemma}

\begin{theorem}\label{bacerrJpsdRB}
	Let $L(z)$ be a pencil as in~\eqref{eq:defpenciL}, let $\lm \in i\R \setminus \{0\}$ and $u \in \C^{2n+m} \setminus \{0\}$. Partition $u=[u_1^T~ u_2^T ~ u_3^T]^T$ such that $u_1,u_2 \in \C^n$, and $u_3 \in \C^m$. Set $\tilde y=  (J-R + \lm E) u_2$ and $w_1= -(J+R +\lm E)u_1$. Then $\eta^{\mathcal S_d}(J,R,B,\lm,u)$ is finite if and only if $u_3=0$. If the later condition holds and if $u_2$ satisfies that 
	$Ru_2 \neq 0$ and $u_2 =\alpha u_1$ for some nonzero $\alpha \in \C$, then 
	\begin{equation}\label{eq:thmJpsdRB_1} 
		\eta^{\mathcal S_d}(J,R,B,\lm,u) = \sqrt{{\|H_1\|}_F^2+{\|H_2\|}_F^2},
	\end{equation}	
	where 
	\begin{equation}\label{eq:thmJpsdRB_2}
		H_1=\tilde yu_2^\dagger + (w_1 u_1^\dagger)^*\mathcal P_{u_2}+\mathcal P_{u_2}J\mathcal P_{u_2} \quad \text{with}\quad  J=\frac{1}{4\real{(u_2^*{\tilde y}})} \big(\tilde y+\frac{\alpha}{|\alpha|^2}w_1\big)\big(\tilde y+\frac{\alpha}{|\alpha|^2}w_1\big)^*,
	\end{equation}
	and $H_2=u_1u_1^\dagger B$.
\end{theorem}
\proof In view of Remark~\ref{rem:appdsw} and Lemma~\ref{lem:JpsdRB},
the proof is analogous to the proof of Theorem~\ref{bacerrJpsdREB} using Type-1 doubly structured dissipative mapping from Theorem~\ref{map:type1veccase} .
\eproof

\end{document}